\documentclass[leqno,12pt]{amsart}
\usepackage{mathtools}
\usepackage{hyperref}
\hypersetup{pdfstartview={XYZ null null 1.00}}
\usepackage{tabulary}
\usepackage{booktabs}
\usepackage{fancyhdr}
\usepackage{amsmath}
\usepackage{amsfonts}
\usepackage{amssymb}
\usepackage[usenames,dvipsnames]{xcolor}
\usepackage[normalem]{ulem}
\usepackage{ourmacros}
\usepackage{xargs}
\usepackage[bordercolor=gray!20,backgroundcolor=blue!10,linecolor=none,textsize=tiny,textwidth=0.8in]{todonotes}
\usepackage{mathtools}
\usepackage{tikz}
\usetikzlibrary{arrows}
\usetikzlibrary{fadings}
\usetikzlibrary{calc,patterns,angles,quotes}

\newcommand{\tikzAngleOfLine}{\tikz@AngleOfLine}
  \def\tikz@AngleOfLine(#1)(#2)#3{
  \pgfmathanglebetweenpoints{
    \pgfpointanchor{#1}{center}}{
    \pgfpointanchor{#2}{center}}
  \pgfmathsetmacro{#3}{\pgfmathresult}
  }

\mathtoolsset{showonlyrefs}

\newcommand{\jjmax}{{\max\{j,j'\}}}
\newcommand{\jjmin}{{\min\{j,j'\}}}

\newcommand{\strongo}{{\mathbb{O}}}
\newcommand{\strongop}{\mathbb{O}_{\geq}}
\newcommand{\Cinfb}{C^\infty_b}
\newcommand{\CinfbRd}{C^\infty_b(\Rdst)}

\newcommand{\sym}{\eta}
\newcommand{\sign}{\varsigma}

\title[Multi-scale GB parametrix for the Dirichlet problem] {A
  multi-scale Gaussian beam parametrix for the wave equation: the
  Dirichlet boundary value problem.}

\author{Michele Berra}
\address{Dipartimento di Scienze Matematiche, Politecnico di Torino (DISMA), corso Duca degli
Abruzzi 24, 10129 Torino, Italy}
\email{michele.berra@polito.it}
\author{Maarten V. de Hoop}
\address{Department of Computational and Applied Mathematics\\
Rice University, 6100 Main MS-134 Houston, TX 77005, United States of America}
\email{mdehoop@rice.edu}
\author{Jos\'e Luis Romero}
\address{Acoustics Research Institute\\ Austrian Academy of Sciences\\ Wohllebengasse 12-14, Vienna,
1040, Austria}
\email{jlromero@kfs.oeaw.ac.at}

\thanks{
M.B. gratefully acknowledges support from the \textit{Progetto Talenti} funded by \textit{Fondazione
CRT} and from MIUR (Italy) under the PRIN project
\textit{Variet\`a reali e complesse: geometria, topologia e analisi armonica}.
M.V.d.H. gratefully acknowledges support from the Simons
  Foundation under the MATH $+$ X program, the National Science
  Foundation under grant DMS-1559587, and the corporate members of the
  Geo-Mathematical Group at Rice University. J.L.R. gratefully
  acknowledges support from the Austrian Science Fund (FWF): P 29462 - N35, and from a Marie Curie fellowship, within the 7th. European
  Community Framework program, under grant PIIF-GA-2012-327063.
}

\keywords{Gaussian beam, wave equation, wave-atom, parametrix, boundary-value problem}

\subjclass[2010]{35L05, 35L20, 35S05, 42C15}

\begin{document}
\selectlanguage{english}

\begin{abstract}
We present a construction of a multi-scale Gaussian beam parametrix
for the Dirichlet boundary value problem associated with the wave
equation, and study its convergence rate to the true solution in the highly oscillatory regime. The construction elaborates on the
wave-atom parametrix of Bao, Qian, Ying, and Zhang and extends to a
multi-scale setting the technique of Gaussian beam propagation from a
boundary of Katchalov, Kurylev and Lassas.
\end{abstract}

\maketitle

\section{Introduction}

\subsection{The parametrix}
Gaussian beams are high-frequency asymptotic solutions for hyperbolic
partial differential equations, in particular, for the homogeneous
wave equation,
\begin{align*}
u_{tt}(t,x) - \speed(x)^2 \Delta_x u(t,x) = 0.
\end{align*}
Gaussian beams (GB) follow the propagation of singularities, that is,
the bicharacteristics associated with the principal symbol of the wave
operator, which are the flows generated by the Hamiltonians $H(x,p) = \pm \speed(x) \abs{p}$. 
Gaussian beams are
initiated via an Ansatz. They distinguish themselves from standard
geometrical optics solutions in that they capture the asymptotic
behavior in caustics without precautions, see Figure
\ref{fig:caustic}.

GB parametrices for the initial value problem (IVP) for the
wave equation are based on representations of the initial data as
superimposition of certain Gaussian-like wavepackets:
\begin{align}
\label{eq_pack1}
u(0,x) = \sum_\gamma a_\gamma \varphi_\gamma(x),
\qquad
u_t(0,x) = \sum_\gamma b_\gamma \varphi_\gamma(x).
\end{align}
Each wavepacket $\varphi_\gamma$ is then used to generate two GB:
$\varphi_\gamma(x) \approx \Phi_\gamma^\pm(0,x)$, where the choice of
sign $\pm$ corresponds to the two polarized modes of the
wave-equation. Specifically, we construct a frame of such wavepackets that initialize multi-scale Gaussian
beams, and the resulting parametrix has the form
\begin{align}
\label{eq_intro_IVP}
\tilde u(t,x) = \sum_\gamma \alpha^+_\gamma \Phi^+_\gamma(t,x)
+ \sum_\gamma \alpha^-_\gamma \Phi^-_\gamma(t,x),
\end{align}
where the sequences $\alpha^+$, $\alpha^-$ are defined in terms of $a$ and
$b$. The precise form of the packet decomposition in \eqref{eq_pack1}
determines the effectiveness of the parametrix. A detailed study of
the approximation error of such parametrices when the initial data is
a finite sum of Gaussian packets is provided in \cite{MR3008843}. Gaussian-beam parametrices and summation of Gaussian 
beams are naturally connected to Fourier integral operators with complex phase.

Several other parametrices for the wave equation are also based on wavepacket expansions.
Indeed, Smith \cite{MR1724210,MR1644105} introduced the use of a frame of wavepackets
with parabolic scaling (curvelets) in the construction of a parametrix, which,
for smooth wave speeds, can be identified as a Fourier integral operator.
This representation is also underlying the analysis of wave propagators of Cand\`{e}s and Demanet \cite{MR2165380}. 
Further related constructions based on localized wavepackets can be found in the work of Tataru \cite{MR2208883}, 
Koch and Tataru \cite{MR2094851}, Geba and Tataru \cite{MR2353138}, and De Hoop, Uhlmann, Vasy and Wendt 
\cite{MR3061463}.

In \cite{Lex2, Lex3} a GB parametrix was introduced where the beams
are initialized following the wave-atom tiling of phase space
\cite{MR2362408, deyi08}. Thus, the frequency profile of the initial
Gaussian packets is adapted to the cover depicted in Figure
\ref{F_PhSpT}. (See also \cite{Walden}.) The merit of using wave atoms
is that they are both \emph{isotropic} - as required in order to apply
the GB method - and \emph{parabolic} - in the sense that their
frequency center $\xi$ and the diameter of their essential frequency
support $\ell$ satisfy $\ell^2 \approx \abs{\xi}$. The resulting
parametrix has order $1/2$, performing similarly to the ones based on
curvelets with second-order corrections \cite{MR1644105, MR1724210, MR2165380, dHHSU}.

In this paper, we introduce a GB parametrix for the Dirichlet boundary
value problem (BVP) associated with the wave equation and analyze its
approximation properties. For simplicity, we assume that the boundary
is flat and treat the model case of the half space $\Rdplus = \{x \in
\Rdst: x_1>0\}$,
\begin{eqnarray}
\label{eq_intro_dir}
\left\{
\begin{aligned}
& u_{tt} (t,x) - \speed(x)^2 \Delta_x u (t,x) = 0,
                       &\qquad t\in[0,T], x \in \Rdplus,
\\
&u(0, x)= u_t(0,x) = 0, &\qquad x \in \Rdplus,
\\
&u(t,0,y) = h(t,y), &\qquad t\in[0,T], y \in \Rst^{d-1},
\end{aligned}
\right.
\end{eqnarray}
with $c$ being smooth and bounded below by a positive constant, and
$h$ being prescribed. In the applications to reverse-time continuation
from the boundary, as it appears in imaging, for example, $h$
represents boundary data on an acquisition manifold $\{0\} \times
\Rst^{d-1}$ with time interval $[0,T]$.

We consider a wave-atom like expansion of the boundary value,
\begin{align}
\label{eq_intro_match}
h(t,y) = \sum_\gamma h_\gamma \varphi_\gamma(t,y),
\end{align}
and construct adequate Gaussian beams $\Phi^\pm_\gamma$, so that they match the wavepackets along the boundary:
\begin{align}
\label{eq_intro_match2}
\Phi_\gamma^\pm(t,0,y) \approx \varphi_\gamma(t,y).
\end{align}
As parametrix solution for the Dirichlet problem we then propose:
\begin{align}
\label{eq_intro_par}
\tilde u(t,x) = \sum_\gamma h_\gamma \Phi_\gamma^\pm(t,x).
\end{align}
Based on the effectiveness of the parametrix for the IVP, the
expectation is that $u$ be an approximate solution for the homogeneous
wave equation. The beams $\Phi_\gamma$ have to be designed with the
additional requirement that at initial time the parametrix and its
time derivative be approximately null: $\tilde u(0,x)$, $\tilde
u_t(0,x) \approx 0$. With this provision, the energy estimates
\cite{MR0350177,Triggiani} imply that the parametrix solution is close
to the true one.

A key application of the Gaussian beam method is imaging in reflection seismology \cite{MR2661696, Hill1416} - 
see also \cite{MR2944376} for an analysis of imaging and its connection with solving boundary value problems.
There is extensive work done on computations with Gaussian beams and wavepackets \cite{MR2558781, MR2684020, MR2885562, 
Lex1}. We expect these to be instrumental to the implementation of the parametrix 
that we introduce, thus facilitating accurate computations in the presence of caustics. 

We now elaborate on the details of the program for the construction and analysis of the parametrix outlined above.

\emph{(i) Description of boundary restriction of beams}. At an initial
time, Gaussian beams have a prescribed Gaussian profile on $x$. The GB
theory provides estimates for the evolution of this profile for subsequent
times $t$. In contrast, the approximation in \eqref{eq_intro_match2}
requires describing the restriction of a GB to the boundary
$\{x_1=0\}$ \emph{treating the remaining variables
  $(t,x_2,\ldots,x_d)$ jointly as a spatial variable}. Such an
analysis is the first step of our construction: We consider a general
Gaussian beam and approximately describe its restriction to the
acquisition manifold as a Gaussian wavepacket in all remaining
variables including time. This elaborates on a technique of Katchalov,
Kurylev and Lassas \cite{Kat}, who considered the boundary restriction
of a Gaussian beam that intersects the boundary through a normal ray.

\emph{(ii) Packet-beam matching}. Given a general (isotropic) Gaussian
wavepacket, $\varphi_\gamma$, we use the analysis from (i) to
construct an adequate beam satisfying \eqref{eq_intro_match2}. This
defines a map $\SIC$ that assigns to every phase-space parameter
$\gamma$ indexing the packets in the expansion of the boundary value
$h$ \eqref{eq_intro_match} a set of initial conditions $\SIC_\gamma$
for the ordinary differential equations (ODE) that define a Gaussian
beam.

\emph{(iii) Back-propagation}. The packet-beam matching (ii) is
carried out as follows: given a wavepacket $\varphi_\gamma(t,y)$ with
spatial center $(t_\gamma,y_\gamma)$, we construct the beam
$\Phi_\gamma(t,x)$ so that its spatial center intersects the boundary
precisely at time $t=t_\gamma$. The profile of the beam is specified
at time $t = t_\gamma$ and back-propagated to time $t=0$ by means of
the defining ODEs. Additionally, we specify the mode of $\Phi_\gamma$
- that determines in which direction bicharacteristics are traveled -
so that the beam moves \emph{into} the half-space as time evolves. As
a consequence the beam $\Phi_\gamma$ is mostly concentrated outside
the right half-space at $t=0$, and the parametrix approximately
vanishes at initial time, as required in order to apply energy
estimates.

\emph{(iv) Distortion of the phase-space tiling}. Wavepacket
expansions such as \eqref{eq_pack1} and \eqref{eq_intro_match} follow
certain tilings in phase space. For wave-atom expansions, the
frequency variable is partitioned as shown in Figure \ref{F_PhSpT} and
the space variable is resolved following the dual scaling. The IVP
parametrix relies on this pattern: The technical results in
\cite{Lex3} show that for subsequent times the beams
$\Phi^\pm(t,\cdot)$ in \eqref{eq_intro_IVP} are still adapted to a
similar phase-space tiling, and thus enjoy similar spanning
properties. While the expansion of the boundary value in
\eqref{eq_intro_match} fits the framework of wave atoms, the
phase-space tiling governing the profile of the beams in the proposed
parametrix \eqref{eq_intro_par} is impacted by the packet-beam
matching procedure (iii). The analysis of the approximation error of
the parametrix involves a careful quantification of this effect.

\subsection{Assumptions and results}

We assume that the wave speed $\speed$ is smooth, bounded below by a
positive constant and has globally bounded derivatives of every
order. (The smoothness assumptions could be relaxed at the cost of a
more technical presentation.) The essential condition for the
effectiveness of the parametrix that we introduce is that the rays of
the associated Hamiltonian that take off from the boundary do not
return to the boundary in the time interval in question, so that the
back-propagation step (iii) succeeds (see Section~\ref{sec_ass_source}
for a precise quantitative formulation.)

Besides the standard compatibility condition $h(0,\cdot)=0$, we also
assume that the wavefront set of the boundary value $h$ does not
contain grazing rays. While this assumption is not necessary for the
Dirichlet problem to be well-posed, our parametrix is ultimately based
on oscillatory integrals and the theory of elliptic boundary value
problems, and these techniques do require that the bicharacteristic be
nowhere tangent to the boundary \cite{nirenberg73}. We enforce these
assumptions by examining the wavepacket expansion of $h$
\eqref{eq_intro_match} and by discarding (or down-weighting) those
coefficients that correspond to the undesired wavefront components. In
order to describe this operation in intrinsic terms (i.e.,
independently of the particular wavepacket expansion that the
parametrix uses) we introduce a pseudodifferential cut-off $\psym$
that eliminates grazing rays and consider a modified Dirichlet problem
with boundary condition $u(t,0,y) = \cuth(t,y) := \psym(t,y,D_t,D_y) h
(t,y)$. Denoting by $u$ the solution of the modified problem, we show
that our parametrix solution $\tilde u$ satisfies:
\begin{align*}
\norm{\tilde{u} - u}_{C^0([0,T],H^1(\Rdplus))
                          \cap C^1([0,T],L^2(\Rdplus))}
\leq \CTime  \norm{h}_{H^{1/2}(\Rdst)}.
\end{align*}
In particular, in the highly oscillatory regime, $\hat{h}(\xi)=0$ for
$\abs{\xi} \leq \xi_{\rm min}$, the error can be estimated in terms of
the scale content of the initial data, giving a bound $\xi_{\rm
  min}^{-1/2 } \cdot \norm{h}_{H^{1}(\Rdst)}$.

We also note that the Dirichlet problem is related to a boundary
source problem. Let $u^r$ be the solution to \eqref{eq_intro_dir} and
$u^l$ the solution to the analogous problem on the left-half space
$\mathbb{R}^d_{-}= \{x \in \Rdst: x_1<0\}$. Let
\begin{align}
\label{eq_intro_ubs}
u(t,x) := u^r(t,x) H(x_1)+u^l(t,x) H(-x_1),
\end{align}
where $H$ denotes the Heaviside function. Then \eqref{eq_intro_ubs} is
a microlocal solution to a boundary-source problem:
\begin{align}
\label{eq_intro_bs}
u_{tt} (t,x) - \speed(x)^2 \Delta_x u (t,x)
                   = (B h)(t,x_*) \delta_{x_1}(x),
\end{align}
where $B$ is an adequate boundary operator \cite{MR2944376,
  CStolk2004111}. Hence, our parametrix provides also a microlocal
solution to \eqref{eq_intro_bs}.

\subsection{Related work}

The construction of Gaussian beams dates back to the 1960's, that is,
the work by Babi\v{c} and Buldyrev \cite{MR0426630}~\footnote{The book of Babi\v{c} and Buldyrev was translated; it contains work that Babich and his colleagues published in the Proceedings of the Steklov
Institute in 1968.}. Later, Gaussian beams
were used in the analysis of regularity and propagation of
singularities in partial differential and pseudodifferential equations
by H\"{o}rmander \cite{MR0344944} and Ralston \cite{Ralston}. Without any
attempt to give a comprehensive list of references, we refer to the
foundational work of Popov \cite {MR660272, popov} and Katchalov and
Popov \cite{KP}, and the applications to seismic wave propagation by
{\v{C}}erven{\`y}, Popov and P{\v{s}}en{\v{c}}{\'\i}k
\cite{vcerveny1982computation}. Furthermore, we mention connections
with complex rays in the work of Keller and Streifer \cite{KEST}, and
Deschamps \cite{DESH} in the early 1970s, and the work of Weston \cite{MR1107969}, who studied the
wave splitting in a flat boundary, which is part of the parametrix construction for boundary value problems.

Our study of the Dirichlet problem builds fundamentally on the work of
Katchalov, Kurylev and Lassas \cite{Kat}, who describe the boundary
restriction of a single normally incident Gaussian beam. We extend
this analysis to a collection of multi-scale Gaussian beams with
varying incidence angles.

Wave parametrices based on Gaussian wavepacket expansions go back to
C\'ordoba and Fefferman \cite{cofe78}, and related techniques can be found, for example,
in the work of Smith \cite{MR1724210,MR1644105}, Cand\`{e}s and Demanet \cite{MR2165380}, 
Tataru \cite{MR2208883}, Koch and Tataru \cite{MR2094851}, and Geba and Tataru \cite{MR2353138}.
In the context of Gaussian
beams, Liu, Runborg and Tanushev studied convergence rates of
parametrices for initial data consisting of a finite sum of Gaussian
wavepackets \cite{MR3008843}. Our construction elaborates particularly
on the work of Bao, Qian, Ying, and Zhang \cite{Lex1,Lex3} who treat
the decomposition of general (multi-scale) initial data. Indeed, much
of the technical work in this article is devoted to show that the
packet-beam matching procedure described above yields a family of
beams that approximately resemble at an initial time the wavepackets
used in \cite{Lex3} as a starting point for the IVP. This task leads
us to introduce the notion of \emph{well-spread family of Gaussian
  beams}, that abstracts the properties that make a multi-scale GB
parametrix effective. As a by-product we revisit the main result of
\cite{Lex3} and give a variant of the parametrix where the initial
data is expanded into \emph{exact} Gaussian wavepackets rather than
frequency truncated ones. Our proofs resort to the notion of
\emph{wave molecules} and provide an alternative to some of the
computations in \cite{Lex3}. We also mention a link of our analysis
with the work of Laptev and Sigal \cite{MR1767504} who constructed a
parametrix for the time-dependent Schr\"{o}dinger equation.

The introduction of pseudodifferential cut-offs to remove grazing rays
is a standard practice \cite{MR2944376}. In our case, we also need to
describe how this operation is reflected on the wavepacket
expansion. To this end, we show that zero-order pseudodifferential
operators are almost diagonalized by packets with wave-atom geometry,
paralleling related results for curvelets \cite{dHHSU}. While the
no-grazing ray assumption is standard in the literature, we note that
the existence and uniqueness theorems for initial-boundary value
problems do not involve these transversality conditions, and, indeed,
Melrose \cite{melroseAiry} introduced a class of operators to treat
glancing points, and used them in a parametrix construction
\cite{melroseDiffractive}; see also \cite{melrose1987boundary}.

In relation to Gaussian beam expansions, we also mention the related notion 
of frozen Gaussian beams \cite{MR2903799, MR2865800}, where the Gaussian packets 
that approximate the solution to the wave equation may themselves not be asymptotic solutions.

\subsection{Organization}

In Section \ref{sec_wp} we introduce multi-scale Gaussian beams and a
frame of wave-atom like Gaussian wavepackets. We also discuss how to
parametrize GB by their initial conditions, introduce the relevant
notation, and collect some facts about the defining ODEs. The
construction of the frame is similar to others in the literature and
the particulars are briefly discussed in Appendix
\ref{sec_frame_proofs}.

The notion of well-spread family of Gaussian beams is introduced in
Section \ref{sec:well_spread}. We show that such families enjoy
suitable uniformity properties, satisfy Bessel bounds and are
approximate solutions to the homogeneous wave equation. We revisit the
parametrix for IVP, presenting a variant of the main result of
\cite{Lex3}. Some of the corresponding technical work is needed later
in more generality and is therefore presented in the appendices.

In Section \ref{sec_dir} we introduce the Dirichlet BVP and the
corresponding assumptions. We also discuss how the assumptions are
reflected by the frame expansion of the boundary value. (This relies
on results developed in the appendices.)

In Section \ref{sec_spation-temporal} we analyze a family of beams at
times where their spatial centers intersect the acquisition
manifold. We identify a suitable Gaussian profile for the
corresponding restrictions, both near the boundary intersection time,
and away from it. The analysis of Section \ref{sec_spation-temporal}
is then used as a guide in Section \ref{sec_match} to introduce the
beam-packet matching procedure. The outcome of this process is
analyzed using tools from Section \ref{sec:well_spread}. For clarity,
the most technical parts of this analysis are postponed to Section
\ref{sec_pr_th_match_ws}. The performance of the parametrix is finally
analyzed in Section \ref{sec_main}.

Appendix \ref{AppendixA} collects various estimates for the wave
equation and Gaussian beams. Appendix \ref{AppendixB} presents the
notion of wave molecule and develops several technical results that
are needed throughout the paper. The notion of wave molecule is a
minor generalization of the one of wave atom \cite{demphd, MR2362408,deyi08}
and the results we
derive are in the spirit of the ones
developed for curvelets in \cite[Appendix A]{dHHSU}.  Appendix
\ref{AppendixC} presents a result of independent interest on the
almost-diagonalization of pseudodifferential operators by a frame of
wave molecules. In this paper, that result is used to analyze how the
assumptions on the boundary value are reflected by its frame
expansion. Appendix \ref{sec_num} provides details on some of the figures.

We now introduce basic notation. More notation is introduced
throughout the paper; a reference table can be found in Appendix
\ref{sec_table_not}.

\subsection{Notation}
\label{sec:not}
We write $x=(x_1, x_*) \in \Rst \times \Rst^{d-1}$,
$\abs{x} = \abs{x}_2$ denotes the Euclidean norm,
$\Rdplus =
(0,+\infty) \times \Rst^{d-1}$, and $\Rst^d_T= [-T,T] \times \R^{d-1}$. We use the notation
$B_r(x)$ for the Euclidean ball of center $x$ and radius $r$. $\Re(z)$ and $\Im(z)$ denote
respectively the real and imaginary part of $z \in \bC$. This notation extends to vectors and
matrices componentwise. Generic constants are denoted by $C, C', C_0$ and their meaning may change
from line to line. Specific constants are given more descriptive notation.

For two non negative functions $f,g$, $f \lesssim g$
means that there exist a constant $C>0$ such that $f(x) \leq C g(x)$, for all $x$. We write $f
\asymp g$ if $f \lesssim g$ and $g \lesssim f$.

Given a domain $\Omega \subseteq \Rdst$, we let $\Cinfb(\Omega)$ be the class of $C^\infty(\Omega)$
functions $f$ such that for every multi-index $\alpha$, $\partial_x^\alpha f \in L^\infty(\Omega)$.

The identity matrix is denoted as $I_d \in \Rst^{d\times d}$. For a matrix
$A \in \bC^{d\times d}$, $A \gtrsim I_d$ means that there exists a constant $\consalone>0$ such that
$A - \consalone \cdot I_d$ is a positive matrix (i.e. Hermitian and with non-negative spectrum).
For a constant $\consalone \geq 0$, we
sometimes write $A \geq \consalone$ instead of $A \geq \consalone I_d$.

The Fourier transform is normalized as:
$\hat{f}(\xi) = \int_{\Rdst} f(x) e^{-2 \pi i x \xi} dx$.
We also let $D_x = \frac{1}{2\pi i} \partial_x$; when it is clear from the context we further denote
$D=D_x$. For a symbol $\psym: \Rdst \times \Rdst \to \bC$, $\psym(x,D)$ denotes its
Kohn-Nirenberg quantization. (Most of our statements generalize immediately to other quantizations.)

The phase-space metric is the function $d: \Rdst \times \left(\Rdst \setminus \{0\}\right) \to
[0,+\infty)$,
\begin{align*}
d((x,\xi),(x',\xi')) =
|\xi||\xi'||x-x'|^{2} + \left|\xi- \xi'\right|^{2}
\qquad (x,\xi), (x',\xi') \in \Rst^{2d}.
\end{align*}

The Hamiltonians are defined as $H^{+}(x,p) = \speed(x) \abs{p}$, $H^{-}(x,p) = -\speed(x)
\abs{p}$ and $H$ denotes generically either $H^+$ or $H^-$. Sometimes we denote time derivatives
with a dot, e.g. $\dot{x}(t) = \partial_t x(t)$.

Throughout the article, $\speed$ denotes a fixed function $\speed \in \Cinfb$
(called velocity) that is assumed to be bounded below away from $0$; i.e.,
\begin{equation}
\label{eq_speed_bounds}
\cspeed := \inf_{x\in\Rdst} \speed(x) >0,
\end{equation}
and $\partial^\alpha_x \speed \in L^\infty(\Rdst)$ for every multi-index $\alpha$.

\section{Gaussian wavepackets and Gaussian beams}
\label{sec_wp}

\subsection{Construction of a frame}
\label{sec_frame}
We construct a frame of wave-atom-like Gaussian packets
with Gaussians as basic waveforms. We start by introducing a frequency cover. For $j \geq 1$ we
let $\sett{\pointjk: k=1, \ldots, n_j} \subseteq \Rdst$
be a set of points such that:
\begin{itemize}
\item The family $\sett{B_{2^j}(\pointjk), k=1,\ldots, n_j}$ is disjoint and each member is
contained in the corona $\corona_j =
B_{4^{j+1}}(0) \setminus B_{4^{j}}(0)$.
\item $\corona_j \subseteq \bigcup_{k=1}^{n_j} B_{2^{j+1}}(\pointjk)$.
\end{itemize}
Hence $\{B_{2^{j+1}}(\pointjk): k=1,\ldots, n_j, j \geq 1\}$
is a cover of $\{\xi \in \Rdst: \abs{\xi} \geq 4\}$; see Figure \ref{F_PhSpT}.

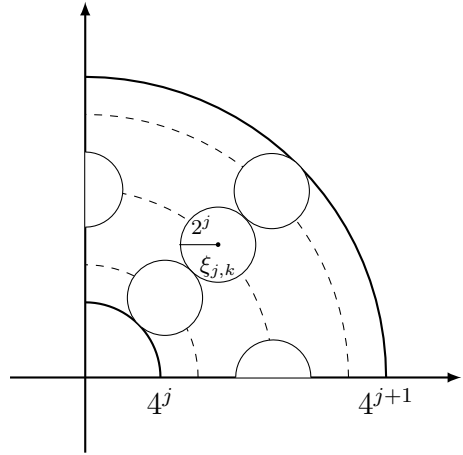
\begin{figure}[ht!]
\begin{center}
\begin{tikzpicture}[>=latex]

\draw[thick,->] (-1,0) -- (5,0) ;
\draw[thick,->] (0,-1) -- (0,5) ;

 \draw [dashed,domain=0:90] plot ({3.5*cos(\x)}, {3.5*sin(\x)});
  \draw [dashed,domain=0:90] plot ({1.5*cos(\x)}, {1.5*sin(\x)});
    \draw [dashed,domain=0:90] plot ({2.5*cos(\x)}, {2.5*sin(\x)});
  \draw [thick,domain=0:90] plot ({cos(\x)}, {sin(\x)});
 \draw [thick,domain=0:90] plot ({4*cos(\x)}, {4*sin(\x)});
\node  at (1, 0) [below] {$4^{j}$};
\node  at (4, 0) [below] {$4^{j+1}$};

\draw  [fill=white] (1.767,1.767) circle (0.5);

\draw  [fill=white] (1.06,1.06)circle (0.5);
\draw  [fill=white] (2.48,2.48)circle (0.5);

\node at (1.767,1.767) [below] {\scriptsize$\pointjk$};
\node at (1.767,1.767){\huge.};
\node at (1.55,1.767) [above] {\tiny $2^{j}$};
\draw  (1.25,1.767) -- (1.767,1.767);

\draw [domain=0:180, fill = white] plot ({(2.5)+(0.5)*cos(\x)}, {(0.5)*sin(\x)});
\draw [domain=-90:90, fill = white] plot ({(0.5)*cos(\x)}, {2.5+(0.5)*sin(\x)});
\end{tikzpicture}
\end{center}
\caption{Frequency-space tiling for the frame.}
\label{F_PhSpT}
\end{figure}

Note that $\abs{\pointjk} \asymp 4^j$. In addition, comparing the volumes of $\corona_j$
to those of the unions of the balls $B_{2^j}(\pointjk)$ and $B_{2^{j+1}}(\pointjk)$, it follows that
\begin{align}
\label{eq_nj}
n_j \asymp 2^{jd}, \qquad j \geq 1.
\end{align}
For convenience, we also introduce the rescaled vector
\begin{align}
\label{eq_pjkt}
\pointjkt = 2\pi\frac{\pointjk}{4^j}.
\end{align}
Hence, $\pointjkt$ is approximately normalized:
$|\pointjkt| \asymp 1$.

We let $\func(x)$ be the Gaussian function
\begin{equation}\label{eq_norm_Gauss}
\func(x) = \normd e^{-\pi\abs{x}^2}, \qquad x\in\R^d,
\end{equation}
and define modulated and scaled waveforms adapted to the frequency cover
\[
\funcjk(x)= \normdj e^{2\pi i \pointjk x}\func(2^j x), \quad j\geq 1, k=0,\dots, n_j,
\]
so that $\widehat\funcjk$ is essentially concentrated on $B_{2^j}(\pointjk)$.
We let $\Lambda \subseteq \Rdst$ be a (full rank) lattice and define
\begin{equation}
\label{eq:def_gamma}
\Gamma = \sett{(j,k,\lambda):j\geq 1, k=0,\ldots,n_j, \lambda \in \Lambda},
\end{equation}
and
\[
\func_\gamma(x)=
\funcjkl(x)= \funcjk(x - 2^{-j}\lambda), \qquad \gamma=(j,k,\lambda) \in \Gamma.
\]
Explicitly,
\begin{align}
\label{eq_def_funcs}
\funcjkl(x) = 2^{j\frac{d}{2}} e^{2 \pi i \pointjk (x-2^{-j}\lambda)}
\func(2^jx-\lambda),
\qquad (j,k,\lambda) \in \Gamma.
\end{align}
See Figure \ref{fig:packet} for a plot. For an index $\gamma \in \Gamma$, we often refer implicitly
to the notation $\gamma=(j,k,\lambda)$.
\begin{figure}
 \centering
  \begin{minipage}[c]{0.45\textwidth}
    \includegraphics[width=\textwidth]{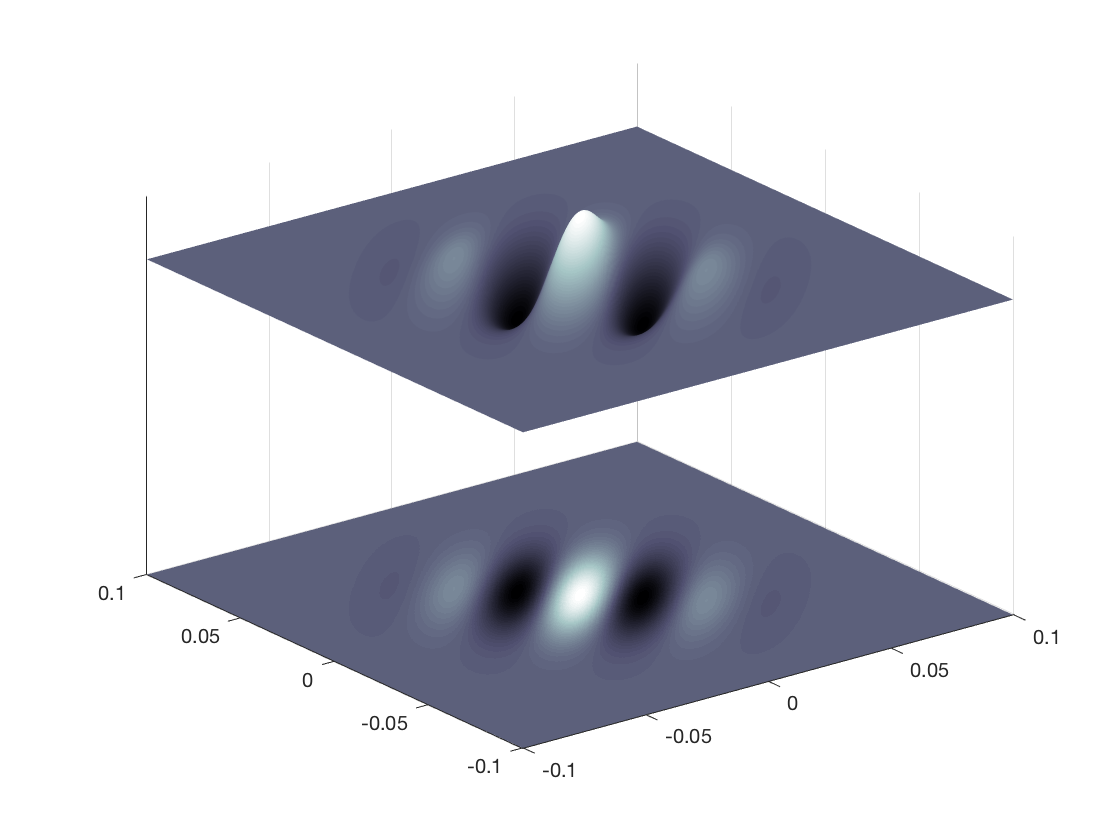}
  \end{minipage}
  \hfill
  \begin{minipage}[c]{0.45\textwidth}
    \includegraphics[width=\textwidth]{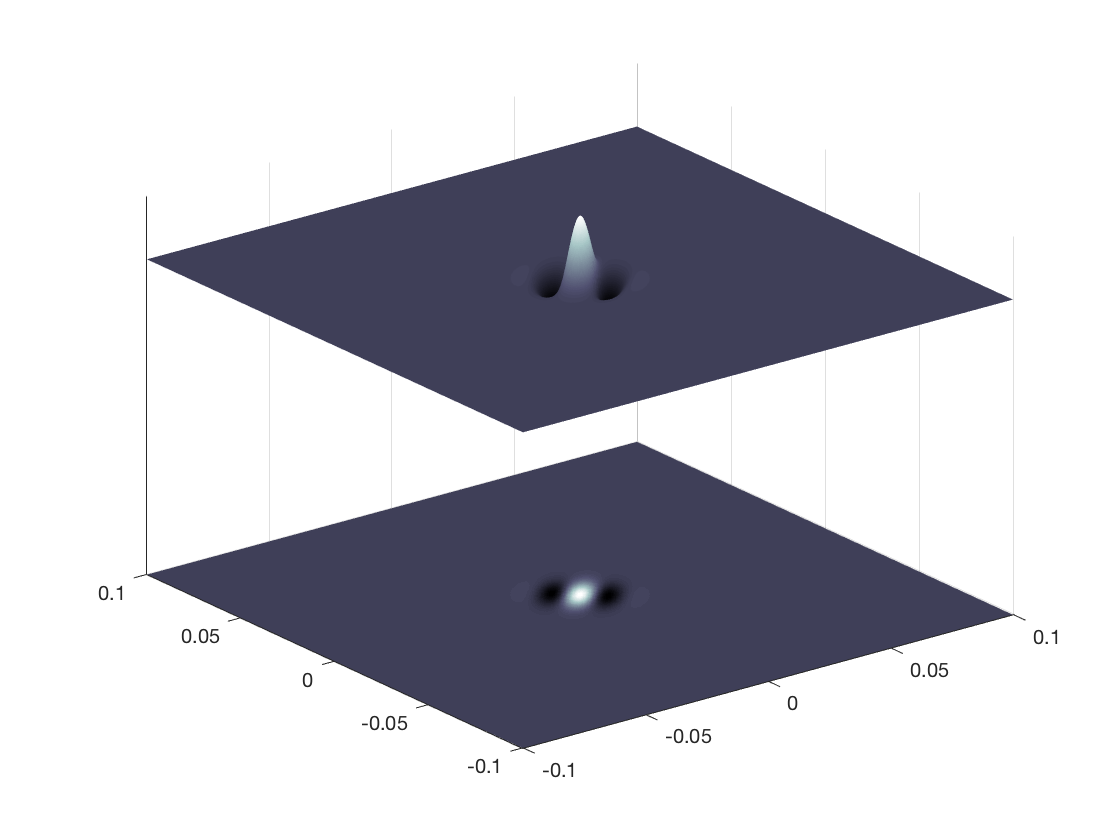}
  \end{minipage}
 \caption{On the left, the space profile and contour plot of a frame element with
 $j=2$. On the right, the same plot with a Gaussian window contracted by a factor
 of $2\pi\ln(16)$, which is better adapted for certain numerical examples. See Appendix
\ref{sec_num}.}
 \label{fig:packet}
\end{figure}

Although we are mainly interested in high frequency expansions:
$\sum_{\gamma \in \Gamma} f_\gamma \varphi_\gamma$, in order to expand an arbitrary function we
need to provide wavepackets adapted to the zeroth-scale. To keep the notation concise,
we let $\funczz := \varphi$, augment the index set $\Gamma$ by
\begin{align}
\label{eq:def_gammafull}
\Gammafull := \Gamma \cup \{(0,0,\lambda): \lambda \in \Lambda\},
\end{align}
and define zeroth-scale wavepackets as:
$\varphi_{0,0,\lambda} := \func(x-\lambda)$.
The complete set of wavepackets can be written as:
\begin{align*}
\frameset = \sett{\varphi_{\gamma}: \gamma \in \Gammafull}.
\end{align*}
We now show that the system thus constructed is indeed rich enough to represent any function.
\begin{theorem}
\label{T_FrProp}
For an adequate lattice $\Lambda \subseteq \Rdst$, the system
$\frameset$ is a frame for the inhomogeneous Sobolev spaces
$\Hs(\Rdst)$, with $-1 \leq s \leq 1$. More precisely, the frame operator
\begin{equation}
\label{eq_FROpt}
\frameop f = \sum_{\gamma \in \Gammafull} \ip{f}{\varphi_\gamma}  \varphi_\gamma
\end{equation}
is invertible on $H^s(\Rdst)$ for $-1\leq s\leq 1$. In addition, we have the norm equivalence
\begin{equation}
\label{eq_FrProp}
\norms{f}^2 \asymp \sum_{\gamma \in \Gammafull}4^{2js}\abs{\ip{f}{\varphi_\gamma}}^2.
\end{equation}
\end{theorem}
We remark that in \eqref{eq_FROpt} and \eqref{eq_FrProp} the symbol $\ip{f}{\varphi_\gamma}$
denotes the standard
$L^2$ inner product. The theorem is proved using a variant of Daubechies's criterion for wavelets.
Details are provided in Appendix \ref{sec_frame_proofs}. The construction provides a concrete
criterion to choose the lattice
$\Lambda$ and
$A \norms{f} \leq \norms{\frameop f} \leq B \norms{f}$ can be satisfied
with $B/A$ reasonably small (see Figure \ref{fig:control}). Hence, the numerical inversion of
$\frameop$ is
well-conditioned.

From now on we fix a lattice $\Lambda$ such that the conclusion of Theorem \ref{T_FrProp} holds.

As a consequence of Theorem \ref{T_FrProp},
every $f \in H^s(\Rdst)$ can be represented by an $H^s$-convergent series
\begin{equation}
\label{eq_repres}
f = \sum_{\gamma \in \Gammafull} f_\gamma \varphi_\gamma,
\qquad \mbox{with } f_\gamma := \ip{f}{\frameop^{-1}\varphi_\gamma},
\end{equation}
and
\begin{align*}
\norm{f}^2_{H^s} \asymp \norm{\frameop^{-1} f}^2_{H^s} \asymp
\sum_{\gamma \in \Gammafull} 4^{2js}\abs {f_\gamma}^2.
\end{align*}
\subsection{Operating on the frame expansion}
\label{sec_cutoff}
We will be mostly interested in the higher scales $j \geq 1$. We can truncate the representation in
\eqref{eq_repres},
\begin{equation}
\label{eq_repres_trunc}
\tilde f = \sum_{\gamma \in \Gamma} f_\gamma \varphi_\gamma,
\end{equation}
and it is easy to see that the error can be bounded as
\begin{align}
\label{eq_trunc_error}
\norm{f-\tilde f}_{H^{1}} \lesssim \norm{f}_{H^{-1}}.
\end{align}
Hence, in the highly oscillatory regime, we only need to consider expansions of the form
\eqref{eq_repres_trunc}.

More generally, we use pseudodifferential cut-offs to operate microlocally on a function $f$ and we
wish to approximately implement those operations by acting directly on the expansion
in
\eqref{eq_repres_trunc}. We recall that a symbol $\psym:\Rdst\times\Rdst \to \bC$ belongs to the
H\"ormander class $S^{0}_{1,0}(\Rdst\times\Rdst)$
if
\begin{align*}
\abs{\partial_x^\beta \partial_\xi^\alpha \psym(x,\xi)}
\leq C_{\alpha,\beta} (1+\abs{\xi})^{-\abs{\alpha}},
\end{align*}
for all multi-indices $\alpha,\beta$. The next lemma will be an important technical tool, and is
proved in more generality in Appendix \ref{AppendixC}.
\begin{theorem}
\label{th_frame_cutoff}
Let $\psym \in S^{0}_{1,0}(\Rdst\times\Rdst)$. Then, for $s \in [1/2,1]$
and $f \in H^s(\Rdst)$,
\begin{align}
\norm{\psym(x,D) f - \sum_{\gamma\in\Gamma} \psym(2^{-j}\lambda,\pointjk) f_\gamma
\varphi_\gamma}_{H^s}
\lesssim \norm{f}_{H^{s-1/2}}.
\end{align}
(Here, $\psym(x,D)$ is the Kohn-Nirenberg quantization of $\psym$,
and $f_\gamma := \ip{f}{\frameop^{-1}\varphi_\gamma}$ are the high-scale frame coefficients
of $f$.)
\end{theorem}

\subsection{Gaussian Beams}\label{sec:GB}
We summarize the construction of Gaussian beams, following the treatment of Katchalov, Kurylev and
Lassas \cite{Kat}.
Consider the wave equation,
\begin{equation}\begin{split}
   \p_{t}^{2} u(t,x) - \speed^{2}(x) \Delta_x u(t,x) &{}= 0
\label{Eq4}
\end{split}\end{equation}
with $\speed \in C^{\infty}(\R^{d})$ (strictly) positive with bounded derivatives
of all orders. We  seek formal asymptotic solutions (in a
``moving'' frame of reference) of the form
\begin{equation}\label{GB_Def}
   \Phi(t,x) =  A(t,x) \me^{\mi \omega \theta(t,x)},
\end{equation}
where the phase function $\theta$ and amplitude function $A$ are
smooth complex-valued functions of $(t,x)$, and $\omega$ is the frequency
parameter. Substituting this asymptotic trial solution into the wave
equation, and extracting the leading order terms in $\omega$, one finds
the eikonal and transport equations,
\begin{align}
   (\p_{t} \theta)^{2}
         - c |\nabla_{x} \theta|^{2} &{}= 0,
\label{EIK}\\
   2\p_{t} A \, \p_{t} \theta - 2\speed^{2} \nabla_{x} A \cdot
         \nabla_{x} \theta + A \, (\p_{t}^2 \theta
              - \speed^{2} \Tr(\Delta_{x} \theta)) &{}= 0.
\label{TRA}
\end{align}
We factorize the eikonal equation into
\begin{equation}\label{eq:ham}
   \p_{t} \theta^{\pm} + H^{\pm}(x,\nabla_x \theta^{\pm}) = 0,
\end{equation}
where $\theta^{\pm}$ corresponds to positive and negative frequencies respectively, and
\begin{equation}
\label{eq:hams}
H^{\pm}(x,p) = \pm \speed(x) |p|,
\end{equation}
are the signed Hamiltonians. The propagation of singularities of solutions to the wave equation is
described by the bicharacteristics, $x^{\pm}(t)$, $p^{\pm}(t)$, satisfying the Hamilton system
\begin{equation}\label{eq_ode1}
   \dot{x}^\pm(t) = \phantom{-}\partial_p H^\pm,\quad
   \dot{p}^\pm(t) = -\partial_x H^\pm,
\end{equation}
supplemented with initial conditions, $x^\pm(0) = x_0$, $p^\pm(0) = p_0$. For the sake of
simplicity, we drop the superscript $\pm$ when we do not need to differentiate the two solutions.
In particular, $H$ denotes either $H^+$ or $H^-$.

A Gaussian beam is a solution of the type \eqref{GB_Def}, the phase function of which is assumed
to satisfy
the conditions,
\begin{align}
   \Im \theta(t,x(t)) &{}= 0,
\label{Pf2}\\
   \Im \theta(t,x) &{} \geq F_0(t) \abs{x - x(t)}^2,
\label{Pf1}
\end{align}
where $F_0(t)$ is a continuous positive function. To construct such a
solution, one expands the phase function to second order,
\begin{equation}\label{PHEXP}
   \theta(t,x) = \theta_{0}(t) + p(t)\cdot(x - x(t))
               + \tfrac{1}{2} (x - x(t))^T M(t) (x - x(t)),
\end{equation}
and the amplitude function to zero order,
\begin{equation}
   A(t,x) = A(t,x(t)) = A(t).
\end{equation}
The phase function along the characteristic $\theta_{0}(t) =
\theta(t,x(t))$ satisfies
\[
   \dot{\theta_0}(t) = (\partial_t \theta)(t,x(t))
           + p(t) \cdot (\p_{p} H)(t,x(t))) = 0
\]
in view of the homogeneity of $H$; hence, $\theta_0(t)$ can be taken
to be zero.

The matrix  $M$ satisfies the Riccati
equation,
\begin{equation}\label{eq_ode2}
   \dot{M}(t) + D(t) + B(t) M(t) + M(t) B(t)^{t} + M(t) X(t) M(t) = 0,
\end{equation}
where $B(t), X(t), D(t)$ are $d \times d$ matrices with elements given by the
second-order derivatives of the Hamiltonian,
\[
   D_{ij}(t) = \p_{x_{i}}\p_{x_{j}}H, \quad
   B_{ij}(t) = \p_{x_{i}}\p_{p_{j}}H, \quad
   X_{ij}(t) = \p_{p_{i}}\p_{p_{j}}H,
\]
evaluated along the bicharacteristic
$(x,p) = (x(t), p(t))$. The Riccati equation is supplemented with the
initial condition $M(0) = M_0$. The symplectic structure of the
Hamilton system implies that $M(t)$ is symmetric and has a
positive definite imaginary part provided that it initially does
(see Lemma \ref{lemma_unif_flow_2} and \cite[Lemma 2.56]{Kat}).

The amplitude function $A$ satisfies the
transport equation
\begin{equation}\label{eq_ode3}
   \dot{A}(t) + \frac{A(t)}{2H}
   \left(\mstrut{0.45cm}\right.
   c^{2} \Tr(M(t)) - \p_{p} H \cdot \p_{x} H
                    - (\p_{p} H)^{T} M(t) \, \p_{p} H
                              \left.\mstrut{0.45cm}\right) = 0,
\end{equation}
where $H$ and its derivatives are evaluated along the
bicharacteristic $(x,p) = (x(t), p(t))$. This equation is
supplemented with the initial condition $A(0) =
A_0$. We find solutions, $M(t)=M^\pm(t)$ and
$A(t)=A^\pm(t)$; the latter can
written as
\begin{equation}\label{AMPL}
   A(t) = \frac{A_{0}}{|\det Y(t)|^{1/2}}\
      \exp\left( \int_{0}^{t} \Tr[X(x(s),p(s)] \, ds \right)
\end{equation}
with
\begin{equation*}
   \dot{Y}(t) - B^t Y(t) - X M Y(t) = 0
   \quad \text{and}\quad Y_{0} = I,
\end{equation*}
see \cite[Lemma 6.3]{Lex1} for details.

\begin{figure}[!tbp]
  \centering
  \begin{minipage}[l]{0.45\textwidth}
    \includegraphics[width=\textwidth]{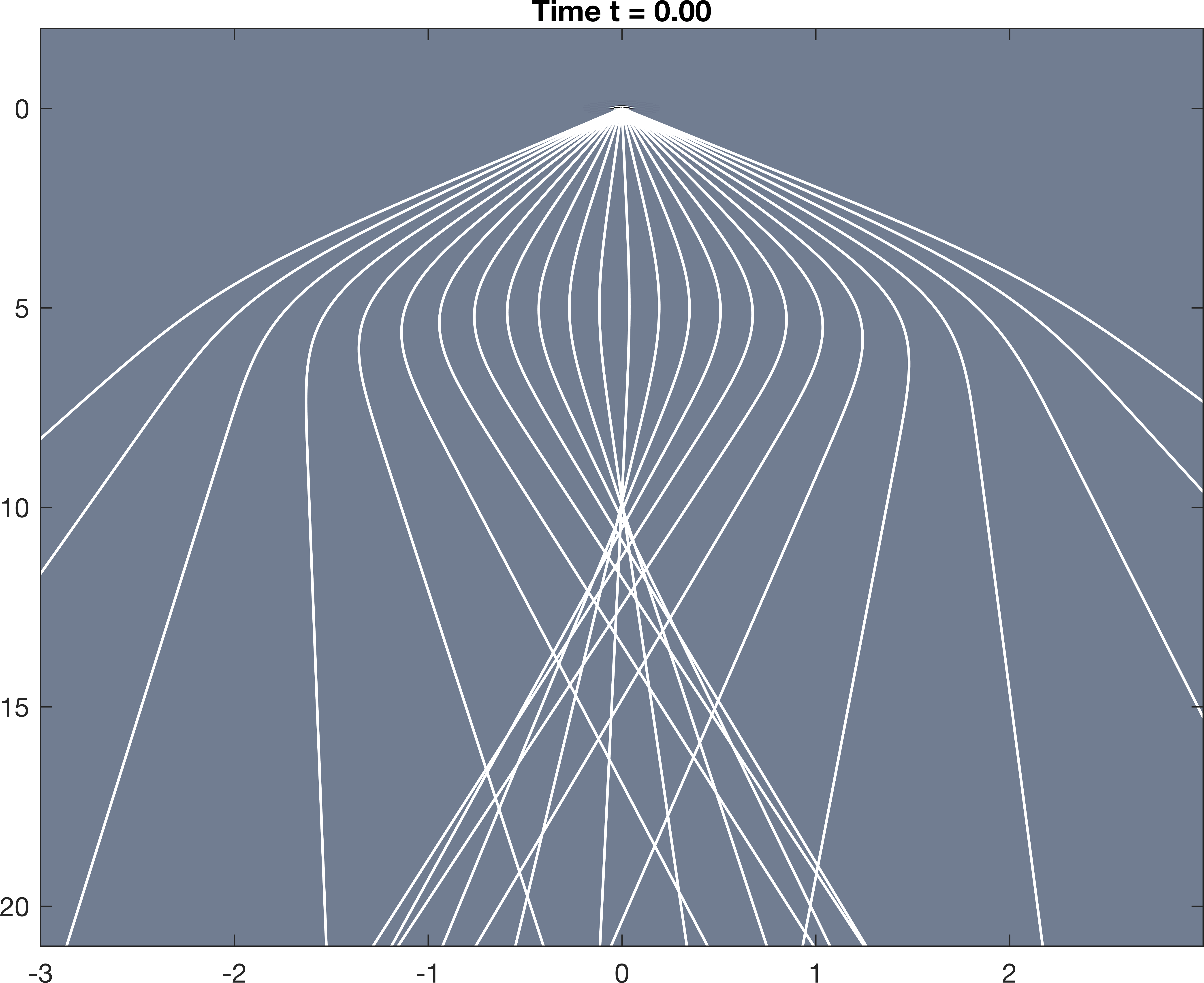}
  \end{minipage}
    \begin{minipage}[l]{0.45\textwidth}
    \includegraphics[width=\textwidth]{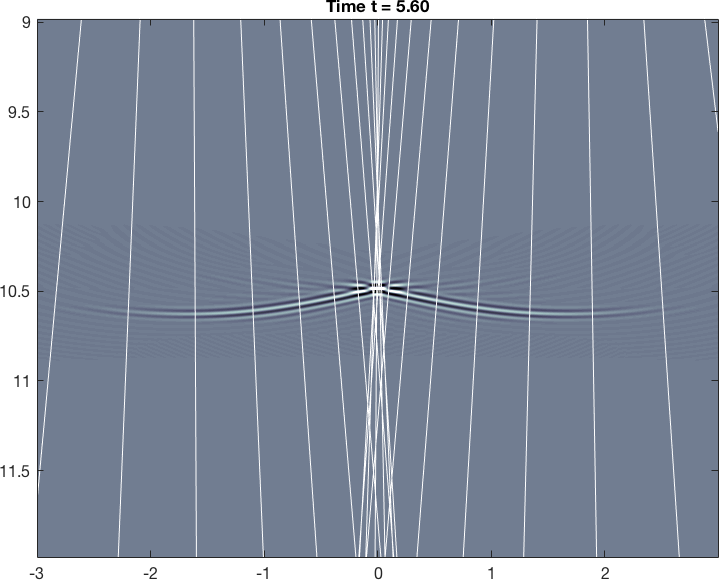}
  \end{minipage}
\caption{The projected characteristics
corresponding to the velocity $\speed(x_1,x_2) = 2 - 0.4*\exp{(-(x_1^2 +(x_2-5)^2)/3)}$,
and a detail of the evolution of a front of Gaussian packets. Initially localized on the
boundary, the front goes through a caustic at time $t\approx 5.60$.
See Appendix \ref{sec_num}.}
\label{fig:caustic}
\end{figure}

\begin{figure}
  \centering
  \begin{minipage}[l]{0.45\textwidth}
    \includegraphics[resolution=400, width=\textwidth]{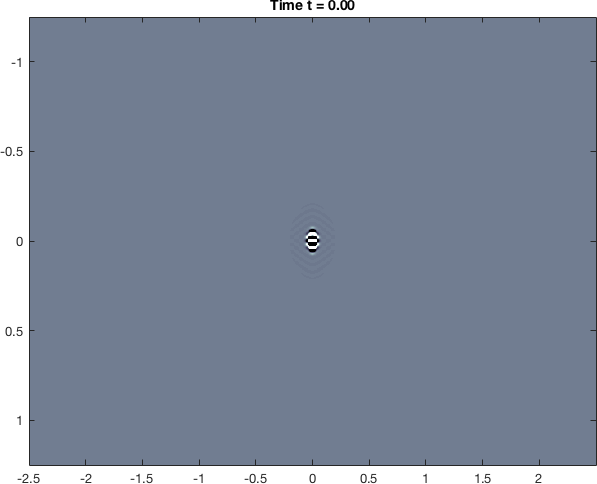}
  \end{minipage}
    \begin{minipage}[l]{0.45\textwidth}
    \includegraphics[resolution=400, width=\textwidth]{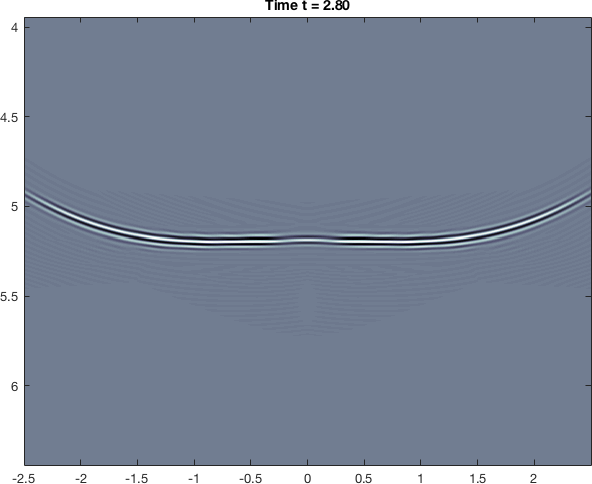}
  \end{minipage}
    \begin{minipage}[l]{0.45\textwidth}
    \includegraphics[resolution=400, width=\textwidth]{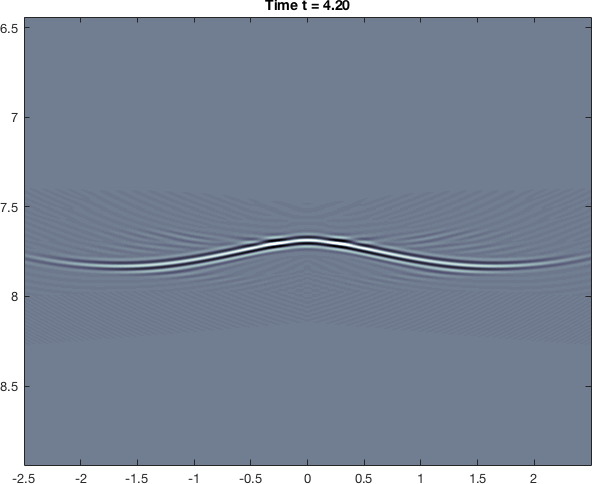}
  \end{minipage}
  \begin{minipage}[l]{0.45\textwidth}
    \includegraphics[resolution=400, width=\textwidth]{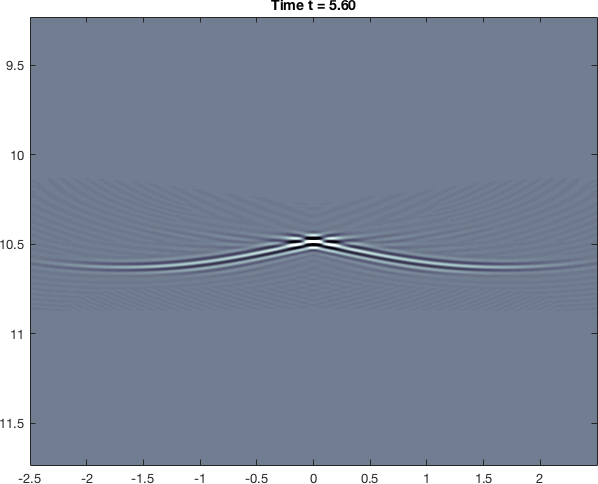}
  \end{minipage}
  \begin{minipage}[l]{0.45\textwidth}
    \includegraphics[resolution=400, width=\textwidth]{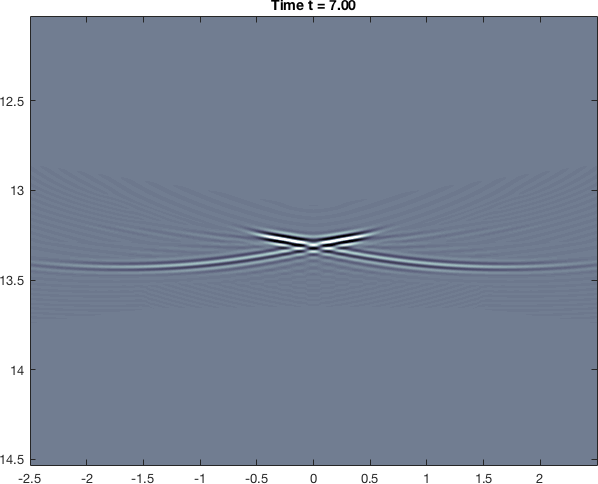}
  \end{minipage}
  \begin{minipage}[l]{0.45\textwidth}
    \includegraphics[resolution=400, width=\textwidth]{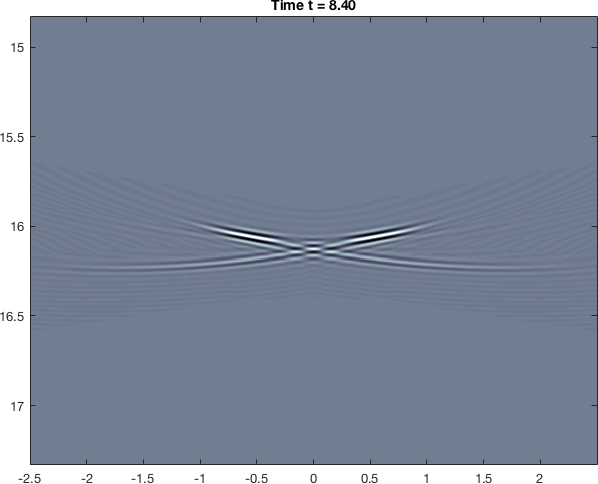}
  \end{minipage}
\caption{Detailed evolution of the wavefront corresponding to Figure~\ref{fig:caustic}.}
\label{fig:caustic2}
\end{figure}

In what follows, we discuss families of Gaussian beams, with the goal of describing superimpositions
and time evolution. See Figures \ref{fig:caustic} and \ref{fig:caustic2} for plots of a front of
Gaussian beams going through a caustic, and Appendix \ref{sec_num} for details on the example.

\subsection{Sets of initial conditions for Gaussian beams}
\label{sec_sets}
We consider families of Gaussian beams associated with sets of parameters described as follows. We
let $\Gamma_0$ be a subset of $\Gamma$ and $\SIC$ be a map
\begin{equation}\label{eq_set_param}
\SIC_{\gamma}=(\omega_\gamma, a_\gamma, \xi_\gamma, \newA_{\gamma},\newM_\gamma) \in
\Rst_+ \times \Rdst \times (\Rdst \setminus \sett{0})\times (\Rst \setminus \{0\}) \times \bC^{d
\times d},
\qquad \gamma \in \Gamma_0,
\end{equation}
such that $\newM_\gamma$ is symmetric and $\Im \newM_\gamma >0$
for all $\gamma \in \Gamma_0$. We associate two functions
$\Phi^+_{\gamma},\Phi^-_{\gamma}:\Rst \times \Rdst \to \bC$ that we now describe.
To simplify the notation we drop the superscript $+$,$-$. Let $x_\gamma (t)$, $p_\gamma (t)$,
$M_\gamma (t)$, $A_\gamma
(t)$ be the solutions to the set of ODEs defined in \eqref{eq_ode1}, \eqref{eq_ode2} and
\eqref{eq_ode3}, supplemented
with initial conditions:
\begin{align}
\label{eq_ic1}
&x|_{t=0} = a_\gamma,
\\
\label{eq_ic2}
&p|_{t=0}= 2 \pi \frac{\xi_\gamma}{\omega_\gamma},
\\
\label{eq_ic3}
&M|_{t=0}= 2 \pi \newM_\gamma,
\\
\label{eq_ic4}
&A|_{t=0}= \newA_\gamma {\omega_\gamma}^{\frac{d}{4}}.
\end{align}
We now define the beams by
\begin{align}\label{eq:GB_elem}
\Phi_{\gamma}(t,x) = A_\gamma(t) e^{i \omega_{\gamma} \theta_\gamma(t,x)},
\end{align}
with
\begin{align}\label{eq:GB_phase}
\theta_\gamma (t,x) = p_\gamma(t)\cdot (x-x_\gamma(t))
+ \tfrac{1}{2} (x-x_\gamma(t))\cdot M_\gamma(t) (x-x_\gamma(t)).
\end{align}
The ODEs in \eqref{eq_ode1}, \eqref{eq_ode2} and \eqref{eq_ode3} have globally defined unique
solutions
for the initial conditions given by \eqref{eq_ic1}-\eqref{eq_ic4}. Indeed, the system of ODEs in
\eqref{eq_ode1} is the flow
associated with the Hamiltonian $H$ and, due to the homogeneity of $H(x,p)$ in $p$, it is
solvable as long as the
initial condition $p|_{t=0}$ is non-zero. That is why we require that $\xi_\gamma \not= 0$.
Once the Hamiltonian flow $(x,p)$ is defined,
\eqref{eq_ode2} has a globally defined unique solution because
the initial datum is symmetric and has a positive imaginary part
\cite[Lemma 2.56]{Kat}. Finally, \eqref{eq_ode3} has also a unique global solution,
since it is a linear ODEs with continuous coefficients.

\begin{rem}
When we need to emphasize the dependence on the choice of sign for $H$ we write: $\Phi^\pm_\gamma$,
$x_\gamma^\pm(t),
p_\gamma^\pm(t), M_\gamma^\pm(t), A_\gamma^\pm(t)$. We stress that these functions depend not only
on the index
$\gamma$, but also on the underlying map from \eqref{eq_set_param}, that describes how to associate
with $\gamma$
initial conditions for the ODEs defining the beam. When we need to stress this dependence we use
further superscripts.
\end{rem}

\begin{rem}
\label{rem:upsilon}
By abuse of language, we often refer to a set of GB parameters
$\Upsilon = \sett{\SIC_\gamma: \gamma \in \Gamma_0}$, although it is not the set $\Upsilon$, but the
underlying map $\SIC$ that matters. Hence, $\Upsilon$ should be considered as an indexed set, that is
formally equivalent to the map $\SIC$.
\end{rem}

\subsection{Standard initial conditions}\label{sec:stset}
We now describe the canonical set of GB parameters, defined so that the corresponding Gaussian beams
at time $t=0$ coincide with the higher-scale part of the frame $\frameset$. We define \emph{the
standard set of Gaussian beam parameters} as
the set $\Upsilon^{st}=\sett{\alphast_\gamma: \gamma \in \Gamma}$ given by
\begin{align}
\label{eq_standard}
\alphast_\gamma = (4^j,2^{-j}\lambda, \pointjk, 2^{\frac{d}{4}}, \ii I_d),
\qquad \gamma = (j,k,\lambda) \in \Gamma.
\end{align}
The corresponding beams are denoted
$\{\Phistpm_{\gamma}: \gamma \in \Gamma\}$.
\begin{obs}\label{ob_beam_frame}
For the standard set of parameters $\Upsilonst$:
\[
\Phistpm_{\gamma}(0,x) = \funcjkl(x), \qquad \gamma = (j,k,\lambda) \in \Gamma.
\]
\end{obs}
This follows by substituting \eqref{eq_standard} into  \eqref{eq_ic1}-\eqref{eq_ic4} and
\eqref{eq:GB_elem}-\eqref{eq:GB_phase}.

\subsection{Properties of the defining ODEs}
We now show that certain uniformity properties of a family of Gaussian beams parameters imply
corresponding uniformity properties for the ODEs defining the beams.
\begin{lemma}
\label{lemma_unif_flow_1}
Let $\{\SIC_{\gamma}: \gamma \in \Gamma_0\}$ be a set of GB parameters. Assume that there exist
$0 < \cons \leq \Cons$
such that $\cons \abs{\xi_\gamma} \leq \omega_{\gamma} \leq \Cons \abs{\xi_\gamma}$. Let $T>0$.
Then the following estimates hold for
$\gamma, \gamma' \in \Gamma_0$ and $t \in [-T,T]$:
\begin{align}
\label{eq:hypFLW_3}
&|a_{\gamma}-a_{\gamma'}|^2 \leq
|x_\gamma(t)-x_{\gamma'}(t)|^2 + C_T,
\:
|x_\gamma(t)-x_{\gamma'}(t)|^2 \leq |a_{\gamma}-a_{\gamma'}|^2 + C_T,
\\[.71em]
\label{eq:hypFLW_3c}
&|p_\gamma(t)| \asymp |p_\gamma(0)| \asymp 1,
\\[.71em]
\label{eq:hypFLW_3b}
&|\dot{x}_\gamma(t)|, |\dot{p}_\gamma(t)| \lesssim 1,
\\[.71em]
&d\Big(
\big(x_\gamma(t),\omega_{\gamma}p_\gamma(t)\big),\big(x_{\gamma'}(t),\omega_{\gamma'}p_{\gamma'}
(t)\big)\Big)
\asymp
d\Big( \big(a_{\gamma},\xi_\gamma\big),\big(a_{\gamma'},\xi_{\gamma'}\big)\Big),\label{eq:hypFLW_2}
\end{align}
where the constant $C_T$ and the implied constants depend on $T$, $\cons$ and $\Cons$ but
not on the particular pair of parameters $\gamma,\gamma'$.
\end{lemma}
\begin{proof}
The bounds \eqref{eq:hypFLW_3} and \eqref{eq:hypFLW_3c}
follow from the assumptions on the
velocity and Gronwall's
lemma; see the proofs of~\cite[Lemmas 3.1 and 3.3]{Lex3}.
Using the equations for the Hamiltonian flow, the assumptions on $\speed$, and
\eqref{eq:hypFLW_3} we get
\begin{align*}
&\abs{\dot{x}_\gamma(t)} = \abs{c(x(t))} \leq \Cspeed,
\\
&\abs{\dot{p}_\gamma(t)} = \abs{\nabla{c}(x(t))} \abs{p(t)} \lesssim \Cspeed,
\end{align*}
where $\Cspeed$ is a constant that depends only on the velocity $\speed$. This gives
\eqref{eq:hypFLW_3b}.
Finally, the estimate in \eqref{eq:hypFLW_2} is proved in
\cite[Lemma 3.2]{Lex3}.
\end{proof}
\begin{rem}
\label{rem_pt}
In Lemma \ref{lemma_unif_flow_1}, the conclusion $\abs{p_\gamma(t)} \asymp 1$ holds because
the initial condition associated with $\gamma$ in \eqref{eq_ic2} ensures that $\abs{p_\gamma(0)}
\asymp 1$, with the assumption that
$\abs{\xi_\gamma}\asymp \omega_{\gamma}$. In general,
if $(x,p)$ is the flow associated with $H^+$ or $H^-$ with arbitrary initial conditions, it follows
from our assumptions in the velocity that $\abs{p(t)} \asymp \abs{p(0)}$ with constants that
are uniform on any bounded interval of time.
\end{rem}

For more particular initial conditions, the following further properties hold.
\begin{lemma}
\label{lemma_unif_flow_2}
Let $\{\SIC_{\gamma}: \gamma \in \Gamma_0\}$ be a set of GB parameters. Assume that there exist
constants $0<\cons \leq \Cons$
such that for all $\gamma \in \Gamma_0$: $ \norm{\newM_\gamma} \leq \Cons$,
$\Im{\newM_\gamma} \geq \cons \cdot I_d$, and $\cons \abs{\xi_\gamma} \leq \omega_{\gamma} \leq
\Cons
\abs{\xi_\gamma}$. Let $T \geq 0$. Then there exist constants $\cons',\Cons'>0$ - that only depend
on $T$, $\cons$ and $\Cons$ -
such that the following
estimates hold for $t \in [-T,T]$:
\begin{align}
&\norm{M_\gamma(t)} \leq \Cons',
\label{eq:hypFLW_0}
\\
\label{eq:hypFLW_00}
&\Im M_\gamma(t) \geq \cons' \cdot I_d,
\\
&\cons' \abs{\newA_{\gamma}\omega_{\gamma}^{\frac{d}{4}}} \leq \abs{A_\gamma(t)} \leq \Cons'
\abs{\newA_{\gamma}\omega_{\gamma}^{\frac{d}{4}}}
\label{eq:hypFLW_1}.
\end{align}
\end{lemma}
\begin{proof}
Consider the matrix-valued ODE in \eqref{eq_ode2}. The derivatives of the Hamiltonian $H(x,p)$ are
bounded
on any set where $\abs{p}$ is bounded above and below. Since $\abs{p_\gamma(t)}$ is bounded above
and below on $[-T,T]$ by Lemma \ref{lemma_unif_flow_1}, it follows that the coefficients in
\eqref{eq_ode2} are bounded.
In addition, the norm of the initial condition $M_\gamma(0)=\newM_{\gamma}$ is bounded by assumption
- cf.
\eqref{eq_ic3}. Therefore,
\eqref{eq:hypFLW_0} follows by Gronwall's lemma. The estimate in \eqref{eq:hypFLW_1} now follows
from \cite[Lemma
3.1]{Lex3} (which requires $\norm{M(t)}$ to be bounded). Finally, \eqref{eq:hypFLW_00}
is proved in \cite[Lemma 2.56]{Kat}. The statement there is non-quantitative, but the argument gives
the desired
conclusion. See also \cite[Lemma 3.4]{Lex3}.
\end{proof}

\section{Well-spread families of Gaussian Beam parameters}
\label{sec:well_spread}
\subsection{Definitions}
We develop criteria under which a family of Gaussian beam parameters behaves qualitatively like the
standard one, given by
\begin{align*}
\alphast_\gamma = (\sigmast_\gamma, \aast_\gamma, \xist_\gamma, \Aast_{\gamma} , \Mst_\gamma)
= (4^j,2^{-j}\lambda, \pointjk, 2^{\frac{d}{4}}, \ii I_d), \qquad \gamma=(j,k,\lambda).
\end{align*}
Our main goal is to show that when an adequate family of parameters is used as initial values, then
a linear combination of the corresponding Gaussian beams satisfies a suitable Bessel bound and
provides an approximate solution to the wave equation.

\begin{definition}\label{def:wsp}
A \emph{well-spread set of Gaussian beam parameters} is an indexed set
\[\Upsilon \equiv \sett{\SIC_\gamma: \gamma \in \Gamma_0} \subseteq
\Rst_+ \times \Rdst \times (\Rdst \setminus \sett{0})\times (\Rst \setminus \{0\}) \times \bC^{d
\times d},
\qquad \gamma \in \Gamma_0
\]
with $\Gamma_0 \subseteq \Gamma$, such that
\begin{itemize}
\item[(i)] $|\aast_{\gamma}-\aast_{\gamma'}| \lesssim
|a_{\gamma}-a_{\gamma'}| +1,
\qquad \gamma, \gamma' \in \Gamma_0$.
\item[(ii)] $d((a_{\gamma},\xi_{\gamma}),(a_{\gamma'},\xi_{\gamma'})) \gtrsim
d((\aast_{\gamma},\xist_{\gamma}),(\aast_{\gamma'},\xist_{\gamma'})),
\qquad \gamma, \gamma' \in \Gamma_0$.
\item[(iii)]
$\newM_{\gamma} \in \bC^{d\times d}$ is symmetric,
$\norm{\newM_{\gamma}} \lesssim 1$ and $\Im(\newM_{\gamma}) \gtrsim I_d$,
$\qquad \gamma \in \Gamma_0$.
\item[(iv)] $\omega_{\gamma} \asymp 4^j$ and $\abs{\xi_{\gamma}} \asymp \omega_{\gamma}$,
$\qquad \gamma \in \Gamma_0$.
\item[(v)] $|\newA_{\gamma}| \asymp 1$, $\qquad \gamma \in \Gamma_0$.
\end{itemize}
\end{definition}
In the last definition, the symbols $\lesssim$, $\gtrsim$ and $\asymp$ should be interpreted as
asserting the
existence of suitable constants that are uniform within the family $\Upsilon$.

\begin{rem}
For short, we say that $\{\Phi^+_\gamma: \gamma \in \Gamma_0\}$
and $\{\Phi^-_\gamma: \gamma \in \Gamma_0\}$ are \emph{well-spread families of Gaussian beams},
implying the existence
of a corresponding well-spread family of Gaussian beam parameters $\Upsilon \equiv
\sett{\SIC_\gamma: \gamma \in
\Gamma_0}$ that defines the beams.

Similarly, when a certain family of GB parameters $\Upsilon$ is discussed, we may denote the
corresponding beams by
just $\Phi^\pm_\gamma$, without remarking their dependence on the map $\SIC$.
\end{rem}

Before proving the main estimates, we define an adequate notion of vanishing order along a family of
Gaussian beams.
\begin{definition}\label{def_order_wsp}
Given a well-spread set of GB parameters $\Upsilon \equiv \sett{\SIC_\gamma: \gamma \in \Gamma_0}$,
an interval $I
\subseteq \Rst$, and $m \in \Nst_0$,
a family of functions $F \equiv \sett{F_\gamma: \gamma \in \Gamma_0}$ is said to be
$F = \strongo^m(I,\Upsilon)$ if
\begin{itemize}
\item $F_\gamma(t,x)= \sum_{\abs{\eta}={m}} G_{\gamma,\eta}(t,x) (x-x_{\gamma}(t))^\eta$,
for some functions $G_{\gamma,\eta}(t,\cdot) \in \CinfbRd$, for all $t \in I$.
\item $\sup_{\gamma \in \Gamma_0, t \in I}
\norm{\partial^{k} G_{\gamma,\eta}(t,\cdot)}_{L^\infty(\Rdst)} < +\infty$, for all
multi-indices $k$ and $\eta$, with $\abs{\eta}=m$.
\end{itemize}
\end{definition}
Thus, $F = \strongo^m(I,\Upsilon)$ means that each
$F_\gamma(t,\cdot) = O(\abs{x-x_{\gamma}(t)}^m,\Rdst)$
and the corresponding bounds are uniform for $t \in I$ and $\gamma \in \Gamma_0$.

We note that the definition of $\strongo^m(I,\Upsilon)$ involves a vanishing condition at
$x=x_{\gamma}(t)$ and
also a growth condition for $\abs{x-x_\gamma(t)} \gg 1$. As a consequence,
$F=\strongo^{m+1}(I,\Upsilon)$ does not imply
$F=\strongo^m(I,\Upsilon)$.
As a remedy, we introduce the following notion.
\begin{definition}\label{def_order_wsp_geq}
Given a well-spread set of GB parameters $\Upsilon \equiv \sett{\SIC_\gamma: \gamma \in \Gamma_0}$
and an interval $I \subseteq \Rst$, a family of functions $F \equiv \sett{F_\gamma: \gamma \in
\Gamma_0}$ is said to be
$F = \strongop^m(I,\Upsilon)$ if there exists a finite family
$F^1 = \strongo^{m_1}(I,\Upsilon), \ldots, F^n = \strongo^{m_n}(I,\Upsilon)$, with $m_1, \ldots, m_n
\geq m$,
such that $F_\gamma(t,x) = F^1_\gamma(t,x)+\ldots+F^n_\gamma(t,x)$.
\end{definition}
Note that $F=\strongop^{m+1}(I,\Upsilon)$ implies that $F=\strongop^{m}(I,\Upsilon)$.

\subsection{Bessel bounds and vanishing orders}
We prove a Bessel bound for the summation of Gaussian beams with factors vanishing at the spatial
center of the beams.
This extends the result obtained in~\cite[Sec. 3]{Lex3}
from $L^2(\R^d)$ to Sobolev spaces, and to more general sets of initial conditions.
\begin{theorem}
\label{theo_bessel_bounds}
Let $\Upsilon=\sett{\SIC_\gamma: \gamma \in \Gamma_0}$ be a well-spread set of Gaussian beam
parameters and let $F =
\strongop^{m}(I,\Upsilon)$, with $I \subseteq \Rst$ a bounded interval and $s \in [0,1]$. Then
\begin{align*}
\sup_{t\in I}
\Big\|\sum_{\gamma \in \Gamma_0}2^{jm}b_{\gamma}\Phi^\pm_{\gamma}(t,\cdot)F_{\gamma}(t,\cdot)
\Big\|^2_{H^s} \lesssim C_I
\sum_{\gamma \in \Gamma_0}4^{2sj}|b_{\gamma}|^2,
\end{align*}
with $b_\gamma \in \bC$ such that the sum on the right-hand side is finite. (Here, the constant
$C_I$ depends on the interval $I$ and the family $F$.)
\end{theorem}
\begin{proof}
Without loss of generality we may assume that $F = \strongo^{m}(I,\Upsilon)$.
According to Definition~\ref{def_order_wsp},
\[ F_\gamma(t,x)= \sum_{\abs{\eta}={m}} G_{\gamma,\eta}(t,x) (x-x_{\gamma}(t))^\eta \]
for some adequate functions $G_{\gamma,\eta}(t,\cdot) \in \CinfbRd$.

The estimates relevant for the Bessel bounds are developed in greater generality in Appendix
\ref{AppendixB}. Indeed, Lemma
\ref{lemma_GB_mol} implies that
the beams $\{\Phi^\pm_{\gamma}(t,\cdot): \gamma \in \Gamma_0\}$
are sets of \emph{wave-molecules} uniformly for $t \in I$
(see Appendix \ref{AppendixB} for the definitions). Second, the multiplication operators
have symbols
$G_{\gamma,\eta}(t,\cdot)$ that belong to the H\"ormander class
$S^{0}_{1,0}$ uniformly for $t \in I$ and $\gamma \in \Gamma_0$. Hence, using
Lemmas~\ref{lemma_ops} and \ref{lemma_pso_wm} we conclude that
$\{2^{jm} F_{\gamma}(t,\cdot)\Phi^\pm_{\gamma}: \gamma \in \Gamma_0\}$ is a set of wave-molecules
uniformly for $t \in
I$, and, by Lemma~\ref{lemma_bes_mol}, it satisfies the desired Bessel bounds.
\end{proof}
\begin{rem}
\label{rem_bess_zero}
The choice $I= \ptg{0}$ is allowed in Theorem \ref{theo_bessel_bounds} and corresponds to
time-independent
functions $F_\gamma(x) = F_\gamma(0,x)$. Since,
for the standard set of GB parameters,
the beams at time $t=0$ coincide with the higher-scale frame elements
- cf. Observation ~\ref{ob_beam_frame} -  we conclude that, given $ F_{\gamma} =
\strongop^{m}(\{0\},\Upsilon^{\textrm{st}})$,
\begin{align*}
\Big\|\sum_{\gamma \in \Gamma}2^{jm}b_{\gamma}\varphi_{\gamma} F_{\gamma} \Big\|^2_{H^s} \lesssim
\sum_{\gamma \in \Gamma}4^{2js}|b_{\gamma}|^2, \qquad s \in [0,1].
\end{align*}
\end{rem}
\subsection{Uniformity of errors for Taylor expansions}
Most of our arguments rely on Taylor expansions for the functions $x,p,A,M$ used in the definition
of
Gaussian
beams. The following lemma is used to justify that, in such arguments, the error terms can be
bounded uniformly
within a given well-spread family of Gaussian beams.
\begin{lemma}
\label{lemma_error_flow}
Let $\Upsilon=\sett{\SIC_\gamma: \gamma \in \Gamma_0}$ be a well-spread set of Gaussian beam
parameters,
let $T \geq 0$ and $k \geq 0$ be an integer. Then the following quantities:
\begin{align*}
\sup_{\gamma \in \Gamma_0} \sup_{t \in [-T,T]} \abs{\partial^{k+1}_t x_{\gamma}(t)},& \qquad
\sup_{\gamma \in \Gamma_0} \sup_{t \in [-T,T]} \abs{\partial^{k}_t p_{\gamma}(t)},
\\
\sup_{\gamma \in \Gamma_0} \sup_{t \in [-T,T]} \abs{\partial^k_t M_{\gamma}(t)},
\end{align*}
are bounded by a constant that depends on $T$, $k$ and $\Upsilon$.
In addition,
\begin{equation}\label{eq_ampl_der}
	 \partial_t A_{\gamma}(t) = A_{\gamma}(t) G_{\gamma}(t),
\end{equation}
\end{lemma}
with \( \displaystyle
\sup_{\gamma \in \Gamma_0} \sup_{t \in [-T,T]}\abs{\partial^k_t G_{\gamma}}\) bounded by a constant
that depends on $T$, $k$ and $\Upsilon$.
\begin{proof}
Since the derivatives of the velocity $\speed$ are bounded, the derivatives of $H(x,p)$ are bounded
on any set where $\abs{p}$ is bounded above and below. By Lemma \ref{lemma_unif_flow_1},
$\abs{p_\gamma(t)} \asymp 1$ and, therefore,
we conclude that
\begin{align}
\label{eq_knm}
\sup_{t\in[-T,T]} \abs{\partial^n_x \partial^m_p H (x_\gamma(t),p_\gamma(t))} \lesssim
\cnumb{n,m,T} < \infty,
\end{align}
for all multi-indices $n,m$.
Inspecting the definition of the Hamiltonian field $(x_{\gamma},p_{\gamma})$ - cf.~\eqref{eq_ode1},
the claim on $x$ and $p$ follows from \eqref{eq_knm}.

For the matrix $\Mg$, we note that, due to \eqref{eq_knm}, it satisfies a Riccati-type ODE where
the coefficients are
bounded and have all the derivatives bounded. Moreover, the corresponding initial condition is
bounded, as part of
Definition \ref{def:wsp}. Hence, the claim on $M_{\gamma}$ follows from a Gronwall-type argument for
linear systems of ODEs -
see for example \cite{chda76} and \cite[Lemma 2.56]{Kat}.

Finally, inspecting \eqref{eq_ode3}, we see that the claim for the amplitude follows from
\eqref{eq_knm} and the
previous bounds.
\end{proof}

\subsection{Asymptotic solutions}
We now clarify how a linear combination of Gaussian beams with well-spread parameters
approximately solves the wave equation.
These results have been proved in~\cite{Lex3} for standard Gaussian beam parameters, and here are
extended to more general initial conditions.
We first analyze the action of the wave operator and time derivatives on a single beam.
The following lemmas are essentially contained in \cite[Lemmas 3.6 and 3.12]{Lex3}.
\begin{lemma}
\label{lemma_op_asympt}
Let $\Upsilon=\sett{\SIC_\gamma: \gamma \in \Gamma_0}$ be a well-spread set of Gaussian beam
parameters
and $T \geq 0$. Then
\begin{align*}
\left( \partial_{t}^{2} - \speed^{2}(x) \Delta_x\right) \Phi^\pm_{\gamma} =
\left(F_\gamma^{(0)} + 4^j F_\gamma^{(1)} + 4^{2j} F_\gamma^{(3)}\right) \Phi^\pm_{\gamma},
\qquad
\gamma=(j,k,\lambda) \in \Gamma_0,
\end{align*}
with $F^{(m)} = \strongop^m ([-T,T], \Upsilon)$.
\end{lemma}
\begin{proof}
See Section \ref{sec_proof_lemma_op_asympt}.
\end{proof}

\begin{lemma}
\label{lemma_tf}
Let $\Upsilon=\sett{\SIC_\gamma: \gamma \in \Gamma_0}$ be a well-spread set of Gaussian beam
parameters and $T \geq 0$. Then, for $(j,k,\lambda) \in \Gamma_0$,
\begin{align}
\partial_t \Phi^\pm_{\gamma}(t,x)=
\left(F_\gamma^{(0)}(t,x) + 4^j F^{(1)}(t,x) -
\mi \omega_\gamma H^\pm(x^\pm_\gamma(t),p^\pm_\gamma(t))\right) \Phi^\pm_{\gamma}(t,x),
\end{align}
with $F^{(m)} = \strongop^m ([-T,T], \Upsilon)$.
\end{lemma}
\begin{proof}
See Section \ref{sec_proof_lemma_op_asympt}.
\end{proof}

\begin{theorem}\label{th_error_wave}
Let $\Upsilon=\sett{\SIC_\gamma: \gamma \in \Gamma_0}$ be a well-spread set of Gaussian beam
parameters. Then
\begin{align*}
\sup_{t \in [0,T]}
\Big\|
{\left( \partial_{t}^{2} - \speed^{2}(x) \Delta_x\right) \sum_{\gamma \in \Gamma_0} b_\gamma
\Phi^\pm_{\gamma}(t,\cdot)}
\Big\|^2_{L^2(\Rdst)} \leq \CTime
\sum_{\gamma \in \Gamma_0} 4^{j}\abs{b_\gamma}^2,
\end{align*}
with $b_\gamma \in \bC$ such that the sum on the right-hand side is finite.
\end{theorem}
\begin{proof}
We apply Lemma~\ref{lemma_op_asympt}
and the weighted Bessel bounds from Theorem \ref{theo_bessel_bounds}.
\end{proof}

\subsection{Initial Value Problem}\label{Sec2.2}
We review the main result of \cite{Lex3}, that gives a parametrix for the initial value problem
for the wave equation in the whole space. In our formulation, we use a frame of pure Gaussian wave
packets that follows the wave-atom geometry. The estimate is similar to the one given in
\cite{dHHSU} using curvelets. We also refer to the related work of \cite{Walden}.

We consider the following initial value problem
\begin{equation}\label{eq_cauchy_ivp}
  \left\{
\begin{aligned}
   \p_{t}^{2} u(t,x) - c^{2}(x) \Delta_x u(t,x) &{}= 0, \qquad
   (t,x) \in [0,T] \times \R^{d},
\\
   u(0,x) &{}= f(x), \qquad x \in \R^{d},
\\
   \p_{t}u(0,x) &{}= g(x), \qquad x \in \R^{d}.
\end{aligned}\right.
\end{equation}
\subsubsection{Heuristic discussion on the parametrix}
We summarize the construction of a parametrix
using the half wave equations; see \cite{Andersson2015} for a complete treatment. For simplicity we
set $f = 0$. Let us define $\mathit{\Xi} = \sqrt{\speed^2(x) D_x^2}$
and consider a microlocal inverse $\mathit{\Xi}^{-1}$
(which operates on highly oscillatory data). Then
\begin{equation}\label{eq:decoupled_firstorder_u_f}
   u_{\pm} = \tfrac{1}{2} u \pm
         \tfrac{1}{2} \mi \mathit{\Xi}^{-1} \p_{t}u,
\end{equation}
approximately satisfy the two first-order half-wave equations
\begin{equation} \label{eq:decoupled_firstorder}
   P_{\pm}(x,D_x,D_t) u_{\pm} = 0,
\end{equation}
where
\begin{equation}
   P_{\pm}(x,D_x,D_t) = \partial_t \pm \mi \mathit{\Xi}(x,D_x) ,
\quad
   P_{+} P_{-} = \partial_t^2 + \speed^2(x) D_x^2,
\end{equation}
supplemented with the initial conditions
\begin{equation}
\label{eq:hpolIVs}
   u_{\pm} |_{t=0} = g_{\pm},\quad
   g_{\pm} = \pm \tfrac{1}{2} \mi \mathit{\Xi}^{-1} g.
\end{equation}
We let the operators $S_{\pm}(t)$ solve the initial value
problem \eqref{eq:hpolIVs}:
$u_{\pm}(t,x) = (S_{\pm}(t) g_{\pm})(x)$. Then,
an approximate solution of \eqref{eq_cauchy_ivp} is given by
\[
   u(t,x) = ([S_{+}(t) - S_{-}(t)]
          \tfrac{1}{2} \mi \mathit{\Xi}^{-1} g)(x).
\]
In what follows, we construct parametrices representing
$S_{\pm}(t)$ and quantify the approximation errors.
\subsubsection{The Gaussian beam parametrix}
We consider the Gaussian beams associated to the standard set of parameters. As noted in
Observation~\ref{ob_beam_frame}, these match the (higher-scale) frame elements at time $t=0$: i.e.,
$\func_\gamma(x)= \Phi_{\gamma}^{st,\pm}(0,x)$, $\gamma \in \Gamma$. To ease the notation, we drop
the superscript $\mbox{st}$.

We expand the initial conditions
$f \in H^1(\R^d), g\in L^2(\R^d)$ in the frame $\frameset$ as
\begin{equation}\label{eq:initial_data_dec}
   f = \sum_{\gamma \in \Gammafull} f_{\gamma} \vp_{\gamma},
       \qquad
   g = \sum_{\gamma \in \Gammafull} g_{\gamma} \vp_{\gamma}.
\end{equation}

We first discard the lower scale, which only introduces a smooth error (cf. Section
\ref{sec_cutoff}), and we approximate the solution~\eqref{eq:decoupled_firstorder_u_f} up to
principal symbols: \[\mathit{\Xi}^{\rm prin}(x,\xi) = H^+(x,\xi).\]
In analogy to Theorem \ref{th_frame_cutoff}, we
approximate the action of $\mathit{\Xi}^{-1}$ on each beam by calculating the value of its symbol
at the center of the packet. Let
\begin{equation}\label{eq:par}
   U(t,x) = \sum_{\gamma \in \Gamma} f_{\gamma} \tfrac{1}{2}
      \pt{\Phi^{+}_{\gamma}(t,x) + \Phi^{-}_{\gamma}(t,x)
    } + \sum_{\gamma \in \Gamma} g_{\gamma} \tfrac{1}{2}
      \pt{\Psi^{+}_{\gamma}(x,t) + \Psi^{-}_{\gamma}(x,t)},
\end{equation}
where
\[
\Psi^{\pm}_{\gamma}(t,x) =
\pm \mi \cdot
   \Phi^{\pm}_{\gamma}(t,x)
    \cdot     H^+(2^{-j}\lb,2\pi\pointjk)^{-1}
= \pm \mi \cdot \Phi^{\pm}_{\gamma}(t,x) \cdot
c(2^{-j}\lb)^{-1} \abs{2\pi\pointjk}^{-1}.
\]
\subsubsection{Bounds for the Gaussian beam parametrix}
We now state the resulting error estimate.
\begin{theorem}\label{thm:IVP}
Let $u = u(t,x)$ be the solution to the Cauchy initial value problem
\eqref{eq_cauchy_ivp}
with $f \in H^1(\R^d)$ and $g \in L^2(\R^d)$.
Let $U =U(t,x)$ denote the approximate solution given in \eqref{eq:par}. Then
$U \in C^{0}([0,T];H^1(\R^{d})) \cap
C^{1}([0,T];L^2(\R^{d}))$. Moreover $U$ satisfies the error estimate,
\begin{equation}\label{eq:43}
   \| u - U \|_{C^{0}([0,T];H^1(\R^{d})) \cap
        C^{1}([0,T];L^2(\R^{d}))} \leq \CTime \ptg{\norm{f}_{H^{1/2}}+\norm{g}_{H^{-1/2}}}.
\end{equation}
In particular, in the highly oscillatory regime:
$\hat{f}(\xi)=\hat{g}(\xi)=0$, for $\abs{\xi} \leq \xi_{\rm min}$, we obtain
\begin{equation}\label{eq:43b}
   \| u - U \|_{C^{0}([0,T];H^1(\R^{d})) \cap
        C^{1}([0,T];L^2(\R^{d}))} \leq
\CTime \cdot
\xi_{\rm min}^{-1/2} \cdot
        \ptg{\norm{f}_{H^{1}}+\norm{g}_{L^2}}.
\end{equation}
\end{theorem}

Before proving Theorem \ref{thm:IVP} we present some auxiliary estimates.
The following lemma, which is a variant of \cite[Lemma 3.10]{Lex3},
is related to the so-called paraxial
approximation.
\begin{lemma}\label{L3.8}
Let $U(t,x)$ be the parametrix given in \eqref{eq:par}. Then following estimates hold.
\begin{enumerate}
\item[(i)] $\sum_{\gamma}4^{j}\ptg{|f_{\gamma}|^2+ |g_{\gamma}
\cdot H^+(2^{-j}\lb,2\pi\pointjk)^{-1}|^{2}} \lesssim \norm{f}^2_{H^{1/2}} + \norm{g}^2_{H^{-1/2}}$.
\item[(ii)] $\norm{U(0,\cdot) - u(0,\cdot)}_{H^1} \lesssim \norm{f}_{H^{-1}}$.
\item[(iii)] $\|\p_{t} U(0,.) - \p_{t} u(0,.) \|_{L^2}
\lesssim \norm{f}_{H^{1/2}}+\norm{g}_{H^{-1/2}}.$
\end{enumerate}
\end{lemma}
\begin{proof}
Part (i) follows from the norm equivalence in Theorem \ref{T_FrProp}
and the fact that $H^+(2^{-j}\lb,2\pi\pointjk) \asymp \abs{\pointjk} \asymp 4^j$, because the
velocity $\speed$ is bounded above and below.
For part (ii) we use Observation \ref{ob_beam_frame} and note that $U(0,\cdot)=\sum_{\gamma \in
\Gamma} f_\gamma \func_\gamma$ is the high-scale part of the frame expansion of $f$, while
$u(0,\cdot)=f$. Hence, the conclusion follows from \eqref{eq_trunc_error}. The proof of (iii) is
similar, this time using Lemma \ref{lemma_tf}. See \cite[Lemma 3.10]{Lex3} for more details.
\end{proof}

We can now prove the announced approximation bounds.
\begin{proof}[Proof of Theorem \ref{thm:IVP}]
The proof is as in \cite[Theorem 3.2]{Lex3}.
The error function $E(t,x) = U(t,x) - u(t,x)$
solves the problem
\begin{equation*}
\left\{
  \begin{aligned}
   &[\p_{t}^{2} - c^{2}(x) \Delta_x] E(t,x) =
          [\p_{t}^{2} - c^{2}(x) \Delta_x] U(t,x), &\qquad
   (t,x) \in [0,T] \times \R^{d},
\\
   &E(0,x) = U(0,\cdot) - u(0,\cdot), &\qquad x \in \R^d,
\\
   &\p_{t} E(0,x) = \p_{t} U(0,x) - \p_{t} u(0,x),
    &\qquad x \in \R^d.
\end{aligned}\right.
\end{equation*}
We use the energy estimate - cf.  Theorem~\ref{th_Las},
\begin{align*}
   &\| E \|_{C^{0}([0,T];H^{1}(\R^{d})) \cap C^{1}([0,T];\ld)}
   \leq C_T \big(
   \norm{U(0,\cdot) - u(0,\cdot)}_{H^1}+
   \\
   &\qquad
   \| \p_{t} U(0,.) - \p_{t} u(0,.) \|_{L^2}
       + \| [\p_{t}^{2} - c^{2} \Delta_x] U
                            \|_{L^{\infty}([0,T];\ld)} \big).
\end{align*}
The first two terms are suitably bounded by Lemma \ref{L3.8}.
The term involving $[\p_{t}^{2} - c^{2} \Delta_x] U$ is bounded,
due to Theorem~ \ref{th_error_wave}, in terms of the weighted coefficient norm in part (i) of
Lemma~\ref{L3.8}, which is in turn bounded by the desired quantity.
\end{proof}

\section{The Dirichlet problem on the half space}
\label{sec_dir}

\subsection{Setting and assumptions}
We are interested in the following problem. Suppose that
$u: [0,T] \times \Rdplus \to \bC$ is a (weak) solution to:
\begin{eqnarray}
\left\{
\begin{aligned}
&\partial^2_{t} u (t,x) - \speed(x)^2 \Delta_x u (t,x) = 0, &\qquad t\in[0,T], x \in \Rdplus,
\\
&u(0, x)= u_t(0,x) = 0, &\qquad x \in \Rdplus,
\\
&u(t,0,y) = h(t,y), &\qquad t\in[0,T], y \in \Rst^{d-1},
\end{aligned}
\right.
\end{eqnarray}
where $h \in H^1([0,T] \times \Rst^{d-1})$ is called \emph{boundary value}.
We assume that we are able to measure the boundary value $h$ and the goal is to approximate the
corresponding solution $u$.
We now introduce several assumptions.

\subsubsection{Assumptions on the boundary value}
\label{sec_ass_source}
In order for the Dirichlet problem to be well-posed
we need to assume that $h$ satisfies the standard compatibility condition
$h(0,\cdot) \equiv 0$.
In addition, the parametrix that we propose is ultimately based
on oscillatory integrals and the theory of elliptic boundary
value problems, and these techniques require that the bicharacteristic directions
be nowhere tangent to the boundary \cite{nirenberg73}. That is why
we exclude \emph{grazing rays} from the wavefront set of $h$.
Following \cite{MR2944376}, we formulate quantitative versions of these assumptions by replacing the
function $h$ with
a new function $\cuth$ that is the result of applying an adequate pseudodifferential cut-off to $h$.
Recall that $h$ is a function of $(t,x_*) \in \Rst \times \R^{d-1}$. We denote the conjugate
(Fourier) variables
by $(\tau, \xi_*)$, and let $\cuth := \sym(t,x_*,D_t, D_{x_*}) h$, where the symbol
$\sym(t,x_*,\tau,\xi_*) :=
a(t,x_*) b(t,x_*,\tau,\xi_*)$ satisfies the following.

\begin{itemize}
\item[(i)] $a$ is smooth with compact support and there exist $\cconcL, \cconcU \in (0,T)$ such that
$\supp(a) \subseteq [\cconcL, \cconcU] \times \R^{d-1}$.
\item[(ii)] $b$ is a smooth symbol of order $0$, and there is a constant $\consalone>0$ such that
$b$ vanishes on the set of all points $(t,x_*,\tau,\xi_*) \in \Rst \times \R^{d-1} \times \Rst
\times \R^{d-1}$ such that $\abs{(\tau,\xi_*)} \geq 1$ and
\begin{align}
\label{eq_graray_pure}
\frac{|\tau|}{\speed(0,x_*)} - |\xi_*| \leq \consalone \abs{(\tau,\xi_*)}.
\end{align}
(This is possible because the condition in
\eqref{eq_graray_pure} is homogeneous of degree zero on $(\tau,\xi_*)$.)
If such rays are not present in the wavefront set of the boundary value,
the action of the cut-off is not needed.
\end{itemize}
We note that, as a result of the cut-off operation, $\cuth \in H^1([0,T] \times \R^{d-1})$,
$\supp(\cuth) \subseteq [\cconcL,\cconcU]\times K$, for some compact set $K\subset\R^{d-1}$,
and $[\cconcL,\cconcU] \subseteq (0,T)$.

\begin{rem}
The assumptions on the boundary value are quantitative versions of the compatibility and no-grazing
ray conditions.
Indeed, we assume that the observation window $[0,T]$ properly contains the time support of $h$, and
that
there is an absolute lower bound on the grazing angles.
\end{rem}

\subsubsection{Assumptions on the velocity}
We recall that the \emph{velocity}~$\speed$ is assumed to be smooth, positive, bounded from below
and with bounded derivatives of all orders. This ensures that suitable energy estimates
are available for the Dirichlet problem.

\subsubsection{The cone condition}
We assume that for every $\cturnd \in (0,1]$, there exist $\cturn>0$
such that
if $(x(t),p(t))$ is a solution to the Hamiltonian flow
with initial conditions at $t_0 \in [\cconcL ,\cconcU]$ satisfying $x_1(t_0) = 0$
and
$|p_1(t_0)| \geq \cturnd {\abs{p(t_0)}}$ then:
\begin{align}
\label{eq_cone}
|x_1(t)| \geq \cturn |t-t_0|, \qquad t \in [-T,T].
\end{align}
\begin{rem}
The cone condition implies that for all take-off angles at the boundary the corresponding rays do
not return to the
boundary in the
time interval in question, and indeed it is a quantitative version of that statement. See Figure
\ref{fig:cone}.
\end{rem}

\begin{rem}
Since
\[
\abs{\dot{x}_1(t)} = \speed(x(t)) \frac{\abs{p_1(t)}}{\abs{p(t)}}
\asymp \frac{\abs{p_1(t)}}{\abs{p(t)}},
\]
the cone condition holds automatically for $t$ near $t_0$. The content of \eqref{eq_cone}
is the validity of the bound on the whole interval $[-T,T]$.
Moreover, since the Hamiltonian is time independent, this condition can be stated at $t_0=0$ and it
is only about the
size of the interval on which the cone condition holds.
Figure \ref{fig:caustic} shows an example of a velocity satisfying the hypothesis.
\end{rem}

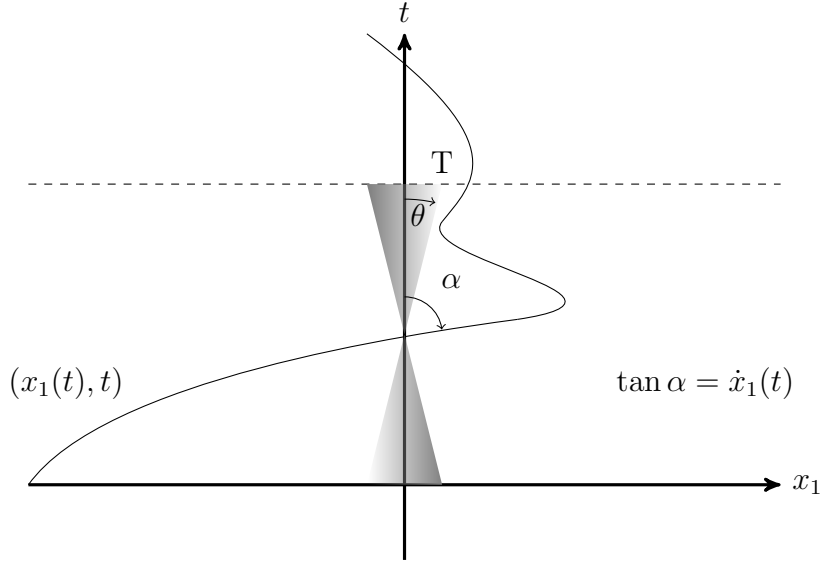
\begin{figure}
\begin{tikzpicture}[
    axis/.style={very thick, ->, >=stealth'},
    important line/.style={thick},
    dashed line/.style={dashed, thin},
    pile/.style={thick, ->, >=stealth', shorten <=2pt, shorten
    >=2pt},
    every node/.style={color=black}
    ]
    \draw[axis] (-5,2)  -- (5,2) node(xline)[right]
        {$x_1$};
    \draw[axis] (0,1) -- (0,8) node(yline)[above] {$t$};
        \draw[dashed] (-5,6) -- (5,6);
          \draw[black]
  (-5,2)
    .. controls (-4,3.4) and (0,4) ..
    (1.5,4.2);
    \draw[black]
  (1.5,4.2)
    .. controls (3.5,4.5) and (.1,5) ..
  (.5,5.5);
    \draw[black]
  (.5,5.5)
    .. controls (.9,6) and (1.5,6.5) ..
  (-.5,8);
  \draw arc[radius=.25cm,start angle=0,end angle=45];
\fill[gray,path fading=west] (-.5,2) -- (0,4) -- (.5,2);
\fill[gray,path fading=east] (-.5,6) -- (0,4) -- (.5,6);
\node at (0.5,6) [above] {$\textrm{T}$};
     \coordinate (A) at (0,4);
     \coordinate (B) at (1.4,10);
    \coordinate (C) at (0,10);

    \tikzAngleOfLine(A)(B){\AngleStart}
    \tikzAngleOfLine(A)(C){\AngleEnd}
    \draw[black,<-] (A)+(\AngleStart:1.8cm) arc (\AngleStart:\AngleEnd:1.8 cm);
 \node at ($(A)+({(\AngleStart+\AngleEnd)/2}:1.6 cm)$) {$\theta$};
 \coordinate (D) at (0,5);
    \coordinate (E) at (1.5,4.2);

    \tikzAngleOfLine(A)(E){\AngleStart}
    \tikzAngleOfLine(A)(D){\AngleEnd}

 \pic [draw, <-, "$\alpha$", angle eccentricity=1.9] {angle = E--A--D};
\node at (-4.5,3) [above] {$(x_1(t),t)$};
\node at (4,3) [above] {$\tan\alpha = \dot{x}_1(t)$};
      \end{tikzpicture}
\caption{The cone condition.}
\label{fig:cone}
\end{figure}

\subsection{Frame expansion of the boundary value}
\label{sec_frame_source}
The recovery method that we introduce in the next sections operates on the
frame expansion of the boundary value
\begin{align}
\label{eq_h}
h(t,x_*) = \sum_{\gamma \in \Gammafull} h_\gamma \varphi_\gamma(t,x_*).
\end{align}
Therefore, we need to show that the assumptions above are reflected by this expansion.
Recall that $\cuth= \eta(t,x_*,D_t,D_{x_*}) h$, where $\eta$ is a zero-order pseudodifferential
symbol. As shown in Section \ref{sec_cutoff}, this operator can be approximately implemented as a
cut-off on the frame coefficients.
More precisely, we first discard to zeroth-scale coefficients, leading to an error bound as in
\eqref{eq_trunc_error}.
Second, we let $\tilde h_\gamma := \eta(2^{-j}\lambda,\pointjk) h_\gamma$, $\Gammah := \{\gamma \in
\Gamma: \tilde h_\gamma \not=0\}$, and set
\begin{align}
\label{eq_tilh}
\tilh = \sum_{\gamma \in \Gammah} \tilde h_\gamma  \varphi_\gamma.
\end{align}
By Theorem \ref{th_frame_cutoff}, we have the following approximation estimate:
\begin{equation}\label{eq:wavefront_h}
\|\cuth - \tilh\|_{H^1} \lesssim \|h\|_{H^{1/2}}.
\end{equation}
We now note some properties of the truncated frame parameters.

\begin{prop}
The set $\Gammah$ satisfies the following.
\begin{itemize}
\item[(i)](\emph{Time concentration and approximate compatibility}). There exist $\cconcL,
\cconcU>0$ such that
for every $\gamma=(j,k,\lambda) \in \Gammah$,
\begin{align}
\label{eq_tcn}
0 < \cconcL \leq 2^{-j} \lambda_1 \leq \cconcU.
\end{align}
\item[(ii)](\emph{Quantitative grazing ray condition}). There exists $\cgraz \in (0,1)$ such that
for every
$\gamma=(j,k,\lambda) \in \Gammah$:
\begin{equation}
	\label{eq_graray}
	\frac{|(\pointjkt)_1|}{\speed(0,2^{-j} \lambda_*)} - |(\pointjkt)_*| \geq \cgraz,
\end{equation}
where the point $\pointjkt$ is defined by \eqref{eq_pjkt}.
\end{itemize}
\end{prop}
\begin{proof}
This follows directly from the properties of the symbol $\eta$. The constants
$\cconcL, \cconcU>0$ are the same as in Section \ref{sec_ass_source}. The constant $\cgraz$ is
related to the constant $\consalone$ from \eqref{eq_graray_pure}. These two numbers are not
exactly the same because the points $\pointjkt$ are not exactly normalized - recall that, however,
$\pointjkt$ is a multiple of $\pointjk$ and $\abs{\pointjkt} \asymp 1$, so a suitable $\cgraz$ can
be found.
\end{proof}

\begin{rem}
The constants $\cconcL,\cconcU, \cgraz$ are given individual notation
for future reference. We remark that the estimates in the rest of the article depend on them,
as well as on the constants in the cone condition.
\end{rem}

\subsubsection{Non-tangential propagation}

Since the velocity $\speed$ is assumed to be bounded from below, the grazing ray condition
\eqref{eq_graray} implies
the following \emph{non-tangential propagation} estimate:
\begin{align}
\label{eq_fpr}
\abs{(\pointjkt)_1} \geq \cforw,  \qquad \gamma=(j,k,\lambda) \in \Gammah,
\end{align}
where $\cforw = \cgraz\cspeed >0 $, and $\cspeed$  -
cf.~\eqref{eq_speed_bounds} - is the minimum value of the velocity $\speed$. In particular
$(\pointjkt)_1 \not =0$. In
what follows, the sign of $(\pointjkt)_1$ plays an important role, and it is convenient to define:
\begin{align}
\label{eq_gammapm}
\Gammah^+ := \sett{\gamma \in \Gammah: (\pointjkt)_1 < 0},
\qquad
\Gammah^- := \sett{\gamma \in \Gammah: (\pointjkt)_1 > 0}.
\end{align}
(The motivation for this notation will be clear later.)

\section{Spatio-temporal analysis of the beams near the boundary}
\label{sec_spation-temporal}
We consider a well-spread family of Gaussian beams $\{\Phi^+_{\gamma}: \gamma \in \Gamma_0\}$
or $\{\Phi^-_{\gamma}: \gamma \in \Gamma_0\}$,
and times $t=t_\gamma$, $\gamma \in \Gamma_0$, at which
the centers of the corresponding beams intersect the boundary $x_1 = 0$,
i.e. $\xoneg (t_{\gamma}) = 0$ - for short, we say that the \emph{beams intersect the boundary} at
those times.
We focus on the case in which $t_\gamma$ belongs to the interval $[\cconcL,\cconcU]$,
where the boundary value is active. We assume that every beam in the family does intersect the
boundary at a suitable time; for
a more general family of beams, the analysis of this section applies by considering a subset
of $\Gamma_0$.

We analyze the restriction of the beams to $x_1=0$, treating the remaining variables $(t,x_*)$ as a
\emph{joint spatial
variable}. We aim to approximately describe the restricted beam $\Phi^\pm_{\gamma}(t,0,x_*)$
as a Gaussian beam with a fixed evolution time. We first identify the spatial center of
$\Phi^\pm_{\gamma}(t,0,x_*)$
and then describe the resulting functions in two different regimes: near the center and away from
it.
The assumption that the family of beams under study is well-spread allows us to obtain a uniform
control on the
approximation errors. This is essential for the applications in the following
sections.

To ease the notation we focus on one of the two modes ($+/-$) and
remove this choice from the notation.
Hence,
most of the symbols below should be supplemented with a $+/-$ superscript.
(In particular, $H$ stands for either $H^+$ of $H^-$.)

\subsection{Local analysis of a beam when it intersects the boundary}
Before stating the estimates, we introduce some auxiliary functions defined in terms of the
functions in \eqref{eq_ode1}, \eqref{eq_ode2} and \eqref{eq_ode3}.

Let $\gamma \in \Gamma_0$ and consider the matrix $\widetilde{M}_{\gamma} \in \bC^{d\times d}$
defined by
\begin{equation}\label{sys:M}
\left\{
\begin{array}{ll}
  \displaystyle \widetilde{M}_{\gamma,11}&{}=   \displaystyle\dot{x}_{\gamma}(t_{\gamma}) \cdot
\Mg(t_{\gamma}) \dot{x}_{\gamma}(t_{\gamma})
-\dot{p}_{\gamma}(t_{\gamma})\cdot\dot{x}_{\gamma}(t_{\gamma}),
\\
   \widetilde{M}_{\gamma,1k}  &{}= \displaystyle
\dot{p}_{\gamma,k}(t_{\gamma})
   -\sum_{n=1}^d \pt{\Mg(t_{\gamma})}_{kn}\dot{x}_{\gamma,n}(t_{\gamma}),
             \qquad k=2,\dots, d,
\\
  \displaystyle  \widetilde{M}_{\gamma,kl} &{}= \pt{\Mg(t_{\gamma})}_{kl},
  \qquad k,l=2,\ldots, d.
  \end{array}
\right.
\end{equation}
In more compact notation,
\begin{align}
\widetilde{M}_{\gamma}
=
\begin{bmatrix}
\widetilde{M}_{\gamma,11}& \widetilde{M}_{\gamma,1*}^t \\
\widetilde{M}_{\gamma,1*}& \pt{M_{\gamma}(t_\gamma)
}_{**}\end{bmatrix},
\end{align}
where
\begin{align}
\widetilde{M}_{\gamma,1*} =
\left(
\dot{p}_{\gamma}(t_{\gamma})
-M_\gamma(t_\gamma)
\dot{x}_{\gamma}(t_{\gamma})
\right)_* \in \bC^{(d-1)\times 1},
\end{align}
and $\pt{M(t_\gamma)}_{**} \in \bC^{(d-1)\times(d-1)}$ is the matrix obtained from $M(t_\gamma)$ by
eliminating the first row
and column.
Let us also consider the following constants and functions:
\begin{equation}
\label{eq_taua}
\taugam = -H(x_\gamma(t_{\gamma}),p_\gamma(t_{\gamma})) = -p_\gamma(t_\gamma) \cdot
\dot{x}_\gamma(t_\gamma),
\end{equation}
\begin{equation}\label{eq:Lin_Term}
	\begin{split}
	 L_{\gamma}(t,x_{*}) =\big(\taugam,\pastg (t_{\gamma})\big)\cdot ((t,x_*) -
(t_{\gamma},\xastg(t_{\gamma}))),
	\end{split}
\end{equation}
\begin{equation}\label{eq:Quad_Term}
\begin{split}
Q_{\gamma}(t,x_*)   &{}=
         \half \pt{(t,x_*) - (t_{\gamma},\xastg)}\cdot \widetilde{M}(t_{\gamma}) \pt{(t,x_*) -
(t_{\gamma},\xastg)}.
\end{split}
\end{equation}

We can now describe a Gaussian beam intersecting the boundary.
\begin{lemma}\label{L:vip}
Let $\Upsilon \equiv \sett{\SIC_\gamma: \gamma \in \Gamma_0}$ be
a well-spread set of GB parameters. For $\gamma \in \Gamma_0$, let $t_{\gamma} \in
[\cconcL,\cconcU]$ be such that $x_{\gamma,1}(t_{\gamma}) = 0$ (i.e. the center of the corresponding
beam
$\Phi_{\gamma}=\Phi^\pm_{\gamma}$ intersects the boundary $x_1 = 0$ at a time $t=t_{\gamma}$ when
the
boundary value is active). Let us write $x_\gamma(t_{\gamma})=(0,x_{\gamma,*}(t_{\gamma}) )$. Then
the restriction of $\Phi_{\gamma}$ to $x_1=0$
admits the following asymptotic expansion around $(t_{\gamma}, x_{\gamma,*}(t_{\gamma}))$:
\begin{equation}
\label{eq_taylor_beam}
\Phi_{\gamma}(t,0,x_*) = A_{\gamma}(t_{\gamma})\left(1+
R_{\gamma}(t)\right)e^{i\omega_{\gamma} \{
L_{\gamma}(t,x_{*}) + Q_{\gamma}(t,x_{*}) + \Theta_{\gamma}(t,x_*)\}},
\:
(t,x_*) \in\Rst^d_T,
\end{equation}
with
\[
	R_{\gamma}(t) = r_{\gamma}(t)(t-t_{\gamma}),\quad
	\Theta_{\gamma}(t,x_*) = \sum_{|\mu| =3}g_{\gamma, \mu}(t)\pt{(t,x_*)-(t_{\gamma},
x_{\gamma,*}(t_{\gamma}))}^\mu
\]
and
$r_\gamma, g_{\gamma, \mu} \in C^\infty_b([-T,T])$,
uniformly on $\gamma$. More precisely, for every $k \geq 0$, the error factors satisfy:
\begin{align}
\label{eq_bound_error_terms}
\sup_{\gamma \in \Gamma_0} \sup_{t \in [-T,T]}
\abs{\partial^k_t g_{\gamma,\mu}(t)},\quad
\sup_{\gamma \in \Gamma_0} \sup_{t \in [-T,T]}
\abs{\partial^k_t r_\gamma(t)} < +\infty.
\end{align}
\end{lemma}
\begin{proof}
We analyze the Gaussian beam
 \begin{align*}
\Phi_{\gamma}(t,x) = A_\gamma(t) e^{i \omega_{\gamma} \theta_\gamma(t,x)},
\end{align*}
by Taylor expanding the amplitude and phase.

\step{1}{The amplitude}. Using the bounds in Lemma \ref{lemma_error_flow} - specifically
\eqref{eq_ampl_der} -
and Lemma \ref{lemma_unif_flow_2} - which is applicable uniformly for $\gamma \in \Gamma_0$ - we see
that the function
$B_\gamma(t) := A_\gamma(t) / A_\gamma(t_\gamma)$ is bounded and has bounded derivatives on
$[-T,T]$,
uniformly for $\gamma \in
\Gamma_0$. Since $B_\gamma(t_\gamma) = 1$, we can write:
$B_\gamma(t)=1+r_\gamma(t)(t-t_\gamma)$, with $r_\gamma$ as in \eqref{eq_bound_error_terms}.
Therefore,
\begin{align}
\label{eq_Aalp}
A_\gamma(t) = A_\gamma(t_\gamma)(1+ r_{\gamma}(t))(t-t_\gamma).
\end{align}
In order to establish \eqref{eq_taylor_beam}, it remains to inspect the exponential factor.

\step{2}{Expansion of the characteristic flow}.
We first expand the characteristics as
\begin{align}
\label{eq_ch1}
    \xxg (t) &{}=\xxg(t_{\gamma})+\dot{x}_{\gamma}(t_{\gamma}) (t-t_{\gamma}) + \half
\ddot{x}_{\gamma}(t_{\gamma}) (t-t_{\gamma})^2
                 + R_{x,\gamma}(t)(t-t_{\gamma})^3,
\\
\label{eq_ch2}
   p_\gamma (t) &{} = \pg(t_{\gamma})
        + \dot{p}_{\gamma}(t_{\gamma})(t-t_{\gamma})
                 +R_{p,\gamma}(t)(t-t_{\gamma})^2,
\end{align}
where $R_{x,\gamma},R_{p,\gamma}\in C^\infty_b([-T,T])$, and the corresponding bounds are uniform
for $\gamma \in
\Gamma_0$, as shown in Lemma~\ref{lemma_error_flow}.

We now focus on the phase function
\begin{equation}
\label{eq_phase_pr}
\theta_\gamma (x,t) = p_\gamma(t)\cdot (x-x_\gamma(t))
+ \tfrac{1}{2} (x-x_\gamma(t))\cdot M_\gamma(t) (x-x_\gamma(t)).
\end{equation}

\step{3}{The linear part of the phase}.
The linear part of $\theta_\gamma$ is
\begin{align}\label{eq:P_tay}
   \pg(t)\cdot\pt{x -  \xxg(t)}
        &{}= \pg(t_{\gamma})\cdot \pt{x -\xxg(t_{\gamma})-\dot{x}_{\gamma}(t_{\gamma})(t-t_{\gamma})
- \half
\ddot{x}_{\gamma}(t_{\gamma}) (t-t_{\gamma})^2}
\\
     &{}\quad + \dot{p}_{\gamma}(t_{\gamma})(t-t_{\gamma})\cdot
\pt{x-\xxg(t_{\gamma})-\dot{x}_{\gamma}(t-t_{\gamma})} + \Theta_{\gamma}\nonumber
\\
    &{}= \pg(t_{\gamma})\cdot \pt{x -\xxg(t_{\gamma})} -\pg(t_{\gamma})\dot{x}_{\gamma}(t_{\gamma})
(t-t_{\gamma})\nonumber
    \\
    &{}\quad - \half(t - t_{\gamma})^2\pt{\pg(t_{\gamma})\cdot\ddot{x}_{\gamma}(t_{\gamma}) +
\dot{p}_{\gamma}(t_{\gamma})\cdot  \dot{x}_{\gamma}(t_{\gamma}) } \label{eq_der_ham}
     \\
      &{}\quad - \half(t - t_{\gamma})^2\pt{\dot{p}_{\gamma}(t_{\gamma})\cdot
\dot{x}_{\gamma}(t_{\gamma}) \nonumber
}
     \\
     &{}\quad + (t-t_{\gamma})\dot{p}_{\gamma}(t_{\gamma})\cdot \pt{x-\xxg(t_{\gamma})}
+\Theta_{\gamma},\nonumber
\end{align}
where $\Theta_\gamma$ denotes a function of the form:
\[
\Theta_{\gamma}= \sum_{|\mu|=3}R_{(x,p), \gamma, \mu}(t)\pt{(t,x)-(t_{\gamma},
x_{\gamma}(t_{\gamma}))}^\mu,\qquad R_{(x,p), \gamma,\mu}\in C^{\infty}_b(\ptq{-T,T}).
\]
Indeed, note that the error factors $R_{(x,p), \gamma, \mu}(t)$ involve the error factors
$R_{x,\gamma},
R_{p,\gamma}$
from \eqref{eq_ch1} and \eqref{eq_ch2} multiplied by $\dot{x}_\gamma(t_{\gamma})$, $\pg(t_\gamma)$
and
similar quantities involving higher order derivatives, which are uniformly bounded by Lemma
\ref{lemma_error_flow}.

Since
\[
H(\xxg(t),\pg(t)) = \pg(t)\cdot\dot{x}_{\gamma}(t)
\]
is constant on $t$, it follows that
\[
\pt{\pg(t_{\gamma})\cdot\ddot{x}_{\gamma}(t_{\gamma}) + \dot{p}_{\gamma}(t_{\gamma})\cdot
\dot{x}_{\gamma}(t_{\gamma})
}
= \partial_t {H}(\xxg(t),\pg(t))_{|t=t_\gamma} = 0,
\]
and the term in ~\eqref{eq_der_ham} vanishes.
Thus, \eqref{eq:P_tay} reads
\begin{align*}
   \pg(t)\cdot\pt{x -  \xxg(t)}
        &{}= \pg(t_{\gamma})\cdot \pt{x -\xxg(t_{\gamma})} -\pg(t_{\gamma})\cdot
\dot{x}_{\gamma}(t_{\gamma})
(t-t_{\gamma})\nonumber
     \\
        &{}\quad - \half(t -
t_{\gamma})^2\pt{\dot{p}_{\gamma}(t_{\gamma})\cdot\dot{x}_{\gamma}(t_{\gamma}) }
     + (t-t_{\gamma})\dot{p}_{\gamma}(t_{\gamma})\cdot \pt{x-\xxg(t_{\gamma})} +\Theta_{\gamma}.
\end{align*}
Specializing on the boundary we obtain
that for $x_1 = \xoneg(t_{\gamma})=0$,
\begin{equation}
\label{eq_L}
\begin{split}
&\pg(t)\cdot\pt{x -  \xxg(t)}
= L_{\gamma}(t,x_{*}) -
\half\pt{\dot{p}_{\gamma}(t_{\gamma})\cdot
\dot{x}_{\gamma}(t_{\gamma}) }(t - t_{\gamma})^2
     \\
     &{}\quad + (t-t_{\gamma})\dot{p}_{\gamma,*}(t_{\gamma})\cdot \pt{x_*-x_{\gamma,*}(t_{\gamma})}
+\Theta_{\gamma}\vert_{x_1=0},
\end{split}
\end{equation}
where $L(t,x_*)$ is defined by \eqref{eq:Lin_Term}.

\step{4}{The quadratic part of the phase}.
We linearize the Riccati matrix $\Mg$ as
\[
\Mg(t) = \Mg(t_{\gamma}) + N_\gamma(t)(t-t_\gamma),
\]
with $N_\gamma \in C^{\infty}_b(\ptq{-T,T})$ uniformly on $\gamma$, due to
Lemma~\ref{lemma_error_flow}.

Using \eqref{eq_ch1},
we can expand the quadratic part of $\theta_\gamma$ as
\begin{align*}
&\pt{x-\xxg(t)}\cdot \Mg(t) \pt{x-\xxg(t)}\hfill
\\
\qquad &= \pt{x-\xxg(t)}\cdot \Mg(t_\gamma) \pt{x-\xxg(t)}
+ {\Theta}_{\gamma}
\\
\qquad &=\pt{x-\xxg(t_{\gamma})-\dot{x}_{\gamma}(t_{\gamma})(t-t_{\gamma})}\cdot \Mg(t_{\gamma})
\pt{x-\xxg(t_{\gamma})-\dot{x}_{\gamma}(t_{\gamma})(t-t_{\gamma})}
+{\Theta}_{\gamma},
\end{align*}
where, in each line, $\Theta_{\gamma}$ denotes a function of the form:
\[
\Theta_{\gamma}= \sum_{|\mu|=3}R_{(x,M), \gamma,\mu}(t)\pt{(t,x)-(t_{\gamma},
x_{\gamma}(t_{\gamma}))}^\mu,\qquad
\mbox{with }
R_{(x,M), \gamma,\mu}\in C^{\infty}_b(\ptq{-T,T}),
\]
uniformly on $\gamma$.

Specializing on the boundary we obtain
that for $x_1 = \xoneg(t_{\gamma})=0$,
\begin{equation}
\label{eq_q}
\begin{split}
&\pt{x-x_{\gamma}(t)}\cdot \Mg(t) \pt{x-x_{\gamma}(t)}
=
\dot{x}_{\gamma}(t_{\gamma})\cdot
\Mg(t_{\gamma})\dot{x}_{\gamma}(t_{\gamma})
(t-t_{\gamma})^2
\\
&\qquad+
\pt{x_*-x_{\gamma,*}(t_{\gamma})}\cdot \Mg(t_{\gamma})_{**} \pt{x_*-x_{\gamma,*}(t_{\gamma})}
\\
&\qquad
-2 \left(\dot{x}_{\gamma}(t_{\gamma}) \cdot \Mg(t_\gamma) \right)_*
\pt{x_*-x_{\gamma,*}(t_{\gamma})}(t-t_\gamma)
+{\Theta}_{\gamma}\vert_{x_1=0}.
\end{split}
\end{equation}
\step{5}{Collecting terms}.
Finally, we combine \eqref{eq_L} and \eqref{eq_q}, noting that the quadratic terms
add up precisely to $Q_\gamma(t,x_*)$, as defined by \eqref{eq:Quad_Term}.
\end{proof}
\subsection{Global analysis of a beam when it intersects the boundary}
We now describe the global profile of the restriction of the beam to the boundary $\{x_1=0\}$.
We aim to show that the restricted beams $\Phi_{\gamma}(t,0,x_*)$ display a Gaussian profile
\emph{in the
$(t,x_*)$ variables}. To this end, we consider an additional assumption on the way in which the
original beams
intersect the boundary. We say that a family
of beams $\{\Phi_{\gamma}: \gamma \in \Gamma_0\}$ \emph{intersects the boundary in a uniformly
transversal
fashion} at times $\{t_\gamma: \gamma \in \Gamma_0\}$ if: (i) $x_{\gamma,1}(t_\gamma)=0$, for all
$\gamma \in \Gamma_0$,
and (ii) there exists a constant $\cnumb{1} \in (0,1)$ such that
\begin{align}
\label{eq_transv_meet}
\abs{p_{\gamma,1}(t_{\gamma})} \geq \cnumb{1} \abs{p_{\gamma}(t_{\gamma})}, \qquad \gamma \in
\Gamma_0.
\end{align}

The following lemma provides the desired description.
\begin{lemma}
\label{L:vip2}
Let $\Upsilon \equiv \sett{\SIC_\gamma: \gamma \in \Gamma_0}$ be
a well-spread set of GB parameters. For $\gamma \in \Gamma_0$, let $t_{\gamma} \in
[\cconcL,\cconcU]$ be such that $x_{\gamma,1}(t_{\gamma}) = 0$ (i.e. the center of the corresponding
beam
$\Phi_{\gamma}=\Phi^\pm_{\gamma}$ intersects the boundary $x_1 = 0$ at a time $t=t_{\gamma}$
when the boundary value is active). Assume also that the beams intersect the boundary in a uniformly
transversal fashion;
i.e., there exist a constant $\cnumb{1} \in (0,1)$ such that \eqref{eq_transv_meet} holds.

Let us write $x_{\gamma}(t_{\gamma})=(0,x_{\gamma,*}(t_{\gamma}))$.
Then the restriction of $\Phi_{\gamma}$ to $x_1=0$
admits the following description: for $\gamma=(j,k,\lambda) \in \Gamma_0$,
\begin{align*}
\Phi_{\gamma}(t,0,x_*) =
A_{\gamma}(t_{\gamma})
\exp \left(
i 4^j \left[ L_{\gamma}(t,x_*) +
i \consnew \left( (t-t_\gamma)^2 + \abs{x_*-x_{\gamma, *}(t_{\gamma})}^2\right)
\right]\right)
\cdot R_\gamma(t,x_*),
\end{align*}
where $L_{\gamma}$ is given by \eqref{eq:Lin_Term},
$\consnew>0$ is a constant - that depends only on the family $\Upsilon$ and the constant $\cnumb{1}$
-
and
$R_\gamma \in C^\infty_b([-T,T] \times \{x_*: \abs{x_*-x_{\gamma, *}} \geq 1\})$,
uniformly on $\gamma$. More precisely, for
all multi-indices $k, \alpha$, the error factor satisfies:
\begin{align*}
\sup_{\gamma \in \Gamma_0} \sup_{t \in [-T,T]} \sup_{\abs{x_*-x_{\gamma, *}} \geq 1}
\abs{\partial_t^k \partial^\alpha_{x_*} R_\gamma(t,x_*)} < +\infty.
\end{align*}
\end{lemma}
\begin{proof}

As before, all estimates in this proof are to be understood as being uniform for
$\gamma=(j,k,\lambda) \in \Gamma_0$,
and
to be dependent on $T$.

\step{1}{Linearization of the centers}. Using Lemma \ref{lemma_error_flow} we write
\begin{align*}
\xxg(t) = \xxg(t_\gamma)+ (t-t_\gamma)y_\gamma(t),
\end{align*}
where $y_{\gamma,i} \in C^\infty_b([-T,T])$. Since
$x_{\gamma,i}(t_\gamma)=0$, the transversality assumption and the
\emph{cone condition} in \eqref{eq_cone}, imply that
\begin{align}
\label{eq_bb}
\abs{y_{\gamma,1}(t)} \geq \cturn, \qquad t\in[-T,T],
\end{align}
for some constant $\cturn>0$.

\step{2}{The linear part of the phase}. We show that

\begin{align}
\label{eq_L1}
{p_\gamma(t)} \cdot \left( (0,x_*)-x_\gamma(t) \right) = L_\gamma(t,x_*) + E_\gamma^1(t,x_*)
\end{align}
where $E_\gamma^1$ satisfies the following: given multi-indices $k,\alpha$:
\begin{align}
\label{eq_E}
\sup_{t \in [-T,T]} \abs{\partial^k_t \partial^\alpha_{x_*} E^1(t,x_*)}
\leq C_{k,m} \left(1+\abs{x_*-x_{\gamma,*}(t_\gamma)}\right).
\end{align}
(Recall that this estimate is understood to be also uniform on $\gamma$, but dependent on $T$.)

We expand the left-hand side of \eqref{eq_L1}. We use $E$ to denote a function satisfying
a bound similar to \eqref{eq_E}. The meaning of $E$ changes from line to line, and the assertions
are verified
using Step 1 and Lemmas \ref{lemma_unif_flow_1}, \ref{lemma_unif_flow_2}
and \ref{lemma_error_flow}. With this understanding:
\begin{align*}
&{p_\gamma(t)}\cdot\left((0,x_*)-x_\gamma(t) \right)
=
{p_\gamma(t)} \cdot \left((0,x_*)-x_\gamma(t_\gamma) \right)+ E(t,x_*)
\\
&\qquad=
{p_\gamma(t_\gamma)}\cdot \left( (0,x_*)-x_\gamma(t_\gamma) \right)+ E(t,x_*)
\\
&\qquad=
{p_{\gamma,*}(t_\gamma)}\cdot \left(x_*-x_{\gamma,*}(t_\gamma) \right)+ E(t,x_*)
\\
&\qquad=
L_\gamma(t,x_*) -\tau_\gamma(t-t_\gamma) + E(t,x_*)
\\
&\qquad=
L_\gamma(t,x_*) + E(t,x_*),
\end{align*}
as desired.

\step{3}{The quadratic part of the phase}.
Consider the quadratic term:
\begin{align}
{\newC}(t,x_*) = \tfrac{1}{2} \left[ ((0,x_*) - x_\gamma(t)) \cdot M_\gamma(t)
((0,x_*) - x_\gamma(t)) \right].
\end{align}
Let us show that ${\newC}(t,x_*)={\newC}^1(t,x_*)+{\newC}^2(t,x_*)+{\newC}^3(t,x_*)$, with
\begin{align}
\label{eq_def_C1}
{\newC}^1(t,x_*) &=
\consnew \left( (t-t_\gamma)^2 + \abs{x_*-x_{\gamma, *}(t_{\gamma})}^2\right),
\\
\label{eq_def_C2}
{\newC}^2(t,x_*) &=
\left( t-t_\gamma , x_*-x_{\gamma,*}(t_\gamma) \right)\cdot
N^2_\gamma(t) \left( t-t_\gamma, x_*-x_{\gamma,*}(t_\gamma) \right),
\\
{\newC}^3(t,x_*) &= \left((0,x_*)-x_\gamma(t)\right) \cdot N^3_\gamma(t)
\left((0,x_*)-x_\gamma(t)\right),
\end{align}
where $\consnew>0$, $N_\gamma^2(t), N_\gamma^3(t) \in \bC^{d\times d}$
are symmetric, $\Im N_\gamma^2(t), \Im N_\gamma^3(t) \geq \consnew'I_d$, $\consnew'>0$, and
for each $k \geq 0$, there is a constant $C_k$ such that
\begin{align}
\label{eq_Ns}
\sup_{t \in [-T,T]} \abs{\partial^k_t N_\gamma^2(t)},
\sup_{t \in [-T,T]} \abs{\partial^k_t N_\gamma^3(t)}
\leq C_k <+\infty.
\end{align}

By definition of well-spread set of GB parameters and Lemma \ref{lemma_error_flow},
there exists a constant $\varepsilon>0$ (independent of $\gamma$ and $t$)
such that $\Im(M_\gamma(t)) \geq 2\varepsilon I_d$.
We let $N^3_\gamma(t) = \tfrac{1}{2}M_\gamma(t) - \tfrac{\varepsilon}{2}\mi I_d$.
This defines ${\newC}^3$. Note that $\Im(N^3_\gamma(t)) \geq \consnew' I_d$,
with $\consnew'=\tfrac{\varepsilon}{2}$ and that ${\newC}(t,x_*) - {\newC}^3(t,x_*) =
\tfrac{\varepsilon}{2}\mi
\abs{(0,x_*)-x_\gamma(t)}^2$. Expanding that expression and using $x_{\gamma,1}(t_\gamma)=0$,
 we see that
\begin{align*}
&\tfrac{2}{\varepsilon \mi}({\newC}(t,x_*) - {\newC}^3(t,x_*))=
\abs{(0,x_*)-x_\gamma(t_\gamma)-(t-t_\gamma)y_{\gamma}(t)}^2
\\
&\qquad
= \abs{y_\gamma(t)}^2(t-t_\gamma)^2 - 2 (t-t_{\gamma})
y_{\gamma,*}(t) \cdot (x_*-x_{\gamma,*}(t_\gamma))+\abs{x_*-x_{\gamma,*}(t_\gamma)}^2
\\
&\qquad
= (t-t_\gamma,x_*-x_{\gamma,*}(t_\gamma)) \cdot\widetilde{N}_\gamma(t)
(t-t_\gamma,x_*-x_{\gamma,*}(t_\gamma)),
\end{align*}
where:
\begin{align*}
\begin{bmatrix}
\widetilde{N}_{\gamma, 11}(t) & \widetilde{N}_{\gamma, 1*}(t)^t \\
\widetilde{N}_{\gamma, 1*}(t) & \widetilde{N}_{\gamma, **}(t)
\end{bmatrix}
=
\begin{bmatrix}
\ \abs{y_\gamma(t)}^2 & -y_{\gamma,*}(t)^t \\
-y_{\gamma,*}(t) & I_{d-1}
\end{bmatrix}.
\end{align*}
We now invoke Lemma \ref{lemma_LA} (in the appendix)
to see that $\widetilde{N_\gamma}(t) \gtrsim I_d$, for $t \in [-T,T]$. The quantity to estimate is:
$\abs{y_\gamma(t)}^2 - \abs{y_{\gamma,*}(t)}^2=\abs{y_{\gamma,1}(t)}^2$, which is bounded below, by \eqref{eq_bb}.

Hence, we can let $N^2(t) = \frac{\varepsilon}{2} \mi \widetilde{N_\gamma}(t) - \consnew \mi I_d$
with $\consnew>0$ such that $\Im N^2(t) \geq \consnew I_d$.
We now let ${\newC}^1(t,x_*)$ and ${\newC}^2(t,x_*)$ be defined by \eqref{eq_def_C1} and
\eqref{eq_def_C2},
respectively.
Hence, ${\newC}(t,x_*)={\newC}^1(t,x_*)+{\newC}^2(t,x_*)+{\newC}^3(t,x_*)$ as desired. Finally the
bounds in \eqref{eq_Ns}
follow from Lemma
\ref{lemma_error_flow}.

\step{4}{Bounds for the error factor}.
We write the error factor as $R(t,x_*)=R^1(t,x_*) \cdot R^2(t,x_*)$, with
\begin{align*}
R^1(t,x_*) &=  \frac{A(t)}{A(t_\gamma)},
\\
R^2(t,x_*) &= \exp(i 4^j(E^1(t,x_*) + {\newC}^2(t,x_*)+{\newC}^3(t,x_*)))
\end{align*}
By Lemmas \ref{lemma_unif_flow_2}, and \ref{lemma_error_flow}, it follows that
$R^1 \in C^\infty_b(\Rst^d_T)$ - cf. Step 1 in the proof of Lemma \ref{L:vip}.
We focus now on $R^2$. Let $k,m$ be multi-indices. Using the bounds in Steps 2 and 3
(and the fact that $E_\gamma^1$ is real) we conclude that
there exists a number $n=n(k,m)$ and a constant $C_n=C_{k,m}$ such that for $(t,x_*) \in \Rdst_T$:
\begin{align}
\label{eq_B1}
\begin{split}
\abs{\partial^k_t \partial^m_{x_*} R^2(t,x_*)}
&\leq C_n 4^{jn}\left(1+\abs{x_*-x_{\gamma, *}(t_\gamma)}^2\right)^n \cdot
\\
&\qquad \exp(-4^j \left[ \Im(\newC^2(t,x_*))+\Im(\newC^3(t,x_*))\right] ).
\end{split}
\end{align}
Using \eqref{eq_E} and
the fact that $\Im N_\gamma^2, \Im N_\gamma^3 \geq \consnew' I_d$ we obtain:
\begin{align*}
\Im({\newC}^2(t,x_*))+\Im({\newC}^3(t,x_*))
&\geq \consnew'\left(\abs{(t-t_\gamma)}^2 + \abs{x_*-x_{\gamma,*}(t_\gamma)}^2 +
\abs{(0,x_*)-x_\gamma(t)}^2\right)
\\
&\geq \consnew'\abs{x_*-x_{\gamma,*}(t_\gamma)}^2.
\end{align*}
Combining this with \eqref{eq_B1} we obtain:
\begin{align*}
\abs{\partial^k_t \partial^m_{x_*} R^2(t,x_*)} \lesssim
4^{jn}\left(1+\abs{x_*-x_{\gamma, *}(t_\gamma)}^2\right)^n \exp(-4^j \consnew'
\abs{x_*-x_{\gamma,*}(t_\gamma)}^2),
\end{align*}
where the implied constant depends on $k$ and $m$. Finally, for $\abs{x_*-x_{\gamma,*}(t_\gamma)}
\geq 1$ we can
estimate:
\begin{align*}
\abs{\partial^k_t \partial^m_{x_*} R^2(t,x_*)} &\lesssim
4^{jn}\abs{x_*-x_{\gamma, *}(t_\gamma)}^{2n} \exp(-4^j \consnew' \abs{x_*-x_{\gamma,*}(t_\gamma)}^2)
\\
&\lesssim \left(4^j \consnew'\abs{x_*-x_{\gamma, *}(t_\gamma)}^2 \right)^n
\exp(-4^j \consnew' \abs{x_*-x_{\gamma,*}(t_\gamma)}^2) \leq n!
\end{align*}
This completes the proof.
\end{proof}

\section{Packet-beam matching}
\label{sec_match}

The goal of this section is to select, for each index $\gamma=(j,k,\lambda) \in \Gammah$,
a corresponding tuple of initial conditions $\gammastar \in
\Rst_+ \times \Rdst \times (\Rdst \setminus \sett{0})\times (\Rst \setminus \{0\}) \times \bC^{d
\times d}$ and an adequate mode, $+$ \emph{or} $-$, giving initial conditions for a Gaussian beam,
in such a way that
\begin{align}
\label{eq_informal_match}
\Phih_{\gamma}(t,0,x_*) \approx \varphi_\gamma(t,x_*).
\end{align}
We use the analysis of Section \ref{sec_spation-temporal} as a guide.
We first judiciously select a time instant $t_\gamma$ and construct the beam $\Phih_{\gamma}$ in
such a way that it
intersects the boundary $\{x_1=0\}$ at that time. To this end, we design the beam $\Phih_{\gamma}$
by matching the
approximate description of $\Phih_{\gamma}(t,0,x_*)$, provided by Lemma \ref{L:vip}, to
the target frame element $\varphi_\gamma(t,x_*)$. This approximate description is
useful for $t$ near the boundary
meeting time $t_\gamma$. Hence, the construction involves \emph{back-propagating} the profile of the
beam under
construction by means of the ODEs in \eqref{eq_ode1}, \eqref{eq_ode2}, \eqref{eq_ode3}, from time
$t=t_\gamma$ to time
$t=0$. The matching procedure is depicted in Figure~\ref{F_match}.

Afterwards, we analyze the family of parameters that results from this procedure, and
prove that they are well-spread in the sense of Section \ref{sec:well_spread}, and that they
intersect the boundary at
the
prescribed times in a uniformly transversal fashion - cf. Section \ref{sec_spation-temporal}. With
this information,
the approximate description of Section \ref{sec_spation-temporal}, that was initially used as a
guide, is rigorously
justified and can be used to quantify \eqref{eq_informal_match}.

\begin{figure}
\begin{center}
  \begin{tikzpicture}
    [x={(0.866cm,-0.5cm)}, y={(0.866cm,0.5cm)}, z={(0cm,1cm)}, scale=1.0,
        >=stealth,
    inner sep=0pt, outer sep=2pt,
    axis/.style={thick,->},
    wave/.style={thick,color=#1,smooth},
    polaroid/.style={fill=black!60!white, opacity=0.3},]

    \colorlet{darkgreen}{green!50!black}
    \colorlet{lightgreen}{green!80!black}
    \colorlet{darkred}{red!50!black}
    \colorlet{lightred}{red!80!black}

    \coordinate (O) at (0, 0, 0);
    \draw[axis] (O) -- +(3, 0,   0) node [below] {$x_*$};
    \draw[axis] (O) -- +(0,  3, 0) node [below] {$t$};
    \draw[axis] (O) -- +(0,  0,   2) node [left] {$x_1$};

    \draw[thick] (0,-2,0) -- (O);

  \draw[wave=blue, variable=\y,samples at={-2.5,-2.25,...,2.5}]
    plot ({(\y/2)^2+1},\y,\y/4)node[anchor=north]{};

  \draw[wave=blue, variable=\y,samples at={-2.5,-2.25,...,2.5}]
    plot ({(\y/2)^2+1},\y,\y/4)node[anchor=north]{};
  \draw[wave=blue, variable=\y,samples at={-2.5,-2.25,...,2.5},->]
    plot ({(\y/2)^2+1},\y,\y/4)node[anchor=north]{};

    \node at (3,0,0.5) [above] {$\varphi_{\gamma}(t,x_*)$};
    \node at (-0.45,0,0) [left] {$t_{\gamma}$};
    \node at (3.5,-.35,-.7) [left] {$\Phi_{\gamma}(t,x_1,x_*)$};
    \draw [very thick,darkgreen,rotate around={0:(1,0,0)}] (1,0,0) circle ({.6});
          \foreach\i in {0,0.1,...,.5} {
        \fill[opacity=\i,blue,rotate around={90:(1.6,2.65)}] (1.6,2.65) ellipse ({0.5-0.5*\i} and
{1-1*\i});
      }
      \fill[black] (1.6,2.65) circle (.5mm);
          \foreach\i in {0,0.1,...,.5} {
        \fill[opacity=\i,blue,rotate around={90:(2.4,-2.4)}] (2.4,-2.4) ellipse ({0.5-0.5*\i} and
{1-1*\i});
      }
      \fill[black] (2.4,-2.4) circle (.5mm);

      \foreach\i in {0,0.1,...,0.5} {
        \fill[opacity=\i,darkgreen,rotate around={315:(1,0,0)}] (1,0,0) ellipse ({0.5-0.5*\i} and
{1-1*\i});  }
      \fill[black] (1,0,0) circle (.5mm);

    \end{tikzpicture}

    \end{center}
\caption{A beam intersects the boundary $x_1 = 0$ at $t = t_\gamma$ moving along the projected
bicharacteristic
$x_{\gamma}(t)$. The frame element - represented by a circle - is matched to a beam
- represented by a filled ellipse - which has an approximate Gaussian profile in $(t,x_*)$.}
\label{F_match}
\end{figure}
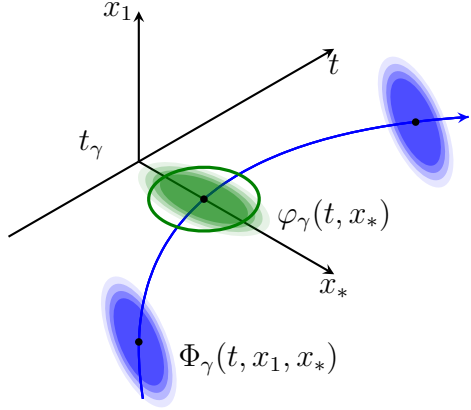

\subsection{Back-propagating beams}
\label{sec_back}

Given a frame element $\func_\gamma$, with $\gamma \in \Gammah$,
we look for a mode, $+$ or $-$, and a tuple of Gaussian beam parameters
\begin{align*}
\gammastar=(\omega_{\gamma}^{h},a_{\gamma}^{h}, \xi_{\gamma}^{h},\newA_{\gamma}^{h},
\newM_{\gamma}^{h}),
\end{align*} such that \eqref{eq_informal_match} holds.
We write explicitly
\begin{align*}
\func_\gamma(t,x_*)
=2^{\left(j+1/2\right)d/2} \cdot \exp
\Big[
\mi 4^j &\big(
\pointjkt \left(t-2^{-j}\lambda_1, x_*-2^{-j}\lambda_*\right)
\\
&\quad
+\mi \pi \left( (t-2^{-j}\lambda_1)^2+ \abs{x_*-2^{-j}\lambda_*}^2 \right)
\big)\Big]
\end{align*}
and compare this expression to \eqref{eq_taylor_beam}. We want
to construct a beam such that:
\begin{align}
\label{eq_des1}
&A_\gamma(t_\gamma)=2^{\left(j+1/2\right)d/2},
\\
\label{eq_des2}
&\omega_\gamma L_\gamma(t,x_*) =
4^j \pointjkt \left(t-2^{-j}\lambda_1, x_*-2^{-j}\lambda_*\right),
\\
\label{eq_des3}
&\omega_\gamma Q_\gamma(t,x_*)
=
\mi \pi 4^j \left( (t-2^{-j}\lambda_1)^2+ \abs{x_*-2^{-j}\lambda_*}^2\right).
\end{align}

\step{1}{Choice of mode and scale}. Recall that by the non-tangential propagation estimate - cf.
\eqref{eq_fpr} - $(\pointjk)_1 \not= 0$.
Let $\sign := \textrm{sign}((\pointjk)_1) \in \{-1,1\}$.
Note that in the asymptotic expansion in \eqref{eq_taylor_beam}, the first component of the linear
part of the
phase is given by \eqref{eq_taua}, which is negative for a $+$ beam and positive for a $-$ one.
Motivated by this fact,
if $\sign=-1$ we construct a $+$ mode, while if $\sign=1$ we construct a $-$ mode. Second, we
choose the scale parameter as $\omega_\gamma^{\pm,h} = 4^j$.
Having made these choices, we ease the notation dropping the superscripts $h, \pm$.

\step{2}{Definition of the boundary intersection time}. We first define the time instant
\begin{align}
\label{eq_def_tg}
\tgam = 2^{-j}\lambda_1.
\end{align}
The center of the Gaussian beam under construction is to intersect the boundary $\{x_1=0\}$ at time
$\tgam$. Note that, due to \eqref{eq_tcn},
\begin{align}
\label{eq_tgamma_in}
t_\gamma \in [\cconcL,\cconcU],
\end{align}
and the constants $\cconcL,\cconcU$ depend only on the boundary value $h$, but not on $\gamma$.

In the following steps, we define functions $(x(t),p(t),M(t),A(t))$ as solutions of the ODEs in
\eqref{eq_ode1}, \eqref{eq_ode2} and \eqref{eq_ode3} by specifying adequate initial conditions at
time $t=\tgam$.
Later we define $\gammastar$ by inspecting $(x(t),p(t),M(t),A(t))$ at time $t=0$.
To this end, we use the description of a beam given in Lemma \ref{L:vip}. We first aim to
match the function
$L_\gamma(t,x_*)$ in \eqref{eq:Lin_Term} to the linear part of the phase in \eqref{eq:GB_phase}.

\step{3}{Definition of $(x(t),p(t))$}. Let $(x,p): \Rst \to \Rst^{2d}$
be the solution of the Hamiltonian flow, cf.~\eqref{eq_ode1},
with initial condition at $t=\tgam$ described as follows. For $x$ we simply set:
\begin{align}\label{eq:initx}
x|_{t=\tgam}=(0,2^{-j}\lambda_*).
\end{align}
This agrees with our intention that the Gaussian beam under construction
$\Phi_{\gamma}$ intersect the boundary at time $\tgam$.

With these choices,
\begin{align}
  \label{eq_match_1}
  (t_\gamma,x_*(t_\gamma)) = 2^{-j}\lambda.
\end{align}

For $p$ we need to specify:
\begin{align*}
p|_{t=\tgam}=((p|_{t=\tgam})_1, (p|_{t=\tgam})_*).
\end{align*}
We first define $(p|_{t=\tgam})_*$ by
\begin{align}\label{eq:initpst}
(p|_{t=\tgam})_* = (\pointjkt)_*,
\end{align}
where $\pointjkt$ is given by \eqref{eq_pjkt}.
Second, we define $(p|_{t=\tgam})_1$ as
\begin{align}\label{eq:q}
(p|_{t=\tgam})_1 =
(-\sign) \cdot \sqrt{\frac{(\pointjkt)_1^2}
{\speed(x|_{t=\tgam})^2} - \abs{(\pointjkt)_*}^2}.
\end{align}
Note that $(p|_{t=\tgam})_1$ is well-defined because of
the grazing ray condition.
Indeed, by \eqref{eq_graray},
\begin{align}
\label{eq_p1big}
\frac{\abs{(\pointjkt)_1}^2}
{\speed(x(\tgam))^2}
\geq \left(\cgraz+\abs{(\pointjkt)_*}\right)^2
\geq \cgraz^2+\abs{(\pointjkt)_*}^2.
\end{align}
In addition,
\begin{align}
\label{eq:inad}
\abs{p|_{t=\tgam}}^2 = \abs{(p|_{t=\tgam})_1}^2 + \abs{(p|_{t=\tgam})_*}^2
= \frac{(\pointjkt)_1^2}
{\speed(x|_{t=\tgam})^2}.
\end{align}

With these choices, since $\sign$ has a sign opposite to the mode of the beam under construction,
\begin{align}
\label{eq_ptg}
\tau_\gamma=
-H(x|_{t=\tgam},p|_{t=\tgam})
=\sign\cdot\speed(x(t_\gamma)) \abs{p|_{t=\tgam}} = \sign\cdot\abs{(\pointjkt)_1} = (\pointjkt)_1.
\end{align}
Consequently
\begin{align}
  \label{eq_match_2}
(\tau_\gamma, (p|_{t=\tgam})_*) = \pointjkt,
\end{align}
and therefore the linear part of the phase of the boundary restriction of the beam under
construction - as a
function of $(t,x_*)$ and
according to the approximate description in Lemma \ref{L:vip} - coincides with the linear part of
the phase of
$\varphi_\gamma(t,x_*)$.
Moreover, we note the following.

\begin{claim}
\label{claim_deriv}
The flow $(x(t),p(t))=(x_\gamma(t),p_\gamma(t))$ defined in Step 3 satisfies:
\begin{align}
\label{eq_lowbdx1}
&\dot{x}_{\gamma,1}(\tgam) > 0, \mbox{ and }
\dot{x}_{\gamma,1}(\tgam) \gtrsim 1,
\\
\label{eq_lowp1}
&\abs{\poneg(\tgam)} \gtrsim \abs{p_{\gamma}(\tgam)},
\end{align}
where the implied constants are uniform for $\gamma \in \Gammah$.
\end{claim}
\begin{proof}
From \eqref{eq:inad} we see that $\abs{p_\gamma(t_\gamma)} \lesssim 1$.
In addition, \eqref{eq:q} and \eqref{eq_p1big}
imply that $\abs{\poneg(\tgam)} \gtrsim 1$, so the claim in \eqref{eq_lowp1} follows.
For \eqref{eq_lowbdx1}, we use one of Hamilton's equations
\begin{align*}
\dot{x}_{\gamma,1}(\tgam) =
\partial_{p_1}{H}(\tgam)=(-\sign)\cdot
\speed\pt{x_{\gamma}(\tgam)}
\frac{\poneg(\tgam)}{|p_{\gamma}(\tgam)|}.
\end{align*}
Inspecting the sign of \eqref{eq:q} we see that
$\dot{x}_{\gamma,1}(\tgam)>0$. In addition, by the assumptions on the velocity
\eqref{eq_speed_bounds} and \eqref{eq_lowp1},
\begin{align*}
\abs{\dot{x}_{\gamma,1}(\tgam)} &
\geq \cspeed \frac{\abs{\poneg(\tgam)}}{|p_{\gamma}(\tgam)|}
\gtrsim 1.
\end{align*}
\end{proof}

\step{4}{Definition of $M(\tgam)$}. Let
\begin{equation}\label{eq:mtilde}
   \widetilde{M}_{\gamma}= 2\pi \mi I_d,
\end{equation} and let $M(\tgam) \in \bC^{d\times d}$ be the unique symmetric matrix that solves the
following system of
equations:
\begin{equation}\label{sys:M_gamma}
\left\{
\begin{array}{ll}
\begin{array}{ll}
  \displaystyle \widetilde{M}_{\gamma,11} &{}=  \displaystyle \dot{x}_{\gamma}(t_{\gamma}) \cdot
M(t_{\gamma}) \dot{x}_{\gamma}(t_{\gamma})- \dot{p}_{\gamma}(t_{\gamma})\cdot
\dot{x}_{\gamma}(t_{\gamma}) ,\vspace{2mm}
\\
   \widetilde{M}_{\gamma,1k}  &{}= \dot{p}_{\gamma,k}(t_{\gamma})
   -\sum_{n=1}^d \pt{M(t_{\gamma})}_{kn}\dot{x}_{\gamma,n}(t_{\gamma}),
             \qquad k=2,\dots d,\vspace{2mm}
\\
  \displaystyle  \widetilde{M}_{\gamma,kl} &{}=\pt{M(t_{\gamma})}_{kl},
  \qquad k,l=2,\ldots d.
  \end{array}
  \end{array}
\right.
\end{equation}
We now check that $M(\tgam)$ is indeed well-defined.

\begin{claim}
\label{Cl_ric}
The system \eqref{sys:M_gamma} has a unique symmetric solution $M(\tgam)$. Moreover,
there exist constants $\cnumb{1}, \cnumb{2} >0$ - independent of $\gamma$ - such that
$\norm{M(\tgam)} \leq \cnumb{1}$ and $\Im M(\tgam) \geq \cnumb{2} \cdot I_d$.
\end{claim}
We postpone the proof of the claim to Section \ref{sec_proof_of_Cl_ric}, so as not to interrupt the
flow of the
construction.

\step{5}{Definition of $M(t)$}. We let $M(t)$ be the solution of \eqref{eq_ode2} with initial
condition at time $t=\tgam$ given by the matrix $M(t_\gamma)$ from Step 4.
Due to Claim \ref{Cl_ric}, this is a valid initial condition - cf. Section \ref{sec_sets}.

\step{6}{Definition of $A(t)$}. Let $A(t)$ be the solution to \eqref{eq_ode3} with initial
condition:
\begin{equation}
\label{eq_initA}
A(\tgam) =  2^{(j+1/2)\frac{d}{2}}.
\end{equation}

\step{7}{Definition of $\gammastarpl$ and $\gammastarmi$}.
We recall the decomposition $\Gammah = \Gammah^+ \cup \Gammah^-$ in \eqref{eq_gammapm} and
define two sets of GB parameters
\begin{align*}
\sett{\gammastarpl: \gamma \in \Gammah^+}, \qquad \sett{\gammastarmi: \gamma \in \Gammah^-},
\end{align*}
with $\gammastarplmi = (\omega_\gamma^{\pm,h},a_\gamma^{\pm,h},
\xi_\gamma^{\pm,h},\newA_\gamma^{\pm,h}, \newM_\gamma^{\pm,h})$
in
the following way:
\begin{equation}
\label{eq_starparameter}
\left\{
\begin{aligned}
&\omega_\gamma^{\pm,h} = 4^j,
\\
&a_\gamma^{\pm,h} = x^\pm(0),
\\
&\xi_\gamma^{\pm,h} = \frac{1}{2\pi}4^j p^\pm(0),
\\
&\newA_\gamma^{\pm,h} = 4^{-j\frac{d}{4}}A^\pm(0),
\\
&\newM_\gamma^{\pm,h} = \frac{1}{2\pi} M^\pm(0).
\end{aligned}
\right.
\end{equation}

The values for $\gammastarplmi$ are
chosen so that if we define the functions $(x^\pm(t),p^\pm(t),M^\pm(t), A^\pm(t))$
by imposing initial conditions at time $t=0$ as described in Section \ref{sec_sets}, they
will satisfy \eqref{eq:initx}, \eqref{eq_match_1}, \eqref{eq_match_2} and \eqref{eq_initA}
at time $t=t_\gamma$. As a result of the construction, \eqref{eq_des1}, \eqref{eq_des2},
\eqref{eq_des3} are satisfied.

\begin{rem}
The function $\gammastarpl$ is defined only on $\Gammah^+$. According to the conventions
in Section \ref{sec_sets}, it would be possible to associate with such map both a family of $+$
beams
and a family of $-$ beams. However, we are only interested in the corresponding family of $+$
beams, because, as we show below, these satisfy the approximation property in
\eqref{eq_informal_match}. A similar
remark applies to $\gammastarmi$.
\end{rem}

\begin{rem}
The choice of sign in \eqref{eq:q} is instrumental to construct a parametrix for the Dirichlet
problem on the right-half space (see Theorem \ref{th_centers_ws} below). For the left-half space,
the opposite sign should be used in \eqref{eq:q}, leading to a different sign in
Claim \ref{claim_deriv}.
\end{rem}

\subsection{Analysis of the back-propagated parameters}

We now analyze the properties of the previous construction. We first state the following fundamental
property.

\begin{theorem}[Well-spreadness]
\label{th_match_ws}
Each of the two families of back-propagated parameters constructed in Section \ref{sec_back},
$\Upsstarp = \sett{\gammastarpl : \gamma \in \Gammah^+}$,
$\Upsstarm = \sett{\gammastarmi : \gamma \in \Gammah^-}$
is a well-spread set of Gaussian beam parameters.
\end{theorem}
The proof of Theorem \ref{th_match_ws} is quite technical and we postpone it to Section
\ref{sec_pr_th_match_ws}. We now analyze the fine properties of the matching procedure.

\begin{theorem}[Transversal boundary intersection]
\label{th_transveral_ws}
Each family of beams $\{\Phihpm_\gamma: \gamma \in \Gammah^\pm\}$ intersects the boundary
$\{x_1=0\}$ at times
$\{t_\gamma: \gamma \in \Gammah^\pm\}$ - given by \eqref{eq_def_tg} - in a uniformly transversal
fashion.
\end{theorem}
\begin{proof}
By \eqref{eq:initx}, $x^h_{\gamma,1}(t_\gamma)=0$. The uniform transversality at the boundary
intersection is proved
in Claim \ref{claim_deriv} - see \eqref{eq_lowp1}.
\end{proof}

\begin{theorem}[Rightwards propagation]
\label{th_centers_ws}
The spatial centers of the beams $\{\Phihpm_\gamma: \gamma \in \Gammah^\pm\}$
are uniformly away from the right-half plane $\Rdplus$
at time $t=0$. More precisely, there exists a constant $\epsilon>0$
such that, for all $\gamma \in \Gamma^\pm_h$,
\begin{align}
\label{eq:distance_from_bd}
x^{h,\pm}_{\gamma,1}(0) \leq -\epsilon.
\end{align}
\end{theorem}
\begin{proof}
By Theorem \ref{th_transveral_ws}, $x_{\gamma,1}(t_\gamma)=0$
and $\abs{p_{\gamma,1}(t_{\gamma})} \gtrsim \abs{p_{\gamma}(t_{\gamma})}$. In addition,
$t_\gamma=2^{-j} \lambda_1 \in
\suppsource \subseteq [0,T]$ by the \emph{approximate compatibility condition}~\eqref{eq_tcn}.
Let us write
\begin{align*}
x^{h,\pm}_{\gamma}(t)=x^{h,\pm}_{\gamma}(t_\gamma) +
(t-t_\gamma) b(t),
\end{align*}
with $b$ smooth. The \emph{cone
condition}~\eqref{eq_cone} implies that
$\abs{b_1(t)} \gtrsim 1$, for $t \in [0,T]$. In addition,
$b_1(t_\gamma)=\dot{x}^{h,\pm}_{\gamma,1}(t_\gamma) >0$ by
Claim \ref{claim_deriv}. Hence, $b_1>0$ on $[0,T]$
and, moreover, $b_1(t) \gtrsim 1$ for all $t \in [0,T]$. Second,
the \emph{approximate compatibility condition} \eqref{eq_tcn} implies that $t_\gamma=2^{-j}
\lambda_1 \gtrsim 1$. Therefore,
\begin{align*}
-{x}^{h,\pm}_{\gamma,1}(0)=t_\gamma \cdot b_1(t) \gtrsim 1,
\end{align*}
as claimed.
\end{proof}

\begin{theorem}[Beams match frame elements on the boundary]
\label{th_match_estimates}
When restricted to the boundary $\{x_1=0\}$, the beams $\{\Phihpm_\gamma: \gamma \in \Gammah^\pm\}$
match the frame elements in the following sense. Let $\eta^1 \in C^\infty(\Rst)$ be compactly
supported. Let $\eta^2 \in C^\infty(\Rst^{d-1})$
be a smooth function supported on $B_2(0)$ that is $\equiv 1$ on $B_1(0)$, and let
$\eta^2_\gamma(x_*) = \eta(x_*-x_{\gamma,*}(t_\gamma))$.

$\bullet$ \emph{Local description}:
\begin{align}
\left(\Phihpm_{\gamma}(t,0,x_*) - \varphi_\gamma(t,x_*)\right)
\cdot \eta^1(t)\cdot \eta^2_\gamma(x_*) =
\left( 4^j \cdot R^1_\gamma(t,x_*) + R^2_\gamma(t,x_*) \right)\cdot
\varphi_\gamma(t,x_*),
\end{align}
with $\gamma=(j,k,\lambda) \in \Gammah^\pm$, $R^1 = \strongop^3(\sett{0},\Upsstarpm)$ and
$R^2 = \strongop^1(\sett{0},\Upsstarpm)$.

$\bullet$ \emph{Global description}:
\begin{align}
\left(\Phihpm_{\gamma}(t,0,x_*) - \varphi_\gamma(t,x_*)\right)
\cdot \eta^1(t)\cdot (1-\eta^2_\gamma(x_*)) = \newPhi^\pm_{\gamma}(t,0,x_*) \cdot
R^3_{\gamma}(t,x_*),
\end{align}
with $\newUpsilon^{\pm} \equiv \{\gammaprime: \gamma \in \Gamma^\pm_h\}$ well-spread sets of GB
parameters,
$\newPhi^\pm_\gamma$ the corresponding beams and
$R^3 = \strongop^1(\sett{0},\newUpsilon^{\pm})$.
\end{theorem}
\begin{rem}
We stress that here the time variable $t$ is not considered as an evolution variable; rather
$(t,x_*)$ functions as a
spatial variable. In accordance, $\sett{0}$ is the time-evolution set in the $\strongop$ notation.
\end{rem}

\begin{proof}[Proof of Theorem \ref{th_match_estimates}]
We invoke Lemmas ~\ref{L:vip} and \ref{L:vip2}. The corresponding hypothesis are satisfied, thanks
to
Theorems~\ref{th_match_ws} and \ref{th_transveral_ws}. We use the notation
$\gammastarplmi=(\omega_\gamma^{\pm,h},
a_\gamma^{\pm,h},
\xi_\gamma^{\pm,h},
\newA_\gamma^{\pm,h},
\newM_\gamma^{\pm,h})$.

For the \emph{local description}, due to Theorem \ref{th_match_ws}, we can invoke Lemma
\ref{L:vip}.
We substitute the values of the beam parameters defined in Section~\ref{sec_back}
into \eqref{eq_taylor_beam} - cf. \eqref{eq_des1}, \eqref{eq_des2}, \eqref{eq_des3}
and obtain:
\begin{equation*}
\Phihpm_{\gamma}(t,0,x_*) =
\varphi_\gamma(t,x_*)
\left(1+R_{\gamma}(t)\right)e^{i\cdot4^j \cdot
 \Theta_{\gamma}(t,x_*)},
\qquad
(t,x_*) \in\Rst^d_T,
\end{equation*}
with $R_\gamma$ and $\Theta_\gamma$ as in Lemma \ref{L:vip}. Second, we note that
\begin{align*}
  e^{i\cdot4^j \cdot   \Theta_{\gamma}(t,x_*)}
  \cdot \eta^1(t)\cdot \eta^2_\gamma(x_*)
  = 1 + 4^j \cdot R'_\gamma(t,x_*),
\end{align*}
with $R' = \strongop^3(\sett{0},\Upsstarpm)$, and the conclusion follows.

For the \emph{global description}, with the notation of Lemma \ref{L:vip2},
  \[
  \Phihpm_{\gamma}(t,0,x_*) =
  A_{\gamma}(t_{\gamma})
  \exp \left[
  i 4^j (L_{\gamma}(t,x_*) +
  i \consnew \left( (t-t_\gamma)^2 + \abs{x_*-x_{\gamma, *}(t_{\gamma})}^2\right))
  \right] \cdot R_\gamma(t,x_*).
  \]
  Substituting the values of the parameters defined in Section~\ref{sec_match} - cf.
\eqref{eq_match_1} and
\eqref{eq_match_2} - we obtain
  \[
  \Phihpm_{\gamma}(t,0,x_*) =
  2^{(j+1/2)\frac{d}{2}} e^{2\pi i \pt{(t,x_*) - 2^{-j}\lambda}\pointjk
  - \consnew 4^j\abs{(t,x_*) - 2^{-j}\lambda}^2}\cdot
R_\gamma(t,x_*).
  \]
We let $R^3_\gamma(t,x_*) := \eta^1(t) (1-\eta^2_\gamma(x_*)) R_\gamma(t,x_*)$ and
  \[
  \gammaprime^\pm = (4^j,2^{-j}\lambda, \pointjk, 2^{\frac{d}{4}}, \ii \tfrac{\consnew}{\pi} I_d),
\qquad
\gamma=(j,k,\lambda) \in
\Gammah^\pm.
  \]
It is straightforward to verify that this defines a well-spread set of GB parameters. Indeed,
for $\gamma \in \Gammah^\pm$, the tuple $\gammaprime^\pm$
is very similar to the standard one $\alphast_\gamma$, defined in \eqref{eq_standard}: the only
difference is that, in
the new set, the standard matrix element $\newM_\gamma=\ii I_d$ is replaced by $\ii
\tfrac{\consnew}{\pi}
I_d$, with $\consnew>0$ a constant.
\end{proof}

\section{Parametrix estimates for the Dirichlet problem}
\label{sec_main}

Finally, we derive the parametrix for the boundary Dirichlet problem and give suitable estimates.

\begin{theorem}\label{th:boundary}
With the assumptions and notation from Section \ref{sec_dir},
let $u: [0,T] \times \Rdplus \to \bC$ be the (weak) solution
to the problem:
\begin{equation}
\label{eq_problem}
\left\{
\begin{aligned}
&\partial^2_{t} u (t,x) - \speed(x)^2 \Delta_x u (t,x) = 0, &\qquad t\in[0,T], x \in \Rdplus,
\\
&u(0, x)= u_t(0,x) = 0, &\qquad x \in \Rdplus,
\\
&u(t,0,y) = \cuth(t,y), &\qquad t\in[0,T], y \in \Rst^{d-1}.
\end{aligned}
\right.
\end{equation}
Let $\tilh = \sum_{\gamma \in \Gammah} \tilde h_\gamma  \varphi_\gamma$
be the truncated frame expansion of $h$ defined in Section \ref{eq_tilh} and consider the GB
parameters
$\gammastarplmi$ constructed in Section \ref{sec_match}. Let $\tilde{u}$ be defined as
\begin{align}\label{eq_solution_bvp}
\tilde{u} = \sum_{\gamma \in \Gammah^+} \tilde h_\gamma \Phihp_{\gamma}
+ \sum_{\gamma \in \Gammah^-} \tilde h_\gamma \Phihm_{\gamma}.
\end{align}
Then
\begin{align*}
\norm{\tilde{u} - u}_{C^0([0,T],H^1(\Rdplus)) \cap C^1([0,T],L^2(\Rdplus))}
\leq \CTime  \norm{h}_{H^{1/2}(\Rdst)}.
\end{align*}
In particular, in the highly oscillatory regime:
$\hat{h}(\xi)=0$ for $\abs{\xi} \leq \xi_{\rm min}$, we obtain
\begin{align*}
\norm{\tilde{u} - u}_{C^0([0,T],H^1(\Rdplus)) \cap C^1([0,T],L^2(\Rdplus))}
\leq \CTime \cdot \xi_{\rm min}^{-1/2 } \cdot \norm{h}_{H^{1}(\Rdst)}.
\end{align*}
\end{theorem}

\begin{rem}
The problem in \eqref{eq_problem} is well-posed because
$\cuth$ satisfies the compatibility condition $\cuth(0,\cdot) \equiv 0$,
cf. Section \ref{sec_dir}.
\end{rem}

The strategy to prove Theorem \ref{th:boundary} is similar to the one for Theorem \ref{thm:IVP}. We
show that
the proposed GB solution approximately solves the boundary-value problem and then conclude, by means
of energy
estimates, that it must suitably approximate the true solution. The results from Section
\ref{sec:well_spread},
together with the analysis of the back-propagated parameters in Section \ref{sec_match}, imply that
the wave operator
approximately annihilates the GB solution. In the next section, we show that the other conditions of
the
boundary-value problem are also approximately satisfied.

\subsection{Preliminary steps}
As a first step towards the proof of Theorem \ref{th:boundary}, we show that the approximate
solution in
\eqref{eq_solution_bvp} satisfies zero boundary conditions, up to the error of the parametrix. More
precisely, we have
the following lemma.

\begin{lemma}[Asymptotic vanishing of the initial conditions]
\label{lem_prep1}
Under the hypothesis of Theorem~\ref{th:boundary}, consider the approximate solution
defined in \eqref{eq_solution_bvp}. Then
\begin{align*}
\norm{\tilde{u}|_{t=0}}_{H^1(\Rdplus)},
\norm{\partial_t \tilde{u}|_{t=0}}_{L^2(\Rdplus)}
\lesssim \CTime \|h\|_{H^{1/2}(\Rdst)}.
\end{align*}
\end{lemma}
\begin{proof}
We use the short notation $u = \sum_{\gamma\in\Gammah} \tilde h_\gamma \Phih_{\gamma}$, with the
understanding that
$\Phih_\gamma$ is a $+$ mode for $\gamma \in \Gammah^+$ and a $-$ mode for $\gamma \in \Gammah^-$.
We drop the $\pm$
superscripts on solutions to the defining ODEs, with a similar convention.
At time $t=0$ we have
\begin{align}
\label{eq_from_def}
\tilde{u}|_{t=0} &= \sum_{\gamma \in \Gammah} \tilde h_{\gamma}\Phih_{\gamma}(0,\cdot),
\\
\label{eq_from_def2}
\partial_t \tilde{u}|_{t=0} &= \sum_{\gamma \in \Gammah} \tilde h_{\gamma}
\p_t\Phih_{\gamma}(0,\cdot).
\end{align}

By Theorem \ref{th_centers_ws}, the centers of beams $\Phihpm_{\gamma}$ are away from the boundary
$\{x_1=0\}$ at initial time; we let $\epsilon>0$ be such that \eqref{eq:distance_from_bd} holds.

\step{1}{Localization}. Intuitively, \eqref{eq:distance_from_bd} means that, at time $t=0$, the
right-half space is away from the wave-front set of the solution, and the
parametrix is micro-locally of lower order. To formalize this reasoning, let us consider a smooth
cut-off function $\eta: \Rdst \to [0,1]$, such that
\[
\eta(x) =
\begin{cases}
0, \qquad x_1 \leq - \epsilon,
\\
1, \qquad x_1 \geq - \epsilon/2.\\
\end{cases}
\]
We also define $\eta_\gamma := \eta$ for all $\gamma \in \Gammah$.
By \eqref{eq:distance_from_bd}, $\eta$
vanishes near $x^{h}_{\gamma}(0)$ and, therefore,
\begin{align}
\label{eq_all}
\{\eta_\gamma:\gamma\in\Gammah^\pm\} = \strongop^k\pt{\Upsstarpm,\{0\}},
\qquad \mbox{for all } k \geq 0.
\end{align}

\step{2}{The $H^1$ norm of \eqref{eq_from_def}}.
We use the fact that the back-propagated parameters are well-spread
- Theorem \ref{th_match_ws}, the Bessel bounds
- Theorem \ref{theo_bessel_bounds}, and \eqref{eq_all} with $k=1$ to estimate
\begin{align*}
\Big\|\sum_{\gamma \in \Gammah}
\tilde h_{\gamma}\Phih_{\gamma}(0,\cdot) \Big\|^2_{H^1(\Rdplus)}
&\leq
\Big\|\sum_{\gamma\in \Gammah} \tilde h_{\gamma}
\Phih_{\gamma}(0,\cdot)\eta \Big\|^2_{H^1(\Rdst)}
\\
&\leq
\CTime \sum_{\gamma\in \Gammah}4^{j(2-1)}
|\tilde h_{\gamma}|^2
\leq \CTime \norm{h}^2_{H^{1/2}}.
\end{align*}

\step{3}{The $L^2$ norm of \eqref{eq_from_def2}}.
Using Lemma \ref{lemma_tf} we see that
\begin{align*}
\p_t \Phih_{\gamma}(0,x)
= \left(F^1_\gamma(x)+4^j F^2_\gamma(x)\right)
\Phihpm_{\gamma},
\end{align*}
with $F^1, F^2 = \strongop^0(\Upsstarpm,\{0\})$. Combining this with \eqref{eq_all}, we can proceed
as in Step 2 to deduce that
\begin{align*}
&\bignorm{\sum_{\gamma \in \Gammah} \tilde h_{\gamma} \p_t\Phih_{\gamma}(0,\cdot)}_{L^2(\Rdplus)}
\leq
\bignorm{\sum_{\gamma \in \Gammah} \tilde h_{\gamma} \p_t\Phih_{\gamma}(0,\cdot) \eta}_{L^2(\Rdst)}
\leq \CTime \norm{h}_{H^{1/2}}.
\end{align*}
This completes the proof.
\end{proof}

\begin{theorem}[Boundary conditions are asymptotically satisfied]\label{T:bound}
Under the hypothesis of Theorem~\ref{th:boundary}, the Gaussian beam solution $\tilde u$ satisfies:
\begin{align*}
\norm{\tilde u(\cdot,0,\cdot)-\cuth}_{H^1([0,T]\times\R^{d-1})} \leq \CTime
\norm{h}_{H^{1/2}(\Rdst)}.
\end{align*}
\end{theorem}
\begin{proof}
We use the same short-hand notation as in the proof of Lemma \ref{lem_prep1}.
According to the definitions,
\begin{align*}
&\tilh(t,x_*)=\sum_{\gamma \in \Gammah} \tilde h_\gamma \vp_{\g}(t,x_*),
\\
&\tilde u(t,0,x_*) = \sum_{\gamma \in \Gammah} \tilde h_\gamma \Phih_{\gamma}(t,0,x_*),
\qquad t \in \Rst, x_* \in \Rst^{d-1}.
\end{align*}
Therefore,
\begin{align*}
&\norm{\tilde u(\cdot,0,\cdot)-\cuth}_{H^1([0,T]\times\R^{d-1})}
\\
&\qquad \leq
\norm{\cuth-\tilh}_{H^1([0,T]\times\R^{d-1})} +
\Big\Vert
{\sum_{\gamma \in \Gammah} \tilde h_\gamma
\big( \vp_{\g}-\Phih_{\gamma}(\cdot,0,\cdot)
\big)} \Big\Vert
_{H^1([0,T]\times\R^{d-1})}.
\end{align*}
By \eqref{eq:wavefront_h}, the first term in the last equation is suitably bounded. Let us focus on
the second
term.

We invoke Theorem~\ref{th_match_estimates}. Let
$\eta^1 \in C^\infty(\Rst)$ be a smooth compactly-supported cut-off window
such that $\eta^1 \equiv 1$ on $[0,T]$ and
$\eta^2 \in C^\infty(\Rst^{d-1})$
a smooth function supported on $B_2(0)$ that is $\equiv 1$ on $B_1(0)$. We write
$\eta^2_{\gamma}(x_*) = \eta^2(x_*-x_{\gamma,*}(t_\gamma))$. With the notation of
Theorem~\ref{th_match_estimates},
\begin{align*}
&\Big\|\sum_{\g \in\Gammah} \tilde h_\gamma \big( \vp_{\g}-\Phih_{\gamma}(\cdot,0,\cdot)
\big) \Big\|_{H^1([0,T]\times \R^{d-1})}
\leq \Big\|\sum_{\g \in\Gammah} \tilde h_\gamma \big( \vp_{\g}-\Phih_{\gamma}(\cdot,0,\cdot)
\big)
\eta^1 \Big\|_{H^1(\R^{d})}
\\
&\qquad \leq  \Big\|\sum_{\g \in\Gammah} \tilde h_\gamma
\big(\func_\gamma-\Phih_{\gamma}(\cdot,0,\cdot)
\big)
\eta^1
\pt{1-\eta^2_{\gamma}} \Big\|_{H^1(\R^{d})}
\\
&\qquad \qquad + \Big\|\sum_{\g \in\Gammah} \tilde h_\gamma
\big( \func_\gamma -\Phih_{\gamma}(\cdot,0,\cdot)
\big)
\eta^1
\eta^2_\gamma \Big\|_{H^1(\R^{d})}
\\
&\qquad \leq \Big\|\sum_{\g \in\Gammah} \tilde h_\gamma\varphi_\gamma
\pt{4^j R^1_\gamma
+R^2_\gamma} \Big\|_{H^1(\R^{d})}
+ \Big\|\sum_{\g \in\Gammah} \tilde h_\gamma \newPhi_{\gamma}(\cdot,0,\cdot)
R^3_{\gamma}
\Big\|_{H^1(\R^{d})}.
\end{align*}
We use the information on the vanishing orders of $R^k$, $k=1,2,3$,
the fact that the beams $\{\newPhi_{\gamma}\}$ are well-spread,
and the Bessel bounds from Theorem~\ref{theo_bessel_bounds}
- with $(t,x_*)$ as integration variable instead of $x$ - to conclude
see that the remaining terms are dominated by $\|h\|_{H^{1/2}(\R^{d})}$.
This completes the proof.
\end{proof}

\subsection{Proof of the main result}
\begin{proof}[{Proof of Theorem~\ref{th:boundary}}]
The function $v:= \tilde{u} - u$ solves the problem:
\begin{eqnarray}
\left\{
\begin{aligned}
&\partial^2_{t} v (t,x) - \speed(x)^2 \Delta_x v (t,x) = f(t,x),  &\qquad t\in[0,T], x \in \Rdplus,
\\
&v(0, x)= \tilde{u}(0,x), &\qquad x \in \Rdplus,
\\
&v_t(0,x) = \tilde{u}_t(0,x),  &\qquad x \in \Rdplus,
\\
&v(t,0,x_*) = \tilde{u}(t,0,x_*)-\cuth(t,x_*),  &\qquad t\in[0,T], x_* \in \Rst^{d-1},
\end{aligned}
\right.
\end{eqnarray}
where $f(t,x)=\partial^2_{t} \tilde{u} (t,x) - \speed(x)^2 \Delta_x \tilde{u} (t,x)$.
By the energy estimates for the wave equation (see Theorem~\ref{th_in_llt})
we have
\begin{align*}
&\sup_{t\in[0,T]} \norm{v(t,\cdot)}_{H^1(\Rdplus)}
+
\sup_{t\in[0,T]} \norm{\partial_t v(t,\cdot)}_{L^2(\Rdplus)}
\\
&\qquad
\leq C_T \Big(
\norm{\tilde{u}(0,\cdot)}_{H^1(\Rdplus)}+\norm{\tilde{u}_t(0,\cdot)}_{L^2(\Rdplus)}+\sup_{t\in[0,T]}
\norm{f(t,\cdot)}_{
L^2(\Rdplus)}
\\
&\qquad\qquad
+\norm{\tilde{u}(\cdot,0,\cdot)-\cuth}_{H^1([0,T]\times\Rst^{d-1})}
\Big).
\end{align*}
The term involving $f$ can be estimated by Theorems \ref{th_error_wave} and \ref{th_match_ws}
as
\begin{align*}
\sup_{t\in[0,T]}\norm{f(t,\cdot)}_{L^2}^2
\leq C_T \sum_{\gamma \in \Gammah}
4^j | \tilde h_\gamma |^2 \leq C_T
\norm{h}^2_{H^{1/2}(\Rdst)},
\end{align*}
while the other three terms are similarly bounded, by Lemma~\ref{lem_prep1} and
Theorem~\ref{T:bound}. This completes the proof.
\end{proof}

\section{Proofs related to the back-propagated parameters}
\label{sec_pr_th_match_ws}
This section is devoted to pending proofs related to Section \ref{sec_match}.

\subsection{Proof of Claim \ref{Cl_ric}}
\label{sec_proof_of_Cl_ric}
\begin{proof}
We use the notation of Section \ref{sec_back}.

\step{1}{Existence and uniqueness}.
In compact notation, we look for a symmetric matrix $M(t_\gamma) \in \bC^{d\times d}$ such that
\begin{align}
\begin{bmatrix}
\label{eq_ma}
\widetilde{M}_{\gamma,11}& \widetilde{M}_{\gamma,1*}^t \\
\widetilde{M}_{\gamma,1*} & \pt{M_{\gamma}(t_\gamma)}_{**}
\end{bmatrix}
=
\begin{bmatrix}
2\pi \mi & 0 \\
0 & 2 \pi \mi I_{d-1}
\end{bmatrix}
\end{align}
where
\begin{align}
\label{eq_mb}
&\widetilde{M}_{\gamma,11}= \dot{x}_{\gamma}(t_{\gamma})
\cdot M(t_{\gamma}) \dot{x}_{\gamma}(t_{\gamma}) -\dot{p}_{\gamma}(t_{\gamma}) \cdot
\dot{x}_{\gamma}(t_{\gamma}),
\\
\label{eq_mc}
&\widetilde{M}_{\gamma,1*}=
\left(
\dot{p}_{\gamma}(t_\gamma)
-M(t_\gamma)
\dot{x}_{\gamma}(t_\gamma)
\right)_*.
\end{align}
We first assume that we have such a matrix $M(t_\gamma)$ and deduce the values of its entries.
From \eqref{eq_ma} we see that
\begin{align}
\label{eq_defm1}
(M(t_\gamma))_{**}=2\pi \mi I_{d-1}.
\end{align}
Using this together with \eqref{eq_mc} and \eqref{eq_ma} we see that
\[
\pt{M(t_{\gamma})}_{k1}\dot{x}_{\gamma,1}(t_\gamma)
= \dot{p}_{\gamma,k}(t_\gamma)-\ksx \dot{x}_{\gamma,k}(t_\gamma),
\qquad k=2,\dots d.
\]
Since, by \eqref{eq_lowbdx1}, $\dot{x}_{\gamma,1}(t_{\gamma}) \not= 0$ and
$M(t_\gamma)$ is symmetric, we can solve
\begin{align}
\label{eq_defm2}
\pt{M(t_{\gamma})}_{1*}  =
\left(\dot{x}_{\gamma,1}(t_\gamma)\right)^{-1}
\left(\dot{p}_{\gamma,*}(t_\gamma)
-\ksx
\dot{x}_{\gamma,*}(t_\gamma) \right).
\end{align}
We now compare the $(1,1)$ entries in \eqref{eq_ma} and use \eqref{eq_defm1} and \eqref{eq_defm2}
together with
\eqref{eq_mb} to obtain
\begin{align*}
2 \pi \mi &= \dot{x}_{\gamma}(t_{\gamma})
\cdot M(t_{\gamma})\dot{x}_{\gamma}(t_{\gamma}) -\dot{p}_{\gamma}
(t_{\gamma})\cdot \dot{x}_{\gamma}(t_{\gamma})\vspace{3mm}
\\
&= \abs{\dot{x}_{\gamma,1}(t_{\gamma})}^2 \pt{M(t_\gamma)}_{11}
+ 2\dot{x}_{\gamma,1}(t_{\gamma})\cdot
\pt{M(t_\gamma)}_{1*} \dot{x}_{\gamma,*}(t_{\gamma})\vspace{3mm}
\\
&\qquad
+
\dot{x}_{\gamma,*}(t_{\gamma})
\cdot \pt{M(t_\gamma)}_{**}
\dot{x}_{\gamma,*}(t_{\gamma})
-
\dot{p}_{\gamma}(t_{\gamma})\cdot
\dot{x}_{\gamma}(t_{\gamma})\vspace{3mm}
\\
&=\abs{\dot{x}_{\gamma,1}(t_{\gamma})}^2 \pt{M(t_\gamma)}_{11}
+ 2 \left(
\dot{p}_{\gamma,*}(t_{\gamma}) - 2\pi \mi
\dot{x}_{\gamma,*}(t_{\gamma})
\right) \cdot
\dot{x}_{\gamma,*}(t_{\gamma})
\vspace{2mm} \\
&\qquad
+
2\pi \mi \abs{\dot{x}_{\gamma,*}(t_{\gamma})}^2
-
\dot{p}_{\gamma}(t_{\gamma})\cdot
\dot{x}_{\gamma}(t_{\gamma})
\vspace{3mm}\\
&=\abs{\dot{x}_{\gamma,1}(t_{\gamma})}^2 \pt{M(t_\gamma)}_{11}
+ 2\dot{p}_{\gamma,*}(t_{\gamma})\cdot
\dot{x}_{\gamma,*}(t_{\gamma})
\vspace{3mm} \\
&\qquad
-
2\pi \mi \abs{\dot{x}_{\gamma,*}(t_{\gamma})}^2
-
\dot{p}_{\gamma}(t_{\gamma})\cdot
\dot{x}_{\gamma}(t_{\gamma}).
\end{align*}
Using again that, by \eqref{eq_lowbdx1}, $\dot{x}_{\gamma,*}(t_{\gamma}) \not= 0$ and we conclude
that
\begin{equation}
\label{eq_defm3}
\pt{M(t_\gamma)}_{11}
=
\abs{\dot{x}_{\gamma,1}(t_{\gamma})}^{-2}
\left(
2\pi \mi
\pt{
1\!+\!\abs{\dot{x}_{\gamma,*}(t_{\gamma})}^2
}
\!-\!2 \dot{p}_{\gamma,*}(t_{\gamma})\cdot
 \dot{x}_{\gamma,*}(t_{\gamma})
+\dot{p}_{\gamma}(t_{\gamma})\cdot
\dot{x}_{\gamma}(t_{\gamma})
\right).
\end{equation}
Hence, the matrix $M(t_\gamma)$ is completely determined by the desired conditions.
Let us define $M(t_\gamma)$ by \eqref{eq_defm1}, \eqref{eq_defm2}
and \eqref{eq_defm3} and the requirement of symmetry. We see that such matrix
solves \eqref{sys:M_gamma}.

\step{2}{Positivity and bounds}.
Inspecting \eqref{eq_defm1}, \eqref{eq_defm2} and \eqref{eq_defm3} and using
Claim \ref{claim_deriv} we see that $\norm{M(t_\gamma)}$ is bounded uniformly for $\gamma \in
\Gammah$. According to the
definitions, the imaginary part of $M(t_\gamma)$
is of the form
\begin{align*}
\tfrac{1}{2\pi}\Im{M(t_\gamma)}=
\begin{bmatrix}
a_{11}
&
a_{12} & a_{13} & \ldots & a_{1d}
\\
a_{12} & 1 & 0 & \ldots & 0
\\
a_{13} & 0 & 1 & \ldots & 0
 \\
\ldots & \ldots & \ldots & \ldots & 0
\\
a_{1d} & 0 & 0 & \ldots & 1
\end{bmatrix},
\end{align*}
where
\begin{align}
a_{11} &= |\dot{x}_{\gamma,1}(t_{\gamma})|^{-2} (1+\abs{x_*(t_\gamma)}^2),
\\
a_{1k} &= -\dot{x}_{\gamma,1}(t_{\gamma}))^{-1} \dot{x}_{\gamma,k}(t_{\gamma}),
\qquad k=2,\dots d.
\end{align}
Note that
\begin{align*}
a_{11}-a_{12}^2-\ldots -a_{1d}^2=|\dot{x}_{\gamma,1}(t_{\gamma})|^{-2}
\gtrsim 1,
\end{align*}
by \eqref{eq_lowbdx1}.
Hence, by Lemma \ref{lemma_LA} in the Appendix, we conclude that
$\Im(M(t_\gamma))$ is a positive matrix and $\Im(M(t_\gamma)) \gtrsim I_d$, as desired.
\end{proof}

\subsection{Proof of Theorem \ref{th_match_ws}}
The goal of this section is to show that both families of back-propagated GB parameters constructed
in Section \ref{sec_match}
are well-spread. This involves comparing the constructed maps
\begin{align*}
\gammastarplmi=(\omega_\gamma^{\pm,h},
a_\gamma^{\pm,h},
\xi_\gamma^{\pm,h},
\newA_\gamma^{\pm,h},
\newM_\gamma^{\pm,h}), \qquad \gamma \in \Gammah^\pm,
\end{align*}
to the standard one
\begin{align}
\alphast_\gamma = (4^j,2^{-j}\lambda, \pointjk, 2^{\frac{d}{4}}, \ii I_d),
\qquad \gamma = (j,k,\lambda) \in \Gamma.
\end{align}
We follow the notation of Section \ref{sec_match}:
when convenient, we drop the superscripts for the functions $x^{h,\pm}_\gamma(t)$,
$p^{h,\pm}_\gamma(t), \ldots$,
writing instead $x_\gamma(t)$, $p_\gamma(t), \ldots$ We keep however the superscripts in the tuple
of parameters
$\gammastarplmi$ to avoid confusion with the standard one.

Recall from Section \ref{sec_match} that for $\gamma \in \Gammah$,
$t_\gamma = 2^{-j}\lambda_1$ and that $x_{\gamma}(t_\gamma) = (0, 2^{j} \lambda_*)$.

As a preparation for the proof of Theorem \ref{th_match_ws}, we show the following.
\begin{lemma}
\label{lemma_flow_match}
For $\gamma=(j,k,\lambda),
\gamma'=(j',k',\lambda') \in \Gammah^+$:
\begin{align}
\label{eq_lemma_flow_match}
&\left|\pointjk - \pointjkp \right|^2
\lesssim \left|4^{j} p^{h,+}_\gamma(t_\gamma) - 4^{j'}p^{h,+}_{\gamma'}(t_{\gamma'})\right|^2 +
4^{j+j'}
|2^{-j}\lambda_* -
2^{-j'}\lambda_*'|^{2}.
\end{align}
An analogous statement holds for $\Gammah^-$.
\end{lemma}
\begin{proof}
We treat the family $\Gammah^+$.
To further simplify the notation, throughout this proof
we write $p_\gamma = p_\gamma(t_\gamma)$,
$\poneg =p_{\gamma,1}(t_\gamma),
\pastg = p_{\gamma,*}(t_\gamma)$, and $\speed_{\gamma} = \speed(0,2^{-j}\lambda_*)$.
Recall also that $\taugam = (\pointjkt)_1 <0$ - cf. \eqref{eq_ptg} and \eqref{eq_match_2}.
Hence, by \eqref{eq:initpst} and \eqref{eq:q},
\begin{align}
\label{eq_aux}
p_{\gamma,1}=
\sqrt{\tfrac{\taugam^2}
{\speed_{\gamma}^2} - \pastg^2} = \tfrac{|\taugam|}{\speed_{\gamma}}\sqrt{1 -
\tfrac{\speed_{\gamma}^2}{\taugam^2}\pastg^2},
\end{align}
and, by \eqref{eq_pjkt},
$(\pointjk)_1 = \tfrac{4^j}{2\pi} \taugam$,
$(\pointjk)_* = \tfrac{4^j}{2\pi} \pastg$.
With this notation,
the estimate we want to prove is:
\begin{equation*}
\left |4^{j} \taugam-4^{j'} \taugamp\right|^2
+
| 4^{j}\pastg -  4^{j'} \pastgp |^2
\lesssim
| 4^j p_\gamma - 4^{j'}p_{\gamma'} |^2
+ 4^{j+j'}|2^{-j}\lambda_{*} -2^{-j'}\lambda'_{*}|^2.
\end{equation*}
Clearly, it suffices to show that
\begin{equation}
\label{eq_est_new}
\left |4^{j} \taugam-4^{j'} \taugamp\right|^2
\lesssim
| 4^j p_\gamma - 4^{j'}p_{\gamma'} |^2
+ 4^{j+j'}|2^{-j}\lambda_{*} -2^{-j'}\lambda'_{*}|^2.
\end{equation}

\step{1}{We show that}
\begin{equation}
\left |4^{j}
\taugam-4^{j'}
\taugamp\right|^2
\lesssim
\left |4^{j} \frac{\taugam}{\speed_{\gamma}}
	 	-4^{j'}\frac{\taugamp}{\speed_{\gamma'}}\right|^2
+4^{j+j'}|2^{-j}\lambda_{*} -2^{-j'}\lambda'_{*}|^2.
	 	\label{eq:last_tbp}
\end{equation}
Using that $\abs{\taugam} \leq \cnumb{1}$ for some constant $\cnumb{1}$ - independent of $\gamma$
- cf. \eqref{eq_speed_bounds} and \eqref{eq_pjkt}, we estimate
\begin{align*}
\left |4^{j} \frac{\taugam}{\speed_{\gamma}}
-4^{j'}\frac{\taugamp}{\speed_{\gamma'}}\right|
& =
\left |4^{j} \frac{\taugam}{\speed_{\gamma}}
-4^{j'}\frac{\taugamp}{\speed_{\gamma}} + 4^{j'}\frac{\taugamp}{\speed_{\gamma}}
-4^{j'}\frac{\taugamp}{\speed_{\gamma'}} \right|
\\
&\geq
\left |4^{j} \frac{\taugam}{\speed_{\gamma}}
-4^{j'}\frac{\taugamp}{\speed_{\gamma}}
\right|
- \left |4^{j'}\frac{\taugamp}{\speed_{\gamma}}
-4^{j'}\frac{\taugamp}{\speed_{\gamma'}} \right|
\\
&\geq
\cspeed^{-1} \left| 4^{j}\taugam-4^{j'} \taugamp \right|
- 4^{j'}\cnumb{1}\left|\frac{1}{\speed_{\gamma}}-
\frac{1}{\speed_{\gamma'}}\right|
\\
&= \cspeed^{-1}
\left |4^{j} \taugam-4^{j'} \taugamp\right|
-4^{j'}\cnumb{1}\left|\frac{1}{\speed(0,2^{-j}\lambda_{*})}-\frac{1}{\speed(0,2^{-j'}\lambda'_{*})}
\right|.
\end{align*}
Since the velocity $\speed$ has (uniformly) bounded derivatives and is bounded below - cf.
\eqref{eq_speed_bounds} -
we conclude that
$\left|\frac{1}{\speed(0,2^{-j}\lambda_{*})}-\frac{1}{\speed(0,2^{-j'}\lambda'_{*})}
\right| \lesssim \left|2^{-j}\lambda_{*} -2^{-j'}\lambda'_{*}\right|$.
Consequently,
\begin{align*}
\left |4^{j} \taugam-4^{j'} \taugamp\right|^2
&\lesssim
\left |4^{j} \frac{\taugam}{\speed_{\gamma}}
-4^{j'}\frac{\taugamp}{\speed_{\gamma'}}\right|^2
+4^{2j'} \left|2^{-j}\lambda_{*} -2^{-j'}\lambda'_{*}\right|^2.
\end{align*}
Similarly,
$
\left |4^{j} \taugam-4^{j'} \taugamp\right|^2
\lesssim
\left |4^{j} \frac{\taugam}{\speed_{\gamma}}
-4^{j'}\frac{\taugamp}{\speed_{\gamma'}}\right|^2
+4^{2j}\left|2^{-j}\lambda_{*} -2^{-j'}\lambda'_{*}\right|^2
$,
and therefore
\begin{align*}
\left |4^{j} \taugam-4^{j'} \taugamp\right|^2
&\lesssim
\left |4^{j} \frac{\taugam}{\speed_{\gamma}}
-4^{j'}\frac{\taugamp}{\speed_{\gamma'}}\right|^2
+\min\{4^{2j},4^{2j'}\}
\left|2^{-j}\lambda_{*} -2^{-j'}\lambda'_{*}\right|^2,
\\
&\lesssim
\left |4^{j} \frac{\taugam}{\speed_{\gamma}}
-4^{j'}\frac{\taugamp}{\speed_{\gamma'}}\right|^2
+4^{j+j'}\left|2^{-j}\lambda_{*} -2^{-j'}\lambda'_{*}\right|^2,
\end{align*}
showing that \eqref{eq:last_tbp} indeed holds.

\step{2}{We show that}
\begin{align}
\label{eq_lf1}
\left |4^{j} \frac{\taugam}{\speed_{\gamma}}
-4^{j'}\frac{\taugamp}{\speed_{\gamma'}}\right|^2
\lesssim | 4^j p_\gamma - 4^{j'}p_{\gamma'} |^2.
\end{align}

We denote $\epsilon_{\gamma} = \tfrac{\speed_{\gamma}\abs{p_{\gamma,*}}}{\abs{\taugam}}$. Hence,
by \eqref{eq_aux},
$p_{\gamma,1}=
\tfrac{|\taugam|}{\speed_{\gamma}}\sqrt{1 -
\epsilon_\gamma^2}$.
By the grazing ray condition
\eqref{eq_graray},  $0\leq \epsilon_{\gamma} <1$.
We note that for $\gamma, \gamma' \in \Gammah^+$,
$\taugam<0$ and $\taugamp <0$, while $p_{\gamma,1}>0$ and $p_{\gamma',1}>0$.
Keeping these facts in mind and using inequality:
$0 \leq \left( 1 - x^2\right)^{\half}\left( 1 - y^2\right)^{\half} + xy \leq 1$, for $x,y \in
[0,1]\times[0,1]$,
we estimate
\begin{align*}
&\left |4^{j} \frac{\taugam}{\speed_{\gamma}}
-4^{j'}\frac{\taugamp}{\speed_{\gamma'}}\right|^2
=
4^{2j} \frac{\taugam^2}{\speed^2_{\gamma}}
	 	+4^{2j'}\frac{\taugamp^2}{\speed^2_{\gamma'}}
	 	- 2 \cdot
4^{j+j'}\frac{\abs{\taugam}\abs{\taugamp}}{\speed_{\gamma}\speed_{\gamma'}}
\\
&\quad\leq
4^{2j} \frac{\taugam^2}{\speed^2_{\gamma}}
	 	+4^{2j'}\frac{\taugamp^2}{\speed^2_{\gamma'}}
	 	- 2 \cdot
4^{j+j'}\frac{\abs{\taugam}\abs{\taugamp}}{\speed_{\gamma}\speed_{\gamma'}}\left[
	 		   \left( 1 - \epsilon^2_{\gamma}\right)^{\half}
	 		   \left( 1 - \epsilon^2_{\gamma'}\right)^{\half}
	 		   + \epsilon_{\gamma}\epsilon_{\gamma'}\right]
\\
&\quad=
4^{2j} \frac{\taugam^2}{\speed^2_{\gamma}}
	 	+4^{2j'}\frac{\taugamp^2}{\speed^2_{\gamma'}}
	 	- 2 \cdot
4^{j+j'}\frac{\abs{\taugam}\abs{\taugamp}}{\speed_{\gamma}\speed_{\gamma'}}
	 		   \left( 1 - \epsilon^2_{\gamma}\right)^{\half}
	 		   \left( 1 - \epsilon^2_{\gamma'}\right)^{\half}
-2 \cdot 4^{j+j'} \abs{\pastg}\abs{\pastgp}
\\
&\quad=
4^{2j} \frac{\taugam^2}{\speed^2_{\gamma}}
	 	+4^{2j'}\frac{\taugamp^2}{\speed^2_{\gamma'}}
	 	- 2 \cdot 4^{j+j'}
	 		   \big( \tfrac{\taugam^2}{\speed_{\gamma}^2} - \abs{\pastg}^2\big)^{\half}
	 		   \big( \tfrac{\taugamp^2}{\speed_{\gamma'}^2} -
\abs{\pastgp}^2\big)^{\half}
-2 \cdot 4^{j+j'} \abs{\pastg}\abs{\pastgp}.
\end{align*}
Using the arithmetic-geometric means inequality: $
4^{2j} \abs{\pastg}^2 + 4^{2j'} \abs{\pastgp}^2
\leq 2 \cdot 4^{j+j'} \abs{\pastg}\abs{\pastgp}$ we conclude that
\begin{align*}
\left |4^{j} \frac{\taugam}{\speed_{\gamma}}
-4^{j'}\frac{\taugamp}{\speed_{\gamma'}}\right|^2
&\leq
4^{2j} \big(\tfrac{\taugam^2}{\speed^2_{\gamma}}-\abs{\pastg}^2\big)
+4^{2j'}\big(\tfrac{\taugamp^2}{\speed^2_{\gamma'}}-\abs{\pastgp}^2\big)
- 2 \cdot 4^{j+j'}
\big( \tfrac{\taugam^2}{\speed_{\gamma}^2} - \abs{\pastg}^2\big)^{\half}
  \big( \tfrac{\taugamp^2}{\speed_{\gamma'}^2} - \abs{\pastgp}^2\big)^{\half}
\\
&=
4^{2j} \poneg^2
+4^{2j'} \ponegp^2
- 2 \cdot 4^{j+j'} \cdot \poneg \cdot \ponegp
= \left( 4^j \poneg - 4^{j'} \ponegp \right)^2
\\
&\leq | 4^j p_\gamma - 4^{j'}p_{\gamma'} |^2,
\end{align*}
as claimed.

\step{3}{Finally, we combine \eqref{eq:last_tbp} and \eqref{eq_lf1}
to deduce \eqref{eq_est_new}.} The proof for $\Gammah^-$ is similar, with the difference
that a minus signs is present in \eqref{eq_aux}.
\end{proof}

We may now prove the announced result.

\begin{proof}[{Proof of Theorem \ref{th_match_ws}}]
We consider one of the families, $\Upsstarp$ or $\Upsstarm$, and drop the superscript $+$.
We verify the conditions in Definition \ref{def:wsp}.

\step{1}{Estimates for $\omega^{h}_{\gamma}$ and $\xi^{h}_{\gamma}$}.\\
By definition, $\omega^{h}_{\gamma} = 4^j$ - cf. \eqref{eq_starparameter}. Moreover,
using \eqref{eq:inad}, the fact that $c$ is bounded below, and the non-tangential propagation
estimate in
\eqref{eq_fpr}
we conclude that $\abs{p_\gamma(t_\gamma)} \asymp 1$.
Using the fact that the Hamiltonian is constant on its flow, we can propagate this estimate to
$t=0$:
\begin{align*}
1 \asymp \abs{p_\gamma(t_\gamma)} \asymp \abs{c(x(t_\gamma)} \abs{p_\gamma(t_\gamma)}
= \abs{H(x(t_\gamma), p(t_\gamma))}= \abs{H(x(0), p(0))}
\asymp \abs{p_\gamma(0)}.
\end{align*}
Hence $\abs{\xi^{h}_{\gamma}} = \frac{4^j}{2\pi} \abs{p_\gamma(0)}
\asymp 4^j = \omega^{h}_{\gamma}$. This establishes one of the properties that we need in order to
check the well-spreadness of $\Gammah^\pm$, and, additionally, it allows us to invoke Lemma
\ref{lemma_unif_flow_1} for this family of parameters.

Since $\omega^{h}_{\gamma} = \omega^{st}_{\gamma} = 4^j$, in what follows we write unambiguously
$\omega_\gamma$.

\step{2}{Some constants}. Recall the assumption in \eqref{eq_tcn}.  Let us Taylor expand:
\begin{align}
\label{eq_expan}
x_{\gamma}(t) = x_{\gamma}(\tgam)+ (t-\tgam) y_\gamma(t),
\end{align}
where $y_{\gamma,i} \in C([-T,T])$ uniformly on $\gamma$
by Lemma \ref{lemma_unif_flow_1}. Since $x_{\gamma,1}(t_\gamma)=0$,
Claim \ref{claim_deriv} allows us to invoke
the \emph{cone condition} in \eqref{eq_cone} and deduce that
\begin{align*}
\abs{y_{\gamma,1}(t)} \geq \cturn  >0, \qquad t\in[-T,T],
\end{align*}
for some constant $\cturn>0$. In addition, by Lemma \ref{lemma_unif_flow_1},
\begin{align}
\label{eq_acons}
\cnumb{1} := \sup_{\gamma \in \Gammah} \sup_{t \in [0,T]}
\max\sett{\abs{y_{\gamma}(t)},\abs{\dot{p}_{\gamma}(t)}}
\end{align}
is finite. We let $\varepsilon := \half
\min\{\half \tfrac{\cturn}{\cnumb{1}},\cturn\}$ and note that
\begin{align}
\label{eq_important}
\abs{y_{\gamma,1}(t)} \geq \varepsilon \left(
\abs{y_{\gamma,*}(t)}
+1\right), \qquad t\in[-T,T].
\end{align}

\step{3}{We show that}
$|\aast_{\gamma}-\aast_{\gamma'}| \lesssim |a^h_{\gamma}-a^h_{\gamma'}| +1$.

According to the definitions,
\begin{align*}
&|a^h_{\gamma}-a^h_{\gamma'}| = |x_{\gamma}(0)- x_{\gamma'}(0)|,
\\
&|\aast_{\gamma}-\aast_{\gamma'}| = |2^{-j}\lambda - 2^{-j'}\lambda'|.
\end{align*}
By Lemma~\ref{lemma_unif_flow_1}
\begin{align*}
|x_{\gamma}(\tgam)- x_{\gamma'}(\tgam)|^2 \lesssim  |x_{\gamma}(0)- x_{\gamma'}(0)|^2 + 1.
\end{align*}
Hence, it suffices to show that
\begin{equation}\label{eq:Vect_bd}
|2^{-j}\lambda - 2^{-j'}\lambda'|^2
\lesssim
|x_{\gamma}(\tgam)- x_{\gamma'}(\tgam)|^2.
\end{equation}

To this end, we use the linearization in \eqref{eq_expan},
\begin{align*}
x_{\gamma'}(\tgam) = x_{\gamma'}(\tgamp)+ (\tgam-t_{\gamma'}) y_{\gamma'}(t),
\end{align*}
and write
\begin{equation}\label{eq:mvt_comp}
|x_{\gamma}(\tgam)- x_{\gamma'}(\tgam)|
=
|x_{\gamma}(\tgam) - x_{\gamma'}(\tgamp) - ( \tgam- \tgamp) y_{\gamma'}(\tgam)|.
\end{equation}
Recall that $\tgam = 2^{-j}\lambda_1$ and $x_\gamma(t_\gamma)=(0,2^{-j}\lambda_*)$.
We use \eqref{eq_important} to estimate
\begin{align*}
&|x_{\gamma}(\tgam)- x_{\gamma'}(\tgam)|
=
\abs{(-( \tgam- \tgamp) y_{\gamma',1}(\tgam), x_{\gamma,*}(\tgam) - x_{\gamma',*}(\tgamp) - (
\tgam- \tgamp) y_{\gamma',*}(\tgam)}
\\
&\qquad\asymp
\abs{y_{\gamma',1}(\tgam)} \abs{2^{-j}\lambda_1- 2^{-j'}\lambda_1'}
+ \abs{2^{-j}\lambda_* - 2^{-j'}\lambda_*'+ (2^{-j}\lambda_1- 2^{-j'}\lambda_1')
y_{\gamma',*}(\tgam)}
\\
&\qquad \asymp \abs{y_{\gamma',1}(\tgam)}\abs{2^{-j}\lambda_1- 2^{-j'}\lambda_1'}
 + \varepsilon\abs{2^{-j}\lambda_* - 2^{-j'}\lambda_*'+ y_{\gamma',*}(\tgam)
 (2^{-j}\lambda_1- 2^{-j'}\lambda_1')}
\\
&\qquad\geq
\varepsilon |2^{-j}\lambda_* - 2^{-j'}\lambda_*'|
+\big(\abs{y_{\gamma',1}(\tgam)}-\varepsilon \abs{y_{\gamma',*}(\tgam)}\big)
|2^{-j}\lambda_1 - 2^{-j'}\lambda_1'|
\\
&\qquad\geq
\varepsilon |2^{-j}\lambda_* - 2^{-j'}\lambda_*'| +
\varepsilon |2^{-j}\lambda_1 - 2^{-j'}\lambda_1'|
\asymp |2^{-j}\lambda - 2^{-j'}\lambda'|.
\end{align*}
Hence,
\eqref{eq:Vect_bd} follows.

\step{4}{Estimates for $\newM^{h}_{\gamma}$}.\\
By Claim \ref{Cl_ric}, we know that $M_\gamma(t_\gamma)$ is symmetric,
$\norm{M_\gamma(t_\gamma)} \lesssim 1$ and
$\Im(M_\gamma(t_\gamma)) \gtrsim I_d$.
Those conclusions extend to $M_\gamma(0)$, since propagation preserves these
conditions with different time dependent constants. This is stated in
Lemma \ref{lemma_unif_flow_2} for forward propagation $0 \mapsto t$, but the same conclusion is
valid with an arbitrary initial time. (The general reference for this fact is \cite[Lemma
2.56]{Kat}.)
Since $\newM^{h}_{\gamma} = (2\pi)^{-1} M_\gamma(0)$, the conclusion follows.

\step{5}{Estimates for $\newA^{h}_{\gamma}$}.\\
By definition, $A_\gamma(t_\gamma) = 2^{\frac{d}{4}} 4^{j\frac{d}{4}}$ - cf. \eqref{eq_initA}.
Since, by Step 4, $\norm{M_\gamma(0)} \lesssim 1$ and $\abs{t_\gamma} \leq \cconcU$, by
Lemma \ref{lemma_unif_flow_2} we conclude that
$\abs{A_\gamma(0)}\asymp\abs{A_\gamma(t_\gamma)} \asymp  4^{j\frac{d}{4}}$. Hence, the choice made
in
\eqref{eq_starparameter}
yields
\[
\abs{\newA^{h}_{\gamma}} = 4^{-j\frac{d}{4}}A_\gamma(0) \asymp 1
\] as desired.

\step{6}{We show that}
\begin{equation}
\label{eq_other}
\abs{\omega_{\gamma} p_\gamma(t_\gamma) - \omega_{\gamma'}p_{\gamma'}(t_{\gamma'})}^2
\lesssim
\abs{\omega_{\gamma} p_\gamma(t_\gamma) - \omega_{\gamma'}p_{\gamma'}(t_{\gamma})}^2
+ \omega_{\gamma}\omega_{\gamma'} |2^{-j}\lambda_1 - 2^{-j'}\lambda_1'|^2.
\end{equation}
To see this, we assume without loss of generality that $\omega_{\gamma} \geq\omega_{\gamma'}$ and
use the mean value
theorem to find points $\togi \in [0,\cconcU ]$ such that
\begin{align*}
&\abs{\omega_{\gamma}p_\gamma(t_\gamma) - \omega_{\gamma'}p_{\gamma'}(t_{\gamma})}
\\
&\qquad\geq
\abs{\omega_{\gamma} p_\gamma(t_\gamma) - \omega_{\gamma'}p_{\gamma'}(t_{\gamma'})}
- \omega_{\gamma'}\left(\sum_{i=1}^d\abs{\dot{p}_{\gamma',i}(\togi)}\right)\abs{2^{-j}\lambda_1 -
2^{-j'}\lambda_1'}
\\
&\qquad{}\geq
 \abs{\omega_{\gamma} p_\gamma(t_\gamma) - \omega_{\gamma'}p_{\gamma'}(t_{\gamma'})}
 - \cnumb{1} \cdot \sqrt{\omega_{\gamma}} \sqrt{\omega_{\gamma'}} |2^{-j}\lambda_1 -
2^{-j'}\lambda_1'|,
\end{align*}
where $\cnumb{1}$ is given by \eqref{eq_acons}. Therefore, \eqref{eq_other} follows.

\step{7}{We show that}
\begin{align}
\label{eq:VectQ}
\left|\pointjk - \pointjkp \right|^2
\lesssim
\left|\omega_{\gamma}p_\gamma(t_\gamma) -
\omega_{\gamma'}p_{\gamma'}(\tgamp)\right|^2
+ \omega_{\gamma}\omega_{\gamma'}|2^{-j}\lambda_* - 2^{-j'}\lambda_*'|^{2}.
\end{align}
Since $\omega_{\gamma} = 4^j$ - cf. \eqref{eq_starparameter}, this is just
the content of Lemma~\ref{lemma_flow_match}.

\step{8}{We show that}
$d((a^{h}_\gamma,\xi^{h}_\gamma),(a^{h}_{\gamma'},\xi^{h}_{\gamma'})) \gtrsim
d((\aast_{\gamma},\xist_{\gamma}),(\aast_{\gamma'},\xist_{\gamma'})),
\: \gamma, \gamma' \in \Gammah^\pm$.

We combine the previous steps and Lemma \ref{lemma_unif_flow_1}
to obtain:
\begin{align*}
&d((\aast_{\gamma},\xist_{\gamma}),(\aast_{\gamma'},\xist_{\gamma'}))
\asymp
\omega_\gamma \omega_{\gamma'} |2^{-j}\lambda - 2^{-j'}\lambda'|^2 +
\left|\pointjk - \pointjkp \right|^2
&\mbox{}
\\
&\qquad\lesssim
\left|\omega_{\gamma}p_\gamma(t_\gamma) -
\omega_{\gamma'}p_{\gamma'}(\tgamp)\right|^2
+
\omega_\gamma \omega_{\gamma'} |2^{-j}\lambda - 2^{-j'}\lambda'|^2
&\qquad\mbox{by Step 7}
\\
&\qquad\lesssim
\left|\omega_{\gamma}p_\gamma(t_\gamma) -
\omega_{\gamma'}p_{\gamma'}(t_\gamma)\right|^2
+
\omega_\gamma \omega_{\gamma'} |2^{-j}\lambda - 2^{-j'}\lambda'|^2
&\qquad\mbox{by Step 6}
\\
&\qquad\lesssim
\left|\omega_{\gamma}p_\gamma(t_\gamma) -
\omega_{\gamma'}p_{\gamma'}(t_\gamma)\right|^2
+
\omega_\gamma \omega_{\gamma'} |x_{\gamma}(\tgam)- x_{\gamma'}(\tgam)|^2
&\qquad\mbox{by \eqref{eq:Vect_bd}}
\\
&\qquad\asymp
d\big(
\big(x_\gamma(t_\gamma),\omega_{\gamma}p_\gamma(t_\gamma)\big),\big(x_{\gamma'}(t_\gamma),\omega_{
\gamma'} p_{ \gamma' }
(t_\gamma)\big)\big)
& \mbox{ by } \eqref{eq:hypFLW_3c}
\\
&\qquad\asymp
d((a^{h}_\gamma,\xi^{h}_\gamma),(a^{h}_{\gamma'},\xi^{h}_{\gamma'})).
&\mbox{by }\eqref{eq:hypFLW_2}
\end{align*}
This concludes the proof.
\end{proof}

\appendix\section{Auxiliary Results}\label{AppendixA}
\subsection{A linear algebra lemma}
\begin{lemma}
\label{lemma_LA}
Let $A \in \Rst^{d\times d}$ be a matrix of the form:
\begin{align*}
A=
\begin{bmatrix}
a_{11}
&
a_{12} & a_{13} & \ldots & a_{1d}
\\
a_{12} & 1 & 0 & \ldots & 0
\\
a_{13} & 0 & 1 & \ldots & 0
 \\
\ldots & \ldots & \ldots & \ldots & 0
\\
a_{1d} & 0 & 0 & \ldots & 1
\end{bmatrix}.
\end{align*}
Suppose that $\cons,\Cons>0$ are constants, such that $\abs{a_{i,j}} \leq \Cons$
and $a_{11} - a_{12}^2 - \ldots - a_{1d}^2 \geq \cons$. Then there exist
constants $\cons', \Cons'>0$, that only depend on $\cons$ and $\Cons$, such that
$ \cons' I_d \leq A \leq \Cons' I_d$. (In particular, $A$ is positive definite.)
\end{lemma}
\begin{proof}
We premultiply $A$ by an adequate upper triangular matrix with ones in the diagonal
\begin{align}
\label{eq_im_ma}
\begin{bmatrix}
1 & -a_{12} & -a_{13} & \ldots & -a_{1d}
\\
0 & 1 & 0 & \ldots &0
\\
0 & 0 & 1 & \ldots &0
\\
\ldots & \ldots & \ldots & \ldots &0
\\
0 & 0 & 0 & \ldots &1
\end{bmatrix}
\cdot
\begin{bmatrix}
a_{11}
&
a_{12} & a_{13} & \ldots & a_{1d}
\\
a_{12} & 1 & 0 & \ldots & 0
\\
a_{13} & 0 & 1 & \ldots & 0
 \\
\ldots & \ldots & \ldots & \ldots & 0
\\
a_{1d} & 0 & 0 & \ldots & 1
\end{bmatrix},
\end{align}
to obtain
\begin{align*}
\begin{bmatrix}
a_{11}-a_{12}^2-\ldots -a_{1d}^2
&
0 & 0 & \ldots & 0
\\
a_{12} & 1 & 0 & \ldots & 0
\\
a_{13} & 0 & 1 & \ldots & 0
 \\
\ldots & \ldots & \ldots & \ldots & 0
\\
a_{1d} & 0 & 0 & \ldots & 1
\end{bmatrix}.
\end{align*}
Hence, the entries of $A$ are bounded and its determinant is bounded below by a positive constant.
The same argument applies to each principal minor of $A$. Hence, the conclusion follows.
\end{proof}

\subsection{Approximation errors}
\label{sec_proof_lemma_op_asympt}
Here, we give error bounds related to the approximate eikonal and transport equation.
These are proved in \cite{Lex3} in a slightly different form, and we only sketch
the modifications relevant to our setting.

\begin{lemma}
  \label{lemma_eik}
  Let $\Upsilon=\sett{\SIC_\gamma: \gamma \in \Gamma_0}$ be a well-spread set of Gaussian beam
parameters
  and $T \geq 0$.
  Then the following estimates hold
\begin{align}
\label{eq:Eik-Tr-1}
&\partial_{x_j}\theta_\gamma(t,x)=\strongop^0([-T,T],\Upsilon),\qquad
j=1,\ldots,d,
\\
\label{eq:Eik-Tr-2}
&\left(\partial_{t}\theta_\gamma(t,x)\right)^{2}
- \speed(x)^{2}\abs{\nabla_{x}\theta_\gamma(t,x)}^{2} = \strongop^3([-T,T],\Upsilon),
\\
\label{eq:Eik-Tr-3}
&2\partial_{t}\theta_\gamma(t,x)\frac{\partial_{t}\Ag(t)}{\Ag(t)}
+ \p^{2}_{t}\theta_\gamma(t,x)
- \speed(x)^{2}\Tr\left(\p^{2}_{x}\theta_\gamma(t,x)\right)= \strongop^1([-T,T],\Upsilon).
\end{align}
\end{lemma}

\begin{proof}
We first compute:
$\partial_{x_j}\theta_\gamma(t,x)=p_{\gamma,j}(t) +
2 M_{\gamma,j,*}(t) \cdot (x-x_\gamma(t))$, and use Lemma \ref{lemma_error_flow} to show
\eqref{eq:Eik-Tr-1}.
Second, as shown in \cite[Lemma 3.5]{Lex3}:
\[
\partial_{t}\theta^\pm_\gamma(t,x)
\pm \speed(x)\abs{\nabla_{x}\theta^\pm_\gamma(t,x)}  =
\strongop^3([-T,T],\Upsilon).
\]
This estimate, combined with \eqref{eq:Eik-Tr-1}, gives \eqref{eq:Eik-Tr-2}. Finally,
\eqref{eq:Eik-Tr-3} is proved in ~\cite[Lemma 3.12]{Lex3}. (The cited references treat the case of
the standard set of GB parameters, but the same proof applies to a general well-spread set;
the relevant estimates are in Lemma \ref{lemma_error_flow}.)
\end{proof}

\begin{proof}[Proof of Lemma \ref{lemma_op_asympt}]
Recall that - cf.~\eqref{eq:GB_elem}, \eqref{eq:GB_phase} -
\[
\Phi_{\gamma}(t,x) = \Ag (t) e^{i \omega_{\gamma} \theta_\gamma (t,x)}, \quad \gamma \in \Gamma_0,
\]
with
\[
\theta_{\gamma} (x,t) = \pg (t) \cdot (x-\xxg(t))
+ \tfrac{1}{2} (x-\xxg(t))\cdot \Mg(t)(x-\xxg(t)).
\]
A direct computation shows that:
\begin{align*}
\left( \partial_{t}^{2} - \speed(x)^{2} \Delta_x\right) \Phi_{\gamma}(t,x) =
\Phi_{\gamma}(t,x)\sum_{j=0}^2(i \omega_{\gamma})^{2-j} \nu_{\gamma,j}(t,x),
\end{align*}
where
\begin{align}
\nu_{\gamma,0}(t,x) &{} = \left(\partial_{t}\theta_{\gamma}(t,x)\right)^{2}
                        -
\speed(x)^{2}\abs{\nabla_{x}\theta_{\gamma}(t,x)}^{2},\label{eq:eik}
\\
\nu_{\gamma,1}(t,x)&{} =
2\partial_{t}\theta_{\gamma}(t,x)\frac{\partial_{t}\Ag(t)}{\Ag(t)}
							+ \p^{2}_{t}\theta_{\gamma}(t,x)
							-
\speed(x)^{2}\Tr\left(\p^{2}_{x}\theta_{\gamma}(t,x)\right),
\label{eq:trsp}
\\
\nu_{\gamma,2}(t,x) &{}=
\frac{\p_t^2\Ag(t)}{\Ag(t)}.
\end{align}
As a reference, similar results are obtained in \cite[2.117-2.119]{Kat}. By Lemma \ref{lemma_eik},
$\nu_{\gamma,0}(t,x) = \strongop^3([-T,T],\Upsilon)$ and
$\nu_{\gamma,1}(t,x) = \strongop^1([-T,T],\Upsilon)$, while
$\nu_{\gamma,2}(t,x) = \strongop^0([-T,T],\Upsilon)$
by Lemma \ref{lemma_error_flow}.
\end{proof}

\begin{proof}[Proof of Lemma \ref{lemma_tf}]
We drop the $\pm$ superscripts and
calculate:
\begin{align*}
\p_t \Phi_{\gamma}(t,x)
    &= {\p_t} \left[ A_{\gamma}(t)
    \exp(\mi  \omega_\gamma \theta_{\gamma}(t,x)) \right]
\\
    &= \left[ D_{\gamma}(t)
         + \mi \omega_\gamma \p_t\theta_{\gamma}(t,x) \right]
      \Phi_{\gamma}(t,x),
\end{align*}
with
$D_{\gamma}(t) =
      \frac{\p_t A_{\gamma}(t)}{
                        A_{\gamma}(t)}$.
We inspect
\begin{align*}
	 \p_t \theta_{\gamma}(t,x)
   &{} = \p_t p_{\gamma}(t)^T(x-x_{\gamma}(t))
               - p_{\gamma}(t) \cdot \p_t x_{\gamma}(t)
\\
   &\qquad{} + \half(x-x_{\gamma}(t))\cdot \p_t M_{\gamma}(t)(x-x_{\gamma}(t))
                   - \p_t x_{\gamma}(t)M_{\gamma}(t)(x-x_{\gamma}(t)),
\end{align*}
note that $H^\pm(x_\gamma(t),p_\gamma(t))=p_\gamma(t) \cdot \dot{x}_\gamma(t)$,
and use Lemma \ref{lemma_error_flow} to reach the desired conclusion.
\end{proof}

\subsection{Energy estimates}
We recall classical energy estimates for
the wave equation. The fundamental work \cite{MR0350177,Triggiani} treats explicitly only the case
of
bounded domains, but under our assumptions on the velocity the same proofs apply to the whole
space and the half-space. (Alternatively, an argument based on finite speed of propagation permits
the extension to these domains; see also \cite{stolk2000modeling}.)

\begin{theorem}
\label{th_Las}
Let $T \in (0,\infty)$, $f \in L^2([0,T]\times\Rdst)$,
$g_1 \in H^1(\Rdst)$ and $g_2 \in L^2(\Rdst)$. Then
there exists a unique
$v \in C^1\pt{[0,T],H^1(\Rdplus)}\cap C^0\pt{[0,T],L^2(\Rdplus)}$
weak solution to the problem
\begin{equation*}
\left\{
\begin{aligned}
&\partial^2_{t} v (t,x) - \speed(x)^2 \Delta_x v (t,x) = f(t,x), &\qquad t>0, x \in \R^d,
\\
&v(0, x)_{\phantom{t}}= g_1(x), &\qquad x \in {\R^d},
\\
&v_t(0,x) = g_2(x), &\qquad x \in {\R^d}.
\end{aligned}
\right.
\end{equation*}
In addition, $v$ satisfies
\begin{align*}
&\sup_{t\in[0,T]} \norm{v(t,\cdot)}_{H^1(\R^d)}
+
\sup_{t\in[0,T]} \norm{\partial_t v(t,\cdot)}_{L^2(\R^d)}
\\
&\qquad
\leq C_T \left(
\norm{g_1}_{H^1(\R^d)}+\norm{g_2}_{L^2(\R^d)}+\sup_{t\in[0,T]}\norm{f(t,\cdot)}_{L^2(\R^d)}
\right).
\end{align*}
\end{theorem}
\medskip

\begin{theorem}
\label{th_in_llt}
Let $T \in (0,\infty)$, $f \in L^2([0,T]\times\Rdplus)$,
$g_1 \in H^1(\Rdplus)$, $g_2 \in L^2(\Rdplus)$
and $h \in H^1([0,T]\times\R^{d-1})$. Assume that
\[
h(0,y) = g_1(0,y), \qquad y \in \R^{d-1} \qquad \textrm{(Compatibility).}
\]
Then there exists a unique
$v \in C^1\pt{[0,T],H^1(\Rdplus)}\cap C^0\pt{[0,T],L^2(\Rdplus)}$
that is a weak solution to the problem
\begin{equation*}
\left\{
\begin{aligned}
&\partial^2_{t} v (t,x) - \speed(x)^2 \Delta_x v (t,x) = f(t,x), &\qquad t\in[0,T], x \in \Rdplus,
\\
&v(0, x)= g_1(x), &\qquad x \in \Rdplus,
\\
&v_t(0,x) = g_2(x), &\qquad x \in \Rdplus,
\\
&v(t,0,y) = h(t,y), &\qquad t\in[0,T], y \in \Rst^{d-1}.
\end{aligned}
\right.
\end{equation*}
In addition, $v$ satisfies
\begin{align*}
&\sup_{t\in[0,T]} \norm{v(t,\cdot)}_{H^1(\Rdplus)}
+
\sup_{t\in[0,T]} \norm{\partial_t v(t,\cdot)}_{L^2(\Rdplus)}
\\
&\qquad
\leq C_T \left(
\norm{g_1}_{H^1(\Rdplus)}+\norm{g_2}_{L^2(\Rdplus)}+\sup_{t\in[0,T]}\norm{f(t,\cdot)}_{L^2(\Rdplus)}
+\norm{h}_{H^1([0,T]\times\Rst^{d-1})}
\right).
\end{align*}
\end{theorem}

\section{Wave molecules}\label{AppendixB}
Here, we introduce the notion of wave molecule, which is a technical variant of the notion
of wave-atom in \cite{demphd, MR2362408, deyi08}. We also present several basic properties that
parallel those
derived for curvelet molecules in \cite[Appendix A]{dHHSU}.

For simplicity we use the
notation of Section \ref{sec_frame}. While throughout the main part of the article $\Lambda$ denotes
a fixed lattice that provides the
frame expansion granted by Theorem \ref{T_FrProp}, for the results in the Appendices \ref{AppendixB}
and \ref{AppendixC} any lattice $\Lambda$ is adequate and a different choice would yield
equivalent notions and results.

A family of functions $\sett{\phi_\gamma: \gamma \in \Gammam}$ together with a set
$\sett{(a_\gamma, \xi_\gamma): \gamma \in \Gammam}$,
$\Gammam \subseteq \Gamma$, is called
a set of \emph{wave molecules} (WM) if the following conditions hold:
\begin{itemize}
\item[(i)] $|\aast_{\gamma}-\aast_{\gamma'}| \lesssim
|a_{\gamma}-a_{\gamma'}| +1,
\qquad \gamma, \gamma' \in \Gammam$.

\item[(ii)] $d((a_{\gamma},\xi_{\gamma}),(a_{\gamma'},\xi_{\gamma'})) \gtrsim
d((\aast_{\gamma},\xist_{\gamma}),(\aast_{\gamma'},\xist_{\gamma'})),
\qquad \gamma, \gamma' \in \Gammam$.

\item[(iii)] $\abs{\xi_\gamma} \asymp 4^j,
\qquad \gamma=(j,k,\lambda) \in \Gammam$.
\item[(iv)] For all multi-indices
$\alpha$, and $N >0$, there exists a constant $C_{\alpha,N}$ such that
for all $\gamma\in\Gammam$,
\begin{align}
\abs{\partial_\xi^\alpha
\left[ \widehat\phi_\gamma(\xi) e^{-2 \pi i a_\gamma \xi}
\right]}
\leq C_{\alpha,N} \cdot 2^{-jd/2} \cdot 2^{-j\abs{\alpha}} \cdot
\left(1+2^{-j}\abs{\xi-\xi_\gamma}\right)^{-N}, \qquad \xi \in \Rdst,
\end{align}
\end{itemize}
where, $\aast_{\gamma} = 2^{-j}\lambda$, $\xist_{\gamma} = \pointjk$,
as defined in Section~\ref{sec:stset}.

The set $\sett{(a_\gamma, \xi_\gamma):\gamma \in \Gammam}$ is called \emph{the set of
time-frequency nodes} associated with the molecules. Sometimes we refer simply to a set of WM
$\sett{\phi_\gamma: \gamma \in \Gammam}$, understanding implicitly the existence of an adequate set
of time-frequency nodes. We stress that the role of the TF nodes is non-trivial: a set of WM may
cease to satisfy the definitions if the set of TF nodes is replaced by the standard one.

A collection of sets of wave molecules is said to be uniform, if the constants implied in the
definitions
above can be chosen uniformly. All the estimates in the following sections hold uniformly for
uniform families
of wave molecules.

The high-scale part of the frame $\{\varphi_\gamma : \gamma \in \Gammam\}$ is a set of wave
molecules
with the standard choice of nodes $\sett{\pt{\aast_\gamma,\xist_\gamma} : \gamma \in \Gammam}$.
More generally,
we have the
following lemma.

\begin{lemma}
\label{lemma_GB_mol}
Let $\Upsilon=\sett{\SIC_\gamma: \gamma \in \Gamma_0}$ be a well-spread set of GB parameters. Then
the corresponding families of beams $\sett{\Phi_{\gamma}(t,\cdot): \gamma \in \Gamma_0}$,
are families of wave molecules uniformly for $t \in [-T,T]$,
with TF nodes given by $(a_\gamma(t),\xi_\gamma(t)) = (x_\gamma(t),(2\pi)^{-1} \omega_\gamma
p_\gamma(t))$.
\end{lemma}
\begin{proof}
From \eqref{eq:GB_elem} and \eqref{eq:GB_phase} we see that
\begin{align*}
\Phi_\gamma(t,x)= 2^{jd/2} e^{2\pi i \xi(t) \left(x-a_\gamma(t)\right) }
u_\gamma \left(t,
2^{j}\left(x-a_\gamma(t)\right)
\right),
\end{align*}
where
\begin{align*}
u_\gamma(t,x) = 2^{-jd/2} A_\gamma(t) \cdot
\exp \left[-\frac{1}{2}\frac{\omega_\gamma}{4^j}
\left(
\Im M_\gamma(t) x \cdot x \right)
\right]
\cdot
\exp\left[
\frac{\mi}{2} \frac{\omega_\gamma}{4^j} (\Re M_\gamma(t) x \cdot x)
\right].
\end{align*}
Combining Lemmas ~\ref{lemma_unif_flow_1} and \ref{lemma_error_flow}
and Definition \ref{def_order_wsp},
we see that parts (i), (ii) and (iii) of the definition of WM are satisfied.
To verify part (iv), it suffices to show that for all multi-indices
$\alpha$, and $N >0$, there exists a constant $C_{\alpha,N}$ such that
for all $\gamma\in\Gammam$,
\begin{equation}
\abs{\partial_x^\alpha
u_{\gamma}(t,x)}
\leq C_{\alpha,N} \cdot \left(1+\abs{x}\right)^{-N},
\qquad x \in \Rdst,
\end{equation}
since this would imply a similar estimate in the Fourier domain. These conditions follow again from
Lemmas \ref{lemma_unif_flow_1} and \ref{lemma_error_flow} and Definition \ref{def_order_wsp}.
Specifically, we use the facts that
$2^{-jd/2}\abs{A_\gamma(t)} \asymp {\frac{4^j}{\omega_\gamma}} \asymp 1$ and
$\norm{M(t)} \lesssim 1$, with bounds uniform on $[-T,T]$.
\end{proof}
\begin{rem}
As a consequence of Lemma \ref{lemma_GB_mol}, the estimates in the following sections apply to
well-spread Gaussian beams \emph{uniformly} for evolution parameters within a given bounded time
interval.
\end{rem}

\subsection{Operations on wave molecules}
\begin{lemma}
\label{lemma_ops}
Let $\sett{\phi_\gamma: \gamma \in \Gammam}$ be
a set of wave molecules. Then
for multi-indices $\alpha, \beta$ with $\abs{\alpha}=\abs{\beta}=1$,
each of the families
\begin{align*}
\sett{2^{j}(x-a_\gamma)^\beta \phi_\gamma: \gamma \in \Gammam},
\sett{2^{-j}(D-\xi_\gamma)^\alpha \phi_\gamma: \gamma \in \Gammam},
\mbox{ and }
\sett{4^{-j}\partial^\beta_x \phi_\gamma: \gamma \in \Gammam},
\end{align*}
are sets of wave molecules.
\end{lemma}
\begin{proof}
The first two assertions follow easily from the definitions. For the third one
we note that
\begin{align*}
\partial_x^\alpha \phi_\gamma
= (2\pi i) \cdot \left[
(D - \xi_\gamma)^\alpha \phi_\gamma + (\xi_\gamma)^\alpha \phi_\gamma
\right].
\end{align*}
Since $\abs{\xi_\gamma} \asymp 4^j$, the conclusion follows.
\end{proof}

\subsection{The action of pseudodifferential operators}
A symbol $\psym :\Rdst \times \Rdst \to \bC$ belongs to the class $S^{m}_{1,0}$,
$m \in \Rst$, if
\begin{align}
\label{eq_sm}
\abs{\partial_x^\beta \partial_\xi^\alpha \psym(x,\xi)}
\leq C_{\alpha,\beta} (1+\abs{\xi})^{-\abs{\alpha}+m},
\end{align}
for all multi-indices $\alpha,\beta$.

We say that a family of symbols $\sett{\psym_\gamma:\gamma \in \Gammam}$ \emph{belongs to
$S^m_{1,0}$ uniformly on
$\gamma$}
if each $\psym_\gamma$ satisfies \eqref{eq_sm}, with constants independent of $\gamma$.

\begin{lemma}
\label{lemma_pso_wm}
Let $\sett{\phi_\gamma: \gamma \in \Gammam}$ be
a set of wave molecules,
$m \in \Rst$, and let $\sett{\psym_\gamma:\gamma \in \Gammam}$ belong to $S^m_{1,0}$
uniformly on $\gamma$.
Then
\begin{align*}
\sett{4^{-mj} \psym_\gamma(x,D) \phi_\gamma: \gamma \in \Gammam}
\end{align*}
is a set of wave molecules, with the same set of time-frequency nodes.
\end{lemma}
\begin{proof}
Let us write

\begin{align*}
\psi_\gamma &= \psym_\gamma(x,D) \phi_\gamma,\\
\widehat\phi_\gamma(\xi) &= 2^{-jd/2} \widehat\newphi_\gamma(2^{-j}(\xi-\xi_\gamma)) e^{2 \pi i
a_\gamma
\xi},\\
\widehat\psi_\gamma(\xi) &= 2^{-jd/2} \widehat\newpsi_\gamma(2^{-j}(\xi-\xi_\gamma)) e^{2 \pi i
a_\gamma
\xi}.
\end{align*}
With this notation, we know that
\begin{align}
\label{eq_Phi}
\abs{\partial_\xi^\alpha \widehat\newphi_\gamma(\xi)}
\leq C_{\alpha,N} \cdot (1+\abs{\xi})^{-N},
\qquad \mbox{ for all }\alpha, N,
\end{align}
and we want to show that
\begin{align}
\label{eq_Psi}
\abs{\partial_\xi^\alpha \widehat\newpsi_\gamma(\xi)}
\leq C_{\alpha,N} \cdot 4^{jm} \cdot (1+\abs{\xi})^{-N},
\qquad \mbox{ for all }\alpha, N.
\end{align}
To this end, let
\begin{align*}
\psymp_\gamma(x,\xi) = \psym(2^{-j}x+a_\gamma,2^j\xi+\xi_\gamma).
\end{align*}
A simple calculation shows that $\psymp_\gamma(x,D) \newphi_\gamma = \newpsi_\gamma$.
Since $\psym_\gamma \in S^{m}_{1,0}$ uniformly on $\gamma$,
\begin{align*}
\abs{\partial_\xi^\alpha \partial_x^\beta \psymp_\gamma(x,\xi)}
\leq C_{\alpha,\beta} \cdot 2^{-j\abs{\beta}} \cdot 2^{j \abs{\alpha}}
\cdot \left(1+\abs{2^j\xi+\xi_\gamma}\right)^{m-\abs{\alpha}}.
\end{align*}
Let us analyze the expression $L = 2^{j \abs{\alpha}}
\left(1+\abs{2^j\xi+\xi_\gamma}\right)^{m-\abs{\alpha}}$.
Let $A,B>0$ be constants such that $A 4^j \leq \abs{\xi_\gamma} \leq B 4^j$.\\

\noindent \emph{Case I}. If $\abs{\xi} \leq \tfrac{A}{2} 2^j$,
then $\abs{2^j\xi + \xi_\gamma} \asymp 4^j$ and $L \lesssim 2^{j \abs{\alpha}} 4^{j(m-\abs{\alpha})}
\lesssim 4^{jm}$.\\

\noindent \emph{Case II}. If $\abs{\xi} \geq 2^{j+1} B$, then $\abs{2^j\xi + \xi_\gamma} \asymp 2^j
\abs{\xi}$ and
$L \lesssim 2^{j\abs{\alpha}} 2^{j(m-\abs{\alpha})} \abs{\xi}^{m-\abs{\alpha}}
= 2^{jm} \abs{\xi}^m \abs{\xi}^{-\alpha} \lesssim 2^{jm} \abs{\xi}^m$. Distinguishing the cases $m
\geq 0$
and $m < 0$, we see that $L \lesssim 4^{jm} (1+\abs{\xi})^{\abs{m}}$.\\

\noindent \emph{Case III}. If $\tfrac{A}{2} 2^j \leq \abs{\xi} \leq 2^{j+1} B$, then
$L \lesssim 2^{j \abs{\alpha}} \left(1+\abs{2^j\xi+\xi_\gamma}\right)^{m}
\asymp (1+\abs{\xi})^{\abs{\alpha}} \left(1+\abs{2^j\xi+\xi_\gamma}\right)^{m}$.
If $m<0$ we obtain $L \lesssim (1+\abs{\xi})^{\abs{\alpha}}
 \asymp 4^{jm} (1+\abs{\xi})^{\abs{\alpha}-2m}$. If $m>0$,
 we estimate $\abs{2^j\xi+\xi_\gamma} \lesssim 4^j$, giving
 $L \lesssim 4^{jm} (1+\abs{\xi})^{\abs{\alpha}}$. In both cases,
 $L \lesssim 4^{jm} (1+\abs{\xi})^{\abs{\alpha}+2\abs{m}}$.

Considering the three cases, it follows that
\begin{align}
\abs{\partial_\xi^\alpha \partial_x^\beta 4^{-jm} \psymp_\gamma(x,\xi)}
\leq C_{\alpha,\beta} (1+\abs{\xi})^{\abs{\alpha} + 2 \abs{m}}.
\end{align}

A standard argument shows that the operator associated with the symbol
$4^{-jm} \psymp_\gamma$ preserves the ``bump function'' conditions in \eqref{eq_Phi}, see e.g. 
\cite[Chapter 2]{folland1989harmonic}. Hence, \eqref{eq_Psi} follows.
\end{proof}

\subsection{Bessel bounds}
\begin{lemma}
\label{lemma_bes_mol}
Let $\sett{\phi_\gamma: \gamma \in \Gammam}$ be
a set of wave molecules. Then, for $s \in \Rst$,
\begin{align}
\label{eq_bes1}
&\norm{\sum_\gamma c_\gamma \phi_\gamma}^2_{H^s}
\lesssim \sum_\gamma \abs{c_\gamma}^2 4^{2js},
\\
\label{eq_bes2}
&\sum_\gamma 4^{2js} \abs{\ip{f}{\phi_\gamma}}^2
\lesssim \norms{f}^2, \qquad f \in H^s(\Rdst).
\end{align}
\end{lemma}
\begin{proof}
We only prove \eqref{eq_bes1}; then \eqref{eq_bes2} follows by duality.
A standard computation shows that the family $\{\phi_\gamma: \gamma \in \Gammam\}$ satisfies the
following \emph{almost orthogonality estimate}:
\begin{equation*}
\begin{split}
\abs{\ip{\phi_\gamma}{\phi_{\gamma'}}} &\leq C_{N} \cdot 2^{(j+j')d/2} \cdot
\left(1+2^{\jjmin} \abs{a_\gamma-a_{\gamma'}}\right)^{-N}
\\
 &{}\qquad \cdot2^{-\jjmax d} \cdot \left(1+2^{-\jjmax}\abs{\xi_\gamma-\xi_{\gamma'}}\right)^{-N},
\end{split}
\end{equation*}
for all $N>0$. We now show that a similar bound holds with the standard set of time-frequency nodes.
Let
$L := \left(1+2^{\jjmin}\abs{a_\gamma-a_{\gamma'}}\right)^{-N}
\left(1+2^{-\jjmax}\abs{\xi_\gamma-\xi_{\gamma'}}\right)^{-N}$.
First, using condition (i) in the definition of set of WM, we see that
\begin{equation}
\label{eq_aaaa}
L \leq \left(1+2^{\jjmin}\abs{a_\gamma-a_{\gamma'}}\right)^{-N}
\lesssim \left(1+2^{\jjmin}\abs{a^{st}_\gamma-a^{st}_{\gamma'}}\right)^{-N}.
\end{equation}
Second, using condition (ii) and the triangle inequality, we estimate
\begin{align*}
L &\leq \left(1+2^{\jjmin}\abs{a_\gamma-a_{\gamma'}}
+ 2^{-\jjmax}\abs{\xi_\gamma-\xi_{\gamma'}}\right)^{-N}
\\
&\asymp \left(1 + 2^{-\jjmax} \cdot
\left[
d\left((a_{\gamma},\xi_{\gamma}),(a_{\gamma'},\xi_{\gamma'})\right)
\right]^{1/2}
\right)^{-N}
\\
&\asymp \left(1 + 2^{-\jjmax} \cdot
\left[
d\left((\aast_{\gamma},\xist_{\gamma}),(\aast_{\gamma'},\xist_{\gamma'})\right)
\right]^{1/2}
\right)^{-N}
\\
&\lesssim
\left(1+ 2^{-\jjmax}\abs{\xi^{st}_\gamma-\xi^{st}_{\gamma'}}\right)^{-N}.
\end{align*}
Taking the geometric average of this bound and \eqref{eq_aaaa}, we conclude that
\begin{equation*}
\begin{split}
\abs{\ip{\phi_\gamma}{\phi_{\gamma'}}} &\leq C_{N} \cdot 2^{(j+j')d/2} \cdot
\left(1+2^{\jjmin} \abs{a^{st}_\gamma-a^{st}_{\gamma'}}\right)^{-N/2}
\\
 &{}\qquad \cdot2^{-\jjmax d} \cdot
\left(1+2^{-\jjmax}\abs{\xi^{st}_\gamma-\xi^{st}_{\gamma'}}\right)^{-N/2},
\end{split}
\end{equation*}
for all $N>0$. This implies the Schur bound:
\begin{align*}
\sup_{\gamma} \sum_{\gamma'} \abs{\ip{\phi_\gamma}{\phi_{\gamma'}}}, \,
\sup_{\gamma'} \sum_{\gamma} \abs{\ip{\phi_\gamma}{\phi_{\gamma'}}} < +\infty,
\end{align*}
which gives the Bessel bounds for $s=0$.

For $s \in \Zst$, by Lemma \ref{lemma_pso_wm},
$\{4^{-js} (1-\Delta)^{s/2} \phi_\gamma: \gamma \in \Gammam\}$ is a set of wave molecules, so the
conclusion follows from the ``$s=0$'' case. For non-integer $s$, the conclusion follows by
interpolation.
\end{proof}

\section{Almost diagonalization of pseudodifferential operators}
\label{AppendixC}

\subsection{Main result}
In this appendix we prove:

\begin{theorem}
\label{th_diag}
Let $\sett{\phi_\gamma: \gamma \in \Gammam}$ be
a set of wave molecules and let $\psym_\gamma \in S^0_{1,0}$, uniformly for $\gamma \in \Gammam$.
Then
\begin{align}
\psym_\gamma(x,D) \phi_\gamma = \psym_\gamma(a_\gamma, \xi_\gamma) \phi_\gamma
+ 2^{-j} \phi^*_\gamma, \qquad \gamma=(j,k,\lambda) \in \Gammam,
\end{align}
for some set of wave molecules $\sett{\phi^*_\gamma: \gamma \in \Gammam}$, with the same set of
TF nodes
as $\sett{\phi_\gamma: \gamma \in \Gamma_0}$.
\end{theorem}
Later in Section \ref{sec_ppp}, we show how to use Theorem \ref{th_diag} to deduce Theorem
\ref{th_frame_cutoff}.

The proof of Theorem \ref{th_diag} is inspired by \cite[Lemma 3.1]{dsuv09}. Our result,
however, does not require homogeneity of the symbols, and involves a more refined analysis.
We first introduce a preparatory lemma, that is analogous to \cite[Lemma 17]{dHHSU}, this time
in the wave-atom context.

\subsection{Frequency cut-offs}
Let $\sett{\phi_\gamma: \gamma \in \Gammam}$ be
a set of wave molecules with TF nodes $\sett{(a_\gamma,\xi_\gamma): \gamma \in \Gammam}$,
recall that $\abs{\xi_\gamma} \asymp 4^j$ and fix $\varepsilon>0$ such that
\begin{align}
\label{eq_est_a}
B_{\varepsilon 4^j}(\xi_\gamma) \cap B_{\varepsilon 4^j}(0) = \emptyset,
\qquad \mbox{for all }j \geq 1.
\end{align}
Note that this is possible because, if
$B_{\varepsilon 4^j}(\xi_\gamma) \cap B_{\varepsilon 4^{j}}(0) \not= \emptyset$, then
$\abs{\xi_\gamma} \leq 2\varepsilon 4^j$.
Since $\abs{\xi_\gamma} \asymp 4^j$, $\varepsilon$ can be suitably chosen.

Let $\eta \in C^\infty(\Rdst)$ be supported on $B_1(0)$ and
$\eta \equiv 1$ on $B_{1/2}(0)$. Let
\begin{align}
\label{eq_eta}
\eta_\gamma(\xi) =
\eta(\varepsilon^{-1} 4^{-j} (\xi-\xi_\gamma)), \qquad \xi \in \Rdst.
\end{align}

\begin{lemma}
\label{lemma_loc}
Let $\sett{\phi_\gamma: \gamma \in \Gammam}$ be
a set of wave molecules, and let
$\widetilde \phi_\gamma = \eta_\gamma(D) \phi_\gamma$,
where $\eta_\gamma$ is given by \eqref{eq_eta}.
Then, for all $m>0$,
$\sett{2^{jm} (\phi_\gamma-\widetilde \phi_\gamma): \gamma \in \Gammam}$ is a set of wave molecules.
\end{lemma}
\begin{proof}
Let $\alpha,\alpha'$ be multi-indices. Taking into account
the support of $\eta_\gamma$, we estimate for $N,m>0$
\begin{align*}
&\abs{\partial^{\alpha'}_\xi \left[1-\eta_\gamma(\xi)\right] \cdot
\partial^{\alpha}_\xi
\left[\widehat\phi_\gamma(\xi)e^{2\pi i a_\gamma \xi}\right]}
\\
&\qquad \lesssim
2^{-jd/2} \cdot 2^{-j\abs{\alpha}} \cdot
\left(1+2^{-j}\abs{\xi-\xi_\gamma}\right)^{-(N+m)} \cdot
\varepsilon^{-\abs{\alpha'}} \cdot 4^{-j\abs{\alpha'}} \cdot
1_{\abs{\xi-\xi_\gamma} \geq \varepsilon 4^j}
\\
&\qquad \lesssim
2^{-jd/2} \cdot 2^{-j(\abs{\alpha}+\abs{\alpha'})} \cdot
\left(1+2^{-j}\abs{\xi-\xi_\gamma}\right)^{-N} \cdot
\left[\left(1+2^{-j}\abs{\xi-\xi_\gamma}\right)^{-m} \cdot
1_{\abs{\xi-\xi_\gamma} \geq \varepsilon 4^j}\right]
\\
&\qquad \lesssim
2^{-jd/2} \cdot 2^{-j(\abs{\alpha}+\abs{\alpha'})} \cdot 2^{-jm}
\cdot
\left(1+2^{-j}\abs{\xi-\xi_\gamma}\right)^{-N}.
\end{align*}

Since $\widehat{\phi_\gamma(\xi)}-\widehat{\widetilde \phi_\gamma(\xi)}
=\left[1-\eta_\gamma(\xi) \right] \widehat{\phi_\gamma(\xi)}$,
the conclusion follows from Leibniz's rule.
\end{proof}

\subsection{Proof of Theorem \ref{th_diag}}
\begin{proof}
For two families of functions $\sett{f_\gamma: \gamma \in \Gammam}$,
$\sett{g_\gamma: \gamma \in \Gammam}$, we write
\begin{align*}
f_\gamma \sim g_\gamma
\end{align*}
if $\sett{2^j(f_\gamma-g_\gamma): \gamma\in \Gammam}$ is a set of wave molecules.
We want to show that $\psym(x,D) \phi_\gamma \sim \psym(a_\gamma,\xi_\gamma)\phi_\gamma$.

For two families of operators $\sett{T^1_\gamma: \gamma \in \Gammam}$,
$\sett{T^2_\gamma: \gamma \in \Gammam}$, we write
\begin{align*}
T^1_\gamma \sim_m T^2_\gamma
\end{align*}
if $T^1_\gamma - T^2_\gamma = \psymp_\gamma(x,D)$, for some family of symbols
$\{\psymp_\gamma: \gamma \in \Gammam\}$ that belongs to $S^m_{1,0}$ uniformly on $\gamma$.

\step{1}{Frequency localization}. Let $\widetilde \phi_\gamma$ be the functions from Lemma
\ref{lemma_loc}. Then
$\phi_\gamma \sim \widetilde \phi_\gamma$.
Let also $\psym_\gamma(x,\xi) = \psym(x,\xi) \eta_\gamma(\xi)$,
where $\eta_\gamma$ is given by \eqref{eq_eta}. Then, by Lemma \ref{lemma_pso_wm},
\begin{align}
\label{eq_ee}
&\psym(x,D) \phi_\gamma \sim \psym(x,D)\widetilde \phi_\gamma
= \psym_\gamma(x,D) \phi_\gamma,
\\
\label{eq_gg}
&\psym(a_\gamma,\xi_\gamma) \phi_\gamma \sim \psym(a_\gamma,\xi_\gamma)\widetilde \phi_\gamma =
\psym(a_\gamma,\xi_\gamma) \eta(D)
\phi_\gamma.
\end{align}

\step{2}{Linearization}. We linearize the symbol $\psym$ near $(a_\gamma, \xi_\gamma)$ and multiply
that
expansion by
$\eta_\gamma$ to obtain:
\begin{equation}
\label{eq_lin}
\psym_\gamma(x,\xi) = \psym(a_\gamma,\xi_\gamma) \eta_\gamma(\xi)+
\sum_{\abs{\beta}=1} \psymp_\gamma^{1,\beta}(x,\xi) (x-a_\gamma)^\beta
+
\sum_{\abs{\alpha}=1} \psymp_\gamma^{2,\alpha}(x,\xi) (\xi-\xi_\gamma)^\alpha,
\end{equation}
where
\begin{align}
\label{eq_aa1}
\psymp_\gamma^{1,\beta}(x,\xi) &= \eta_\gamma(\xi) \cdot \int_0^1 \partial^\beta_x
\psym(tx + (1-t)a_\gamma, t\xi + (1-t)\xi_\gamma) \,dt,
\\
\label{eq_ba1}
\psymp_\gamma^{2,\alpha}(x,\xi) &= \eta_\gamma(\xi) \cdot \int_0^1 \partial^\alpha_\xi
\psym(tx + (1-t)a_\gamma, t\xi + (1-t)\xi_\gamma) \,dt.
\end{align}
By \eqref{eq_ee} and \eqref{eq_gg}, it suffices to show that the quantization of the linear terms
in \eqref{eq_lin} map $\{\phi_\gamma:\gamma\in \gamma\}$ into $2^{-j}$ multiples of a family of wave
molecules. More
precisely, we want to show that
\begin{align}
\label{eq_b1}
&[\psymp_\gamma^{1,\beta} \cdot (x-a_\gamma)^\beta](x,D) \phi_\gamma \sim 0,
\qquad \mbox{ for } \abs{\beta}=1,
\\
\label{eq_b2}
&[\psymp_\gamma^{2,\alpha} \cdot (\xi-\xi_\gamma)^\alpha](x,D) \phi_\gamma \sim 0,
\qquad \mbox{ for } \abs{\alpha}=1.
\end{align}

\step{3}{Analysis of the linear terms}. In the following claim, the choice of the cut-off functions
in \eqref{eq_eta}
plays a crucial
role.

\noindent\emph{Claim}: $\psymp_\gamma^{1,\beta} \in S^{0}_{1,0}$
and $\psymp_\gamma^{2,\alpha} \in S^{-1}_{1,0}$, uniformly on $\gamma$.
\begin{proof}[Proof of the claim]
Let $\alpha',\beta'$ be multi-indices and let $\xi \in \supp(\eta_\gamma)$.
Then $\xi \in B_{\varepsilon 4^j}(\xi_\gamma)$ and, by \eqref{eq_est_a}
and the fact that $\abs{\xi_\gamma} \asymp 4^j$,
we conclude that $\abs{\xi} \asymp 4^j$. Therefore,
\begin{align}
\label{eq_aaa}
\abs{\partial_\xi^{\alpha'} \eta_\gamma(\xi)}
\leq 4^{-j \abs{\alpha'}} \norm{\partial_\xi^{\alpha'} \eta}_\infty
\lesssim (1+\abs{\xi})^{-\abs{\alpha'}}.
\end{align}
Second, for $t \in [0,1]$, since $\xi,\xi_\gamma \in B_{\varepsilon 4^j}(\xi_\gamma)$,
$\xi_t=t\xi+(1-t)\xi_\gamma \in
B_{\varepsilon 4^j}(\xi_\gamma)$ and
by \eqref{eq_est_a}, $\abs{\xi_t} \asymp 4^j \asymp \abs{\xi}$. Therefore,
\begin{align}
\label{eq_bbb}
\abs{\partial_\xi^{\alpha'} \partial_x^{\beta'}
\psym(tx+(1-t)a_\gamma, \xi_t)}
\lesssim (1+\abs{\xi_t})^{-\abs{\alpha'}}
\asymp (1+\abs{\xi})^{-\abs{\alpha'}}.
\end{align}
We now inspect the expressions in \eqref{eq_aa1} and \eqref{eq_ba1} and use \eqref{eq_aaa} and
\eqref{eq_bbb}, together with Leibniz's rule to conclude that $\psymp_\gamma^{1,\beta} \in
S^{0}_{1,0}$ and $\psymp_\gamma^{2,\alpha} \in S^{-1}_{1,0}$, uniformly on $\gamma$.
\end{proof}

\step{4}{Final estimates}.
Since $\psymp_\gamma^{1,\beta} \in S^{0}_{1,0}$ uniformly on $\gamma$,
for $\abs{\beta}=1$, by Lemma \ref{lemma_pso_wm},
$\{\psymp_\gamma^{1,\beta} \phi_\gamma: \gamma \in \Gamma_0\}$ is a set of WM.
Hence, by Lemma \ref{lemma_ops}
\begin{align*}
[\psymp_\gamma^{1,\beta} \cdot (x-a_\gamma)^\beta](x,D) \phi_\gamma
= (x-a_\gamma)^\beta
\psymp_\gamma^{1,\beta} (x,D) \phi_\gamma \sim 0.
\end{align*}
This gives \eqref{eq_b1}. In addition, for $\abs{\alpha}=1$,
\begin{align*}
[\psymp_\gamma^{2,\alpha} \cdot (\xi-\xi_\gamma)^\alpha](x,D) \phi_\gamma
= \psymp_\gamma^{2,\alpha}(x,D) [(D-\xi_\gamma)^\alpha] \phi_\gamma.
\end{align*}
By Lemma \ref{lemma_ops}, $\{2^{-j}(D-\xi_\gamma)^\alpha \phi_\gamma: \gamma \in \Gammam\}$ is a set
of wave molecules,
while
$\psymp_\gamma^{2,\alpha} \in S^{-1}_{1,0}$ uniformly on $\gamma$. Therefore, \eqref{eq_b2} follows
from Lemma
\ref{lemma_pso_wm}.
This completes the proof.
\end{proof}

\subsection{Proof of Theorem \ref{th_frame_cutoff}}
\label{sec_ppp}

Theorem \ref{th_frame_cutoff} follows immediately
from Theorem \ref{th_diag}, in combination with the bound for the truncation of the coarse scale
in \eqref{eq_trunc_error}, and the Bessel bounds
in Lemma \ref{lemma_bes_mol}.

\section{Sketch of the construction of the frame}
\label{sec_frame_proofs}
This appendix summarizes a proof of Proposition \ref{T_FrProp}.
Since the construction is a variant of Daubechies' criterion for wavelets \cite{da90},
we shall only sketch the main components. See
for example \cite{acm04, kikuwa12, MDH} for variants of Daubechies' criterion in other anisotropic
contexts.

We consider the auxiliary (more concentrated, slimmer) Gaussian function
$\funcsl(x)= 2^{d/2}\func(2x)$, where $\func$ is given by \eqref{eq_norm_Gauss}.
We use the notation
\begin{equation}\label{eq_window_Freq}
\Funcjk(\xi) = \normdjm \Func\pt{2^{-j}\pt{\xi-\pointjk}},
\end{equation}
and define $\funcsljk$ similarly. We also let
$\func_{0,0} := \func$ and $n_0:=1$. It is easy to verify that
for $s \in \Rst$, $j \geq 1$, and $0 \leq k \leq n_j$:
\begin{align}
\label{eq_ineq_plus}
(1+|\xi|)^{s}\abs{\Funcjk(\xi)} &\lesssim (1+|\pointjk|)^{s}\abs{\widehat{\funcsljk}(\xi)},
\qquad \xi \in \Rdst,
\end{align}
where the implied constants depend on $s$ but not on $j$ or $k$.

\subsection{Frequency covering}
The next elementary lemma (whose proof we omit) says that the Gaussian windows adapted to the
frequency cover from
Section \ref{sec_frame} have bounded overlaps.
\begin{lemma}\label{L_boundPhi}
The scale-overlaps control function
\[
\overf(\xi)= \sum_{j\geq 0}\sum_{k=1}^{n_j} 2^{jd}\abs{\Funcjk(\xi)}^2
\]
satisfies $\overf(\xi) \asymp 1$, $\xi \in \Rdst$.
A similar claim holds for the scale-overlaps control function associated with $\funcsl$.
\end{lemma}
\begin{rem}
As the proof of Theorem \ref{T_FrProp} below shows, the ratio between $B:= \sup_\xi \overf(\xi)$ and
$A:= \inf_\xi
\overf(\xi)$ determines the numerical stability of the inversion of the frame operator. In practice
good ratios are
possible; see Figure \ref{fig:control} for a numerical simulation and \cite{acm04, kikuwa12, MDH,
ni14} for more on overlap estimates for frequency covers.
\end{rem}

\begin{figure}[!tbp]
  \centering
  \begin{minipage}[c]{0.5\textwidth}
    \includegraphics[width=\textwidth]{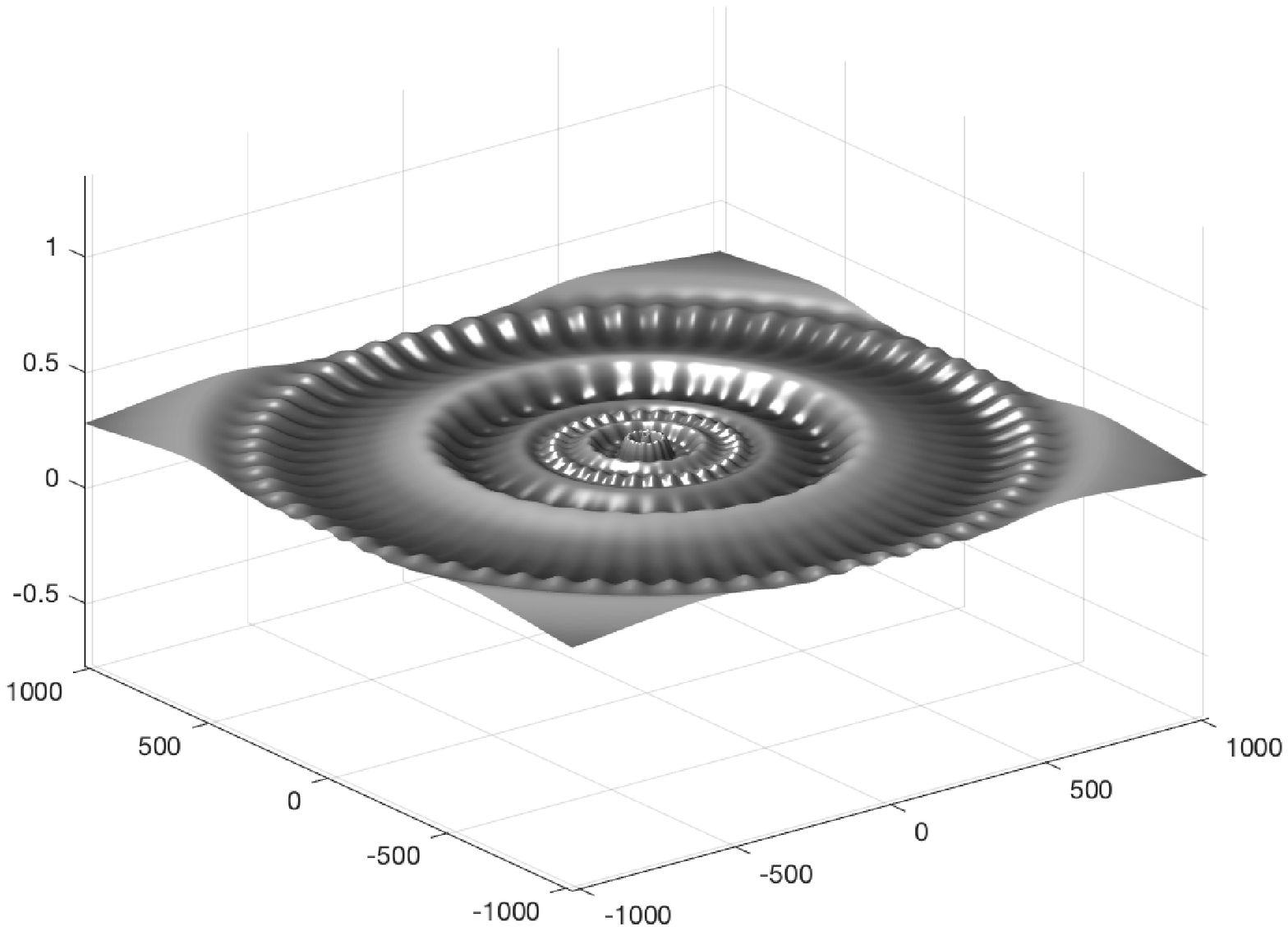}
  \end{minipage}
  \hfill
  \begin{minipage}[c]{0.4\textwidth}
    \includegraphics[width=\textwidth]{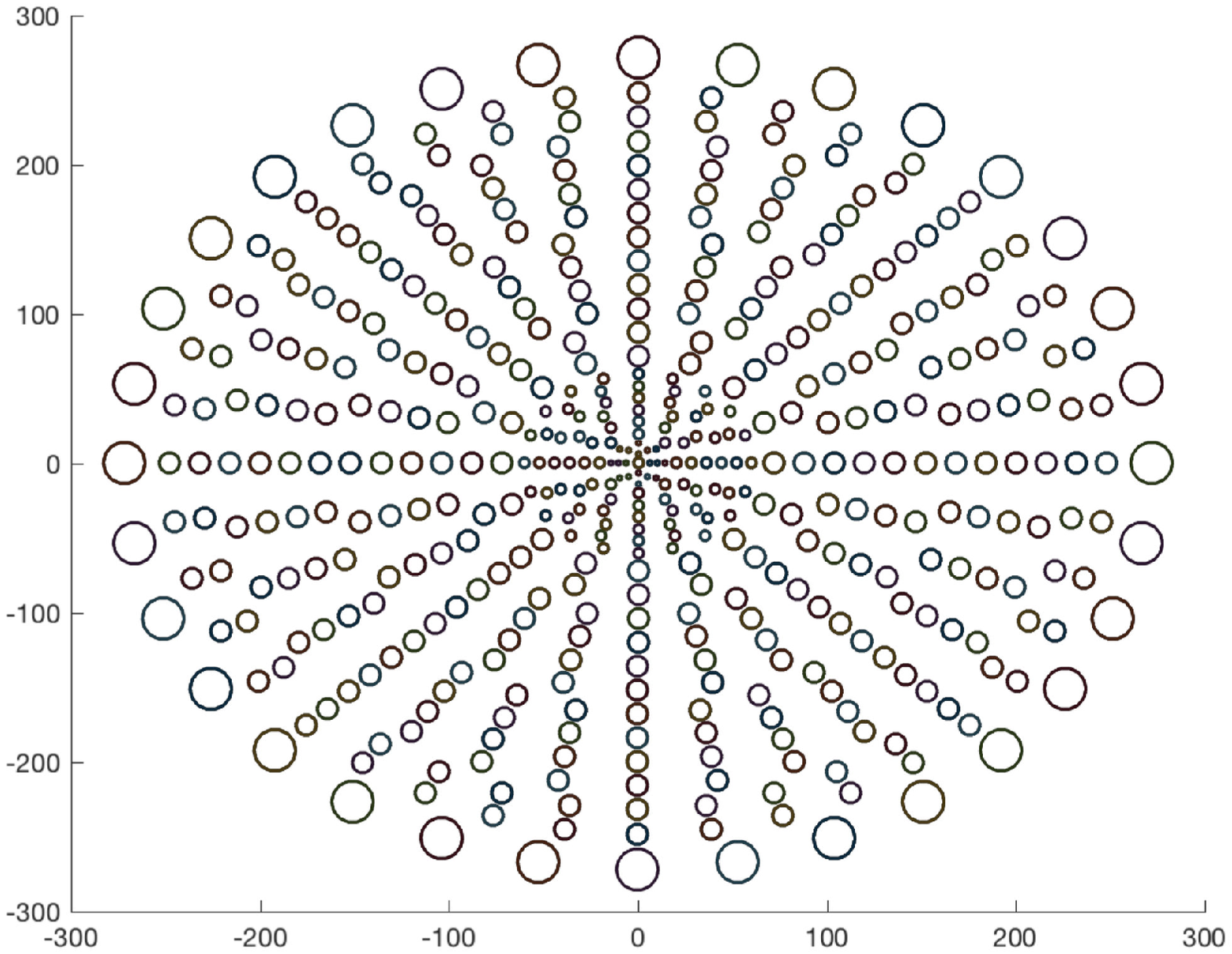}
  \end{minipage}
\caption{A plot of the scale-overlaps control function with $B/A \approx 1.54$,
and the inner balls in the discretization of the corresponding frequency cover.}
\label{fig:control}
\end{figure}

\subsection{Representation of the frame operator}
In the following, we denote by
$\DL$ the dual lattice of $\Lambda$; if $\Lambda = n^{-1} \mathbb{Z}^d$, with $n\in \Nst$, then
$\DL=n \mathbb{Z}^d$. We also use the notation:
$\widehat{T} := \mathcal{F} T \mathcal{F}^{-1}$, for the conjugation of an operator $T$ with the
Fourier transform.
\begin{lemma}[Daubechies-like formula]
\label{lemma_dau}
For $f \in \Hs$:
\begin{equation}
\label{eq_Daubrep}
\frameop f(\xi) = |\Lambda|^{-1}\sum_{\dualpar \in \DL} S_\dualpar f(\xi),
\end{equation}
where
\begin{equation}
\label{eq_Sgamma}
\DFS(\xi) = \sum_{j \geq 0}
\sum_{k=1}^{n_j}
2^{jd}\widehat{f}(\xi - 2^j\dualpar) \overline{\Funcjk(\xi -2^j\dualpar )}\Funcjk(\xi).
\end{equation}
\end{lemma}
\begin{proof}
The representation follows from Poisson's summation formula. See \cite[Lemma 3.1]{MDH} for related
computations.
\end{proof}

Motivated by Lemma \ref{lemma_dau}, we introduced the following quantity. Let
\begin{equation}
\label{eq_Theta}
	\Theta\pt{\zeta}= \ess_{\xi\in\Rdst}
\sum_{j \geq 0} \sum_{k=1}^{n_j}
2^{jd}\abs{\Funcjk^{sl}\pt{\xi-2^j\zeta}}\abs{\Funcjk^{sl}\pt{\xi}},
\qquad \zeta \in \Rdst.
\end{equation}
\begin{lemma}
Let $s \in [0,1]$ and $S_\dualpar$, $\dualpar\in\DL$ be given by \eqref{eq_Sgamma}, then
	\begin{align}
		\norms{S_\dualpar f} & \lesssim
\max\ptg{\Theta\pt{\dualpar},\Theta\pt{-\dualpar}}\norms{f}
\label{eq_core_estimates_1}.
	\end{align}
\end{lemma}
\begin{proof}
Using \eqref{eq_ineq_plus} we estimate
\begin{align}
(1+|\xi|)^{\alpha}\abs{\DFS\pt{\xi}} &\leq
\sum_{j,k}2^{jd}(1+|\xi-2^j\dualpar|)^{\alpha}\abs{\widehat{f}\pt{\xi-2^j\dualpar}}\nonumber
\\&\qquad(1+|\xi|)^{\alpha}\abs{\Funcjk\pt{\xi}}(1+|\xi-2^j\dualpar|)^{-\alpha}\abs{\Funcjk\pt{
\xi-2^j\dualpar}}
\nonumber
\\
&\lesssim
\sum_{j,k}2^{jd}(1+|\xi-2^j\dualpar|)^{\alpha}\abs{\widehat{f}\pt{\xi-2^j\dualpar}}\label{
eq_useLemma2}
\\
&\qquad(1+|\pointjk|)^{\alpha}\abs{\Funcsljk\pt{\xi}}(1+|\pointjk|)^{-\alpha}\abs{\Funcsljk\pt{
\xi-2^j\dualpar}}
\nonumber
\\
&=\sum_{j,k}2^{jd}(1+|\xi-2^j\dualpar|)^{\alpha}\abs{\widehat{f}\pt{\xi-2^j\dualpar}}\abs{
\Funcsljk\pt{\xi}}\abs{
\Funcsljk\pt{\xi-2^j\dualpar}}.
\end{align}
The conclusion now follows from Schur's Lemma.
\end{proof}
\begin{lemma}
For a lattice $\Lambda$ with dual lattice $\DL$
let
\[
\DeL = \sum_{\dualpar \in \DL\backslash\ptg{0}}\max\ptg{\Theta\pt{\dualpar},\Theta\pt{-\dualpar}}.
\]
Then there is constant $c>0$ such that, for
$\Lambda=n^{-1}\mathbb{Z}^d$, $\Delta(\Lambda) \lesssim e^{-c \abs{n}^2}$.
\end{lemma}
\begin{proof}
Using Lemma \ref{L_boundPhi}, we see that
$\abs{\Theta\pt{\dualpar}} \lesssim e^{-c\abs{\dualpar}^2}$, for some constant $c>0$, and the
conclusion follows. See \cite[Lemma 3.5]{MDH} for a related argument.
\end{proof}

\subsection{ Proof of Theorem~\ref{T_FrProp}}
We use the representation in Lemma \ref{lemma_dau}.

\step{1}{Invertibility of the frame operator}. Since $\frameop$ is self-adjoint on $L^2(\Rdst)$, we
only need to
consider $s \in [0,1]$.
Since
\[
\widehat{S}_0 \hat{f}(\xi) = \sum_{j \geq 0} \sum_{k=1}^{n_j} 2^{jd}\abs{\Funcjk(\xi)}^2
\hat{f}(\xi),
\]
it follows from Lemma \ref{L_boundPhi} that $S_0$ is invertible on $H^s$. We let $\Lambda = n^{-1}
\mathbb{Z}^d$ with $n \in \Nst$.
By Lemma \ref{lemma_dau},
\[
\norm{\frameop -  S_0}_{\Hs \rightarrow\Hs} \lesssim |\Lambda|^{-1} \DeL \lesssim n^d e^{-c n^2}.
\]
Therefore, we can choose $n$ such that $\frameop$ is invertible on $H^s$.

\step{2}{Norm equivalence}. For $s \in [-1,1]$, we use the Bessel bounds in
Lemma \ref{lemma_bes_mol}.
These are stated only for the higher scales, but the extension the the zeroth-scale is
straightforward. We conclude that
\begin{align*}
\norms{f}^2 \lesssim \norms{\frameop f}^2
\lesssim \sum_{\gamma \in \Gammafull} 4^{2sj} \abs{\ip{f}{\varphi_\gamma}}^2
\lesssim \norms{f}^2,
\end{align*}
as claimed.

\section{Details on Figures \ref{fig:caustic} and \ref{fig:caustic2}}
\label{sec_num}
We considered the velocity $\speed(x_1,x_2) = 2 - 0.4*\exp{(-(x_1^2 +(x_2-5)^2)/3)}$, $(x_1,x_2)
\in \mathbb{R}^2$, and created a point source by summing $50$ frame elements
with scale $j=4$, centered at the origin, and with frequency directions varying withing a 40-degree
cone around the normal $(0,1)$. For convenience, the frame elements were constructed using as basic
waveform
the dilated Gaussian: $\exp(-2\pi\ln(16)(x_1^2+x_2^2))$ - cf. Figure \ref{fig:packet}.
The GB evolution of the wavefront was computed from $t=0$ to $t=8.40$, following the
construction described in Section~\ref{sec_sets}. The ODEs have been solved with Matlab's ODE45
routine. The
wavefront goes through a caustic at time $t \approx 5.60$. A related example can be found in
\cite{MR2902601}.

For longer times, the numerical solution to the Riccati equation becomes unstable. In the
simulation, a tolerance of $\Im(M(t)) > 0.005$ was set, suppressing the beam if the condition
was not satisfied. In order to improve the precision of the solution, one can reinitialize the
algorithm by re-expanding the solution into wavepackets \cite{MR2558781, MR2684020, Lex1}. We
expect that these techniques will lead to a full numerical implementation of the parametrix.

\section{Table of notation}
\label{sec_table_not}

\begin{table}[htbp]
\begin{center}
     \begin{tabular}{l p{10cm} l}
\toprule
Symbol & Description & Ref.\\
\bottomrule
\vspace{0.05cm}
\\
$\Rdplus$ & $\Rdplus=(0,+\infty)\times\R^{d-1}$. & Sec. \ref{sec:not}\\
$\Rst^d_T$ & $\Rst^d_T= [-T,T] \times \R^{d-1}$. & Sec. \ref{sec:not}\\
$x_*$ & $x=(x_1,x_*) \in \Rdst$, with $x_1 \in \R$ and $x_* \in \R^{d-1}$ & Sec. \ref{sec:not}\\
$\speed=\speed(x)$ & Velocity function. & \\
$h=h(t,x_*)$ & Boundary data for the Dirichlet problem. & Sec.
\ref{sec_ass_source}\\
$[\cconcL, \cconcU]$ & Temporal support of the boundary data. &  Sec. \ref{sec_ass_source}\\
$\cgraz$ & Constant related to the no-grazing ray assumption. &  Sec. \ref{sec_ass_source}\\
$H^\pm=H^\pm(x,p)$ & Signed Hamiltonian functions. & \eqref{eq:hams}\\
$H=H(x,p)$ & Denotes generically either $H^+$ or $H^-$. & \\
$\psym(x,D)$ & Kohn-Nirenberg quantization of a symbol $\psym$ & Sec. \ref{sec:not}\\
$\pointjk$  & Center for the frequency cover. & Sec. \ref{sec_frame} \\
$\pointjkt$  & Approximately normalized version of $\pointjk$. & \eqref{eq_pjkt} \\
$\Lambda$ & A lattice within $\Rdst$. Throughout most of the text, the choice of
$\Lambda$ is fixed by Theorem
\ref{T_FrProp} & \\
$\Gamma$ & Basic scale-angle-position index set. &  \eqref{eq:def_gamma} \\
$\Gammafull$ & Superset of $\Gamma$ augmented with zero-scale. &  \eqref{eq:def_gammafull} \\
$\Gamma_0$ & A generic subset of $\Gamma$. & \\
$\Gammah$ & A subset of $\Gamma$ related to the frame expansion of $h$. & Sec.
\ref{sec_frame_source}\\
$\Gammah^\pm$ & Two subsets $\Gammah^+$ and $\Gammah^-$
that partition $\Gammah$. & \eqref{eq_gammapm}\\
$\gamma$ & Generic element of (a subset of) $\Gammafull$. We refer implicitly to the notation
$\gamma=(j,k,\lambda)$. & Sec. \ref{sec_frame} \\
$\SIC=\SIC_\gamma$ & A function that maps an index $\gamma$ into
a tuple of initial conditions for a GB. & Sec. \ref{sec_sets}\\
$\SIC^{st}=\SIC^{st}_\gamma$ & The standard choice for such a map. & Sec. \ref{sec:stset}\\
$\SIC^{h,\pm}=\SIC^{h,\pm}_\gamma$ & Two particular maps
defined on $\Gammah^\pm$ respectively, constructed in terms of the boundary value $h$.
& Sec. \ref{sec_match} \\
$\Phi^\pm_\gamma$ & GB associated with $\gamma$ by means of an implicit map $\SIC_\gamma$. & Sec.
\ref{sec_sets} \\
$\Phi_\gamma$ & Denotes generically either $\Phi^+_\gamma$
or $\Phi^+_\gamma$. & \\
$\Phi^{st,\pm}_\gamma$ & The beams associated with $\SIC^{st}_\gamma$. & Sec. \ref{sec:stset} \\
$\Phi^{h,\pm}_\gamma$ & The beam associated with $\SIC^{h,\pm}_\gamma$. Here the mode is
\emph{determined} by whether $\gamma \in \Gammah^+$ or $\Gammah^-$. & Sec. \ref{sec_back} \\
$\func$ & Normalized Gaussian function. & \eqref{eq_norm_Gauss} \\
$\func_\gamma$ & Frame element associated with $\gamma \in \Gammafull$. & Sec. \ref{sec_frame}\\
$\phi_\gamma$ & Generic wave molecule. & App. \ref{AppendixB} \\
$\Upsilon$ & Generic set of GB parameters, indexed by a corresponding function $\SIC$.& Rem.
\ref{rem:upsilon}\\
$\Upsilon^{st}$ & The standard choice for such a set. & Sec. \ref{sec:stset}\\
$\Upsilon^{h,\pm}$ & Two particular such sets associated with $h$. & Sec. \ref{sec_back} \\
$F=\strongo^m(I,\Upsilon)$ & A family of functions $F_\gamma$ that vanishes to oder $m$ on the
centers of the beams defined by $\Upsilon$, uniformly on the time interval $I$. & Def.
\ref{def_order_wsp}\\
$F=\strongop^m(I,\Upsilon)$ & Functions with vanishing order at least $m$. & Def.
\ref{def_order_wsp_geq}\\
\smallskip\\
\bottomrule
    \end{tabular}
\end{center}
    \label{tab:notation}
\end{table}
\end{document}